\documentclass[reqno, 12pt]{amsart}

\usepackage[utf8]{inputenc}
\usepackage[T1]{fontenc}
\usepackage{lmodern}
\usepackage{amsmath}
\usepackage{amsthm}
\usepackage{amssymb}
\usepackage{enumerate}
\usepackage{mathrsfs}
\usepackage{mathscinet}
\usepackage{stmaryrd} 
\usepackage{graphicx}
\usepackage{xcolor}
\usepackage[all,cmtip]{xy}
\usepackage{url}
\usepackage{hyperref}
\usepackage[a4paper, hmargin=3cm, vmargin=3cm]{geometry}

\newcommand\mathens[1]{\mathbb{#1}} 

\newcommand{\ud}{\mathrm{d}}

\newcommand{\F}{\mathens{F}}
\newcommand{\N}{\mathens{N}}
\newcommand{\Z}{\mathens{Z}}
\newcommand{\Q}{\mathens{Q}}
\newcommand{\R}{\mathens{R}}
\newcommand{\C}{\mathens{C}}
\newcommand{\D}{\mathens{D}}
\newcommand{\T}{\mathens{T}}
\newcommand{\CP}{\C\mathrm{P}}
\newcommand\sphere[1]{\mathens{S}^{#1}}

\newcommand{\la}{\left\langle}
\newcommand{\ra}{\right\rangle}

\DeclareMathOperator{\length}{length}

\newcommand{\ham}{\mathrm{Ham}}

\newcommand{\id}{\mathrm{id}}

\DeclareMathOperator\im{im}
\DeclareMathOperator\ind{ind}
\DeclareMathOperator\coind{coind}

\newcommand\compi{\circledast}
\newcommand\compii{\diamond}

\newtheorem{thm}{Theorem}[section]
\newtheorem{lem}[thm]{Lemma}
\newtheorem{cor}[thm]{Corollary}
\newtheorem{prop}[thm]{Proposition}

\newtheorem{prop-def}[thm]{Definition-proposition}

\theoremstyle{definition}
\newtheorem{definition}[thm]{Definition}

\theoremstyle{remark}

\newcommand\action{\mathcal{T}}
\newcommand\lochom{\mathrm{C}}
\newcommand\betamax{\beta}
\newcommand\betatot{\beta_{\mathrm{tot}}}

\newcommand\pj{\mathrm{pj}}

\newcommand\bepsilon{{\boldsymbol{\varepsilon}}}
\newcommand\bdelta{{\boldsymbol{\delta}}}
\newcommand\bsigma{{\boldsymbol{\sigma}}}
\newcommand\boldeta{{\boldsymbol{\eta}}}
\newcommand\bB{{\mathcal{B}}}

\usepackage{enumitem}

\newcommand{\fix}{\mathrm{Fix}_c}
\def\mrm#1{{\mathrm{#1}}}

\def\cl#1{{\mathcal{#1}}}

\newcommand{\om}{\omega}
\newcommand{\eps}{\epsilon}
\newcommand{\tmin}{{\text{min},\bK}}
\newcommand{\bK}{{\mathbb{K}}}
\newcommand{\bF}{{\mathbb{F}}}
\DeclareMathOperator{\tot}{\mathrm{tot}}
\DeclareMathOperator{\loc}{\mathrm{loc}}
\DeclareMathOperator{\Spec}{\mathrm{Spec}}
\newcommand{\wh}[1]{\widehat{#1}}
\newcommand{\ol}[1]{\overline{#1}}

\makeatletter\let\@wraptoccontribs\wraptoccontribs\makeatother

\begin{document}

\title
{On the Hofer-Zehnder conjecture on $\CP^d$
via generating functions}

\author[S. Allais]{Simon Allais}
\contrib[with an appendix by]{Egor Shelukhin}
\address{Simon Allais, Université de Paris,
IMJ-PRG,
\newline\indent  8 place Aurélie de Nemours,
75013 Paris, France}
\email{simon.allais@imj-prg.fr}
\urladdr{http://perso.ens-lyon.fr/simon.allais/}
\address{Egor Shelukhin,
    Department of Mathematics and Statistics,
    University of Montreal,
    \newline\indent C.P. 6128 Succ. Centre-Ville Montréal,
QC H3C 3J7, Canada}
\email{shelukhin@dms.umontreal.ca}
\urladdr{https://sites.google.com/site/egorshel/}
\date{October 27, 2020}
\subjclass[2010]{70H12, 37J10, 58E05, 53D40}
\keywords{Generating functions, Hamiltonian, periodic points,
Hofer-Zehnder conjecture, barcodes, persistence modules}

\begin{abstract}
    We use generating function techniques developed by Givental,
    Théret and ourselves to deduce a proof in $\CP^d$ of the homological
    generalization of Franks theorem due to Shelukhin.
    This result proves in particular the Hofer-Zehnder conjecture in the non-degenerate
    case: every Hamiltonian diffeomorphism of $\CP^d$ that has at least
    $d+2$ non-degenerate periodic points has infinitely many periodic points.
    Our proof does not appeal to Floer homology or the theory of
    $J$-holomorphic curves. An appendix written by Shelukhin contains a new
    proof of the Smith-type inequality for barcodes of Hamiltonian
    diffeomorphisms that arise from Floer theory, which lends itself to
    adaptation to the setting of generating functions.
\end{abstract}
\maketitle

\section{Introduction}

Let $\CP^d$ be the complex $d$-dimensional projective space endowed with
the Fubini-Study symplectic structure $\omega$, that is $\pi^*\omega=i^*\Omega$ where
$\pi : \sphere{2d+1}\to\CP^d$ is the quotient map,
$i:\sphere{2d+1}\hookrightarrow\C^{d+1}$ is the inclusion map
and $\Omega := \sum_j \ud q_j \wedge\ud p_j$ is the canonical
symplectic form of $\C^{d+1}\simeq \R^{2(d+1)}$.
We are interested in the study of Hamiltonian diffeomorphisms
of $\CP^d$, which are time-one maps of those vector fields $X_t$
satisfying the Hamilton equations $X_t \lrcorner \omega = \ud h_t$
for some smooth maps $(h_t):[0,1]\times\CP^d\to\R$ called
Hamiltonian maps.
In 1985, Fortune-Weinstein \cite{For85} proved that
any Hamiltonian diffeomorphism of $\CP^d$
has at least $d+1$ fixed points, as was conjectured by Arnol'd.
Given a diffeomorphism $\varphi$, a $k$-periodic point
of $\varphi$ is by definition a fixed point
of the $k$-iterated map $\varphi^k$.
On closed symplectically aspherical manifold (\emph{e.g.} on tori $\T^{2d}$),
every Hamiltonian diffeomorphism has infinitely many periodic points.
This result was conjectured by Conley,
proven for the tori by Hingston \cite{Hin09} and generalized
by Ginzburg \cite{Gin10} after decades of flourishing advances:
Conley-Zehnder proved the non-degenerate case for tori \cite{CZ86},
Salamon-Zehnder proved the non-degenerate case for aspherical manifolds \cite{SZ92},
Franks-Handel and Le Calvez proved the conjecture for surfaces
\cite{FH03,LeC06}.
Contrary to aspherical symplectic manifolds,
the Conley conjecture does not hold in $\CP^d$:
there exists Hamiltonian diffeomorphisms with only finitely many
periodic points.
A simple counter-example is the diffeomorphism
\begin{equation*}
    [z_1:z_2:\cdots:z_{d+1}] \mapsto
    \left[e^{2i\pi a_1}z_1 : e^{2i\pi a_2}z_2:\cdots:
    e^{2i\pi a_{d+1}}z_{d+1}\right],
\end{equation*}
with rationally independent coefficients $a_1,\dotsc,a_{d+1}\in\R$.
This is indeed a Hamiltonian diffeomorphism
whose only periodic points are its fixed points:
the projection of the canonical base of $\C^{d+1}$.
Notice that this Hamiltonian diffeomorphism has the minimal number of
periodic points, that is the minimal number of fixed points conjectured
by Arnol'd.

In the case $d=1$, $\CP^1\simeq\sphere{2}$ and Hamiltonian diffeomorphisms
are the area preserving diffeomorphisms.
Franks \cite{Fra92,Fra96} proved that every area preserving homeomorphism
has either $2$ or infinitely many periodic points.
In 1994, Hofer-Zehnder \cite[p.263]{HZ94} conjectured a higher-dimensional 
generalization of this result:
every Hamiltonian diffeomorphism of $\CP^d$ has either $d+1$
or infinitely many periodic points
(it was stated for more general symplectic manifolds).
In this direction,
a symplectic proof of Franks result
was provided by Collier \emph{et al.} \cite{CKRTZ} in the smooth setting.
Here, we are interested in a version of the Hofer-Zehnder conjecture
proved by Shelukhin in 2019:
if a Hamiltonian diffeomorphism on a closed monotone symplectic manifold
with semisimple quantum homology
has a finite number of contractible periodic points then the sum of
dimension of the local Floer homologies at its contractible fixed points
is equal to the total dimension of the homology of the manifold
\cite{She19}.
As we will see, his proof is based on the theory of barcodes in symplectic topology
introduced by Polterovich-Shelukhin in \cite{PS16}.

In \cite{periodicCPd}, we elaborate on the ideas of Givental \cite{Giv90} and
Théret \cite{The98} to essentially build an analogue of the Floer homology
of Hamiltonian diffeomorphisms of $\CP^d$ with classical Morse theory
and generating functions.
In this article, we complete this construction in order to give an alternative
proof of the theorem of Shelukhin on the Hofer-Zehnder conjecture
``that could have been given in the 90s''.
Shelukhin introduced a homology count over a field $\F$ of the number of fixed points of
a Hamiltonian diffeomorphism with finitely many fixed points $\varphi$ of $\CP^d$.
In our setting, it is defined by
\begin{equation}\label{eq:N}
    N(\varphi;\F) := \sum_{x\in\mathrm{Fix}(\varphi)}
    \dim \lochom_*(\varphi;x;\F) \in \N,
\end{equation}
where the precise definition of the local homology $\lochom_*(\varphi;x;\F)$
of the fixed point $x$ over $\F$ is given in Section~\ref{se:homologysublevel}.
In fact, one can prove that $\lochom_*(\varphi;x;\F)$ is isomorphic
to the local Floer homology of $x$ over the field $\F$,
so that $N(\varphi;\F)$ equals the homological count defined by Shelukhin.
One always has $N(\varphi;\F)\geq d+1$, this is an avatar of
the Fortune-Weinstein theorem.
The theorem of Shelukhin that we prove is the following.
\begin{thm}[{\cite[Theorem~A for $M=\CP^d$]{She19}}]\label{thm:main}
    Every Hamiltonian diffeomorphism $\varphi$ of $\CP^d$
    with finitely many fixed points
    such that $N(\varphi;\F) > d+1$ for some field $\F$ has infinitely many
    periodic points.
    Moreover,
    if $\F$ has characteristic $0$ in the former assumption, there exists
    $A\in\N$ such
    that, for all prime $p\geq A$, $\varphi$ has a $p$-periodic point
    that is not a fixed point;
    if $\F$ has characteristic $p\neq 0$, $\varphi$ has infinitely many
    periodic points with period including in $\{ p^k\ |\ k\in\N\}$.
\end{thm}
In the special case where every fixed point of $\varphi$ is non-degenerate,
one has $\dim\lochom_*(\varphi;x;\F) = 1$ for every fixed point;
hence, $N(\varphi;\F)$ equals the number of fixed points of $\varphi$.
As a special case,
every Hamiltonian diffeomorphism of $\CP^d$ that has at least
$d+2$ non-degenerate periodic points has infinitely many
periodic points and the number grows at least like the sum of prime numbers
(\emph{i.e.} the number of periodic
points of period less than $k$ is $\gtrsim\frac{k^2}{\log k}$).
One can indeed replace $\varphi$ with a power $k$ of itself in order
to get $N(\varphi^k;\F)\geq d+2$ and apply Theorem~\ref{thm:main}.

Our proof follows the same main steps as the original one
and takes advantage of the new proof given by Shelukhin in
Appendix~\ref{se:She} of inequality (\ref{eq:smithbetatot}) below.
Let us give a short outline of it.
Our analogue of the Floer homology of the Hamiltonian diffeomorphism
associated with the Hamiltonian map $(h_s)$ defines with its inclusion morphisms
a persistence module $(G_*^{(-\infty,t)}(h_s;\F))_t$.
Such a persistence module can be represented in a graphical way by a barcode
(see Figure~\ref{fig:barcode}).
Infinite bars of the barcodes always exist in the same cardinality:
they are associated with the spectral values of $(h_s)$.
On the other hand, finite bars exist if and only if $N(\varphi;\F)>d+1$.
The barcodes always have infinitely many bars
but there is a natural free $\Z$-action on their collection, preserving the length of bars
in such a way that, if $\varphi$ has finitely many fixed points,
there are only finitely many $\Z$-orbits of finite bars.
Moreover, if $\varphi$ has finitely many periodic points,
the number of finite bars associated with its iterations
is uniformly bounded.
The proof consists in showing that the existence of a finite bar
implies an unbounded growth of the number of $\Z$-orbits of finite
bars.

Since $\Z$ acts on the collection of finite bars by preserving their length, one can
define the sequence of bar length of
each $\Z$-orbit $0<\beta_1((h_s);\F)\leq \beta_2((h_s);\F)\leq \cdots \leq
\betamax((h_s);\F)$.
We denote by $\betatot((h_s);\F)$ the sum of these lengths.
The proof relies on two results:
the length of a finite bar is bounded above by 1
(with our normalization, see (\ref{eq:action}) below),
the sum of the length satisfies
the Smith-type inequality
\begin{equation}\label{eq:smithbetatot}
    \betatot((h_{ps});\F_p) \geq p \betatot((h_s);\F_p),
\end{equation}
for all prime $p$.
These two steps easily imply Theorem~\ref{thm:main}.
We refer to Section~\ref{se:outline} for a more precise outline of
the proof.

\subsection*{Organization of the paper}
In Section~\ref{se:outline}, we give an outline of the proof.
In Section~\ref{se:preliminary}, we provide the background on persistence
modules and generating functions.
In Section~\ref{se:gfhomology}, we define and study the generating function
homology associated with Hamiltonian maps of $\CP^d$.
In Section~\ref{se:betamax}, we prove that every finite bar
of the barcode associated with a Hamiltonian map of $\CP^d$
has length less than $1$.
In Section~\ref{se:smith}, we show that the sum of the lengths of
the representatives of the finite bars of such a barcode
satisfies a Smith-type inequality.
In Section~\ref{se:proof}, we prove Theorem~\ref{thm:main}.
In Appendix~\ref{se:pj}, we prove fundamental properties of
the projective homology join used throughout the paper.
In Appendix~\ref{se:She}, Shelukhin gives a new proof of
the aforementioned Smith-type inequality on a closed monotone symplectic
manifold in the realm of Floer theory.

\subsection*{Acknowledgments}
I am deeply grateful to Egor Shelukhin
for his helpful advice, his constant support and
sharing with me the proof now written in Appendix~\ref{se:She}.
I thank Vincent Humilière who helps me understand the interest of
the persistence module theory in Hamiltonian dynamics.
I am also thankful to my colleagues at the UMPA for
fruitful related conversations, especially Jean-Claude Sikorav
and Damien Junger who both took some time to study with me
some foundational questions.
I thank my advisor Marco Mazzucchelli for his teaching and support.
All of this would never  have happened if I had not met
Vincent and Egor during the conference \emph{Interactions of Symplectic Topology and Dynamics}
organised by Viktor Ginzburg, Ba\c sak Gürel, Marco and Alfonso Sorrentino in
June 2019 at Cortona
and if I had not participated
in the summer school \emph{Persistent homology and Barcodes}
organised by Peter Albers, Leonid Polterovich and Kai Zehmisch in August 2019 at Marburg
where I had had time to discuss
this project with Egor.

\section{Outline of the proof}\label{se:outline}

Here, we introduce the main tools of the proof.
We postpone the proof of technical statements and the definition
of technical objects to
the remaining sections in order to
give the proof of Theorem~\ref{thm:main} at the end of this section.

Let $(h_s):[0,1]\times\CP^d\to\R$ be a smooth periodic Hamiltonian map
and let $(\varphi_s)$ be the associated Hamiltonian flow on $\CP^d$.
This Hamiltonian map defines a unique Hamiltonian map
$(H_s):[0,1]\times (\C^{d+1}\setminus 0) \to \R$ that is
2-homogeneous, invariant under the diagonal action of $S^1$ on
$\C^{d+1}\setminus 0$
given by
$\lambda\cdot (z_0,\ldots,z_d) := (\lambda z_0,\ldots,\lambda z_d)$
(we will simply write ``$S^1$-invariant'') so that
its restriction on the unit sphere $\sphere{2d+1}\subset\C^{d+1}$
is a lift of $(h_s)$ under the quotient map $\sphere{2d+1}\to\CP^d$.
Let $(\Phi_s)$ be the $\C$-equivariant Hamiltonian flow associated with
$(H_s)$.
We will say that $(H_s)$ and $(\Phi_s)$ are the lifted Hamiltonian map
and Hamiltonian flow of $(h_s)$.
In Section~\ref{se:gf}, we introduce decompositions
of $(\Phi_s)$ into ``small'' Hamiltonian diffeomorphisms (see Section~
\ref{se:gf})
that we usually write $\bsigma = (\sigma_1,\ldots,\sigma_n)$
so that $\sigma_k\circ\cdots\circ\sigma_1 = \Phi_{k/n}$ for
$1\leq k\leq n$.
For such a decomposition $\bsigma$,
we define homology groups $G_*^{(a,b)}(\bsigma)$
for almost all $-\infty\leq a<b\leq +\infty$
in Section~\ref{se:gfhomology}.
In the end of Section~\ref{se:interpolation},
we prove that these homology groups and their natural morphisms
do not depend on the choice of the decomposition $\bsigma$
of $(\Phi_s)$ (up to isomorphism) so that we can write
\begin{equation*}
    G_*^{(a,b)}(h_s) := G_*^{(a,b)}(\bsigma),
\end{equation*}
fixing a decomposition $\bsigma$.
We call these homology groups ``Generating functions homology groups
of $(h_s)$'' or simply ``GF-homology of $(h_s)$''.

Generating functions homology groups of $(h_s)$ satisfy the same key
properties as the Floer homology groups of $(h_s)$
and one can hope that $G_*^{(a,b)}(h_s)$ is isomorphic
to $HF_*^{(\pi a,\pi b)}(h_s)$ with commuting inclusion and boundary morphisms
(the $\pi$ factor is due to our normalisation, see (\ref{eq:action}) below).
These groups are homology groups defined over any chosen ring $R$
(in fact over any group $G$).
Given a fixed point $z\in\CP^d$ of $\varphi_1$
and a capping $u:\D^2\to\CP^d$,
that is a smooth map from the unit 2-disk of $\C$ to $\CP^d$
so that $u(e^{2i\pi s}) = \varphi_s(z)$,
one can define the action $t(\bar{z})\in \R$ of
the capped orbit $\bar{z}:=(z,u)$ by
\begin{equation}\label{eq:action}
    t(\bar{z}) = -\frac{1}{\pi} \left( \int_{\D^2} u^*\omega +
    \int_0^1 h_s \circ \varphi_s(z) \ud s \right).
\end{equation}
Recapping gives a $\Z$-orbit of action values:
$t(A\#\bar{z}) = t(\bar{z}) + k$ where
$A\in\pi_2(\CP^d)$ and
$\pi k = -\la [\omega] , A\ra$.
On Floer homology, the recapping by the generator $A_0\in\pi_2(\CP^d)\simeq\Z$
of symplectic area $\la [\omega],A_0\ra = \pi$
induces the quantum operator 
\begin{equation*}
    q : HF_*^{(a,b)}(h_s) \xrightarrow{\simeq}
    HF_{*-2(d+1)}^{(a-\pi,b-\pi)}(h_s).
\end{equation*}
The analogue isomorphism is defined at (\ref{iso:periodic})
(more precisely, its inverse).

Taking $R$ to be a field $\F$,
these homology groups are $\F$-vector spaces and the
family $(G_*^{(-\infty,t)}(h_s;\F))_t$ together
with its inclusion morphisms define a persistent module
that we call the persistence module associated with $(h_s)$
over the field $\F$.
Assuming that the Hamiltonian diffeomorphism $\varphi_1$
has finitely many fixed points,
the persistence modules of fixed degree $(G_k^{(-\infty,t)}(h_s;\F))_t$,
$k\in\Z$, satisfy suitable finiteness assumptions
and one can define a finite barcode for each of them,
giving a global countable (graded) barcode for the persistence module
associated with $(h_s)$.
Let us describe this barcode (see Figure~\ref{fig:barcode} for an example).
The isomorphism
$G_*^{(-\infty,t)}(h_s) \simeq G_{*+2(d+1)}^{(-\infty,t+1)}(h_s)$
induces a $\Z$-action on the bars of the barcode sending
a bar $(a,b)$ of degree $k$ on a bar $(a+1,b+1)$ of degree $k+2(d+1)$.
Therefore, it is enough to describe a set of representatives of bars
under this action.
In the case where $\varphi_1$ is non-degenerate,
end-points of representative bars are in one-to-one correspondence
with fixed points of $\CP^d$, the value of an end-point being equal
to the action of a capping of the associated fixed point.
In general, a fixed point should be counted with multiplicity
equal to its local homology, which gives $N((h_s);\F)$
end-points.
Among the representative bars, $d+1$ and only $d+1$
are infinite.
This is an avatar of Fortune-Weinstein theorem,
in fact the increasing sequence $(c_k(h_s))_{k\in\Z}$
of values of end-points of the infinite bars of the whole barcode
corresponds to the sequence of spectral invariants of $(h_s)$
(see Theorem~\ref{thm:spectral}).

Theorem~\ref{thm:main} is proved by studying the length of
the finite bars of the barcode of $(h_s)$.
Let us denote by $I_1,\ldots,I_n\subset\R$ representatives of
the $\Z$-orbits of finite bars over the field $\F$.
Up to a permutation, one can assume that $(\length I_k)_k$ is non-decreasing.
Let $\beta_k((h_s);\F)$ be the length of $I_k$,
$\betamax((h_s);\F) := \beta_n((h_s);\F)$ be the length of
the longest bar and
\begin{equation*}
    \betatot((h_s);\F) := \sum_k \beta_k((h_s);\F).
\end{equation*}
The number $\betamax((h_s);\F)$ was first introduced by Usher
and called the boundary depth of $(h_s)$ \cite{Ush11}.
The first important property is that every finite bar has a length
smaller or equal to 1 (see Theorem~\ref{thm:betamax}).
This is the analogue of \cite[Theorem~B]{She19} in the special case $M=\CP^d$
and the
proof follows the same key ideas:
we define a product between GF-homologies and use it
to find an interleaving between the GF-homology of $(h_s)$
and the one of $(h'_s)\equiv 0$ that does not have any finite bar.
The second important property is the Smith inequality (\ref{eq:smithbetatot})
stated in Corollary~\ref{cor:betatotSmith} which is the analogue
of \cite[Theorem~D]{She19} in the special case $M=\CP^d$.
The general strategy follows a new proof of \cite[Theorem~D]{She19}
given by Shelukhin in Appendix~\ref{se:She}.
In the realm of generating functions,
the proof is rather short and very elementary: it essentially relies on the classical
Smith inequality (\ref{eq:smithineq});
it could seem surprising regarding the extraordinary machinery
necessary to prove its Floer theoretical analogue
(although the Floer theoretical proof is available for every closed monotone
symplectic manifold).

\section{Preliminaries}
\label{se:preliminary}

In this section, we fix convention and notation used throughout the article
and recall known results that will be used as well.

\subsection{Persistence modules and barcodes}
\label{se:persistence}

Let us fix a field $\F$.
A persistence module $(V,\pi)$ over the field $\F$ is
an $\R$-family of $\F$-vector spaces $(V^t)_{t\in\R}$
with a collection of morphisms
$\pi_s^t : V^s \to V^t$ for $s\leq t$ such that
$\pi_t^t = \id_{V^t}$ and $\pi_s^t\circ\pi_r^s = \pi_r^t$
whenever $r\leq s\leq t$ that one can call persistence morphisms.
We will always ask our persistence modules $(V,\pi)$ to be
right continuous: for every $t\in\R$, $\pi_t^s$ is an isomorphism
for every $s>t$ close to $t$.
We extend this definition to families $(V^t)_{t\in\R\setminus S}$ 
satisfying the same axioms with the exception that $t\in\R\setminus S$
where $S\subset \R$ is discrete by identifying $(V^t)$ with the persistence
module $\left(\overline{V}^t\right)_{t\in\R}$ defined by taking the direct limit
\begin{equation*}
    \overline{V}^t := \varprojlim_{s>t} V^s, \quad \forall t\in\R,
\end{equation*}
with the obvious extension of the morphisms $\pi_s^t$.
Given two persistence modules $(V,\pi)$ and $(V',\pi')$,
one defines the direct sum of them in the obvious way
$(V\oplus V',\pi\oplus\pi')$ which is a persistence module.
A morphism of persistence module $f:(V,\pi)\to(V',\pi')$
is a family of morphisms $f_t:V^t\to V^{\prime t}$ commuting
with the persistence morphisms.

A persistence module $(V^t)$ is of finite type if
$V^t = 0$ for $t$ sufficiently close to $-\infty$,
every $V^t$ has a finite dimension and there exists
a finite set $S\subset\R$ such that $\pi_s^t$  is an isomorphism
whenever $s$ and $t$ belong to the same connected component of $\R\setminus S$.
The fundamental example is given by $V^t := H_*(\{ f\leq t\};\F)$,
where $f:M\to\R$ is a smooth function on a compact manifold $M$
with finitely many critical points
(by taking a general map on any space, we only have a general persistence module).
Given an interval $I=[a,b)$, with $b\in (a,+\infty]$, we define the persistence
module
$\F(I)$ by $V^t = \F$ for $t\in I$ and $V^t=0$ otherwise,
$\pi_s^t = \id$ when $t$ and $s$ belong to the same connected component of
$\R\setminus \{ a,b\}$
and $\pi_s^t = 0$ otherwise.
It is a persistence module of finite type.
We think of it as representing a class that appears at $t=a$ and
persists until $t=b$ (it persists indefinitely if $b=+\infty$).
Graphically, we represent $\F(I)$ by drawing an horizontal bar from $t=a$
to $t=b$ or without right endpoint if $I=[a,+\infty)$.
The normal form theorem asserts that for every
persistence module $V$ of finite type, there exists a unique finite collection
of couples $(I_k,m_k)$, $I_k\subset \R$ being a bar as above and $m_k\in\N^*$,
so that there is an isomorphism of persistence modules
\begin{equation*}
    V \simeq \bigoplus_k \F(I_k)^{\oplus m_k}
\end{equation*}
(see for instance \cite{Bar94,CZ05}).
The collection $\bB(V) := \{ (I_k,m_k) \}$ is called the barcode of $V$;
it is graphically represented by drawing each horizontal bars of
the $I_k$'s with multiplicity in the same figure (see Figure~\ref{fig:barcode}
for an example with 6 infinite bars and 5 finite bars; ignore the fact
that it is actually part of a larger barcode).

Although we will not need it in its full strength,
we recall the isometry theorem between the bottleneck distance
between barcodes and the interleaving distance between persistence
modules (of finite type).
Given $\delta,\delta'\in\R$ with $\delta+\delta'\geq 0$, a $(\delta,\delta')$-interleaving
between persistence modules $(V,\pi)$ and $(W,\kappa)$ is a couple of
morphisms of persistence modules $f:(V^t)\to (W^{t+\delta})$
and $g:(W^t)\to (V^{t+\delta'})$ such that
$g_{t+\delta}\circ f_t = \pi_t^{t+\delta+\delta'}$
and $f_{t+\delta'}\circ g_t = \kappa_t^{t+\delta+\delta'}$
for all $t\in\R$.
When such an interleaving exists, it is said that $V$ and $W$ are
$(\delta,\delta')$-interleaved.
Given $\delta\geq 0$, a $\delta$-interleaving is by definition
a $(\delta,\delta)$-interleaving.
For instance, if $(V^t)$ and $(W^t)$ are $(\delta,\delta')$-interleaved,
then $(V^t)$ and $(W^{t+\delta_-})$ are $\delta_+$-interleaved, where
$\delta_- = \frac{\delta-\delta'}{2}$ and $\delta_+=\frac{\delta+\delta'}{2}$.
The interleaving distance between two persistence modules $V$ and $W$
is defined by
\begin{equation*}
    d_{\mathrm{int}}(V,W) := \inf\{ \delta\geq 0\ |\
    V \text{ and } W \text{ are $\delta$-interleaved} \}.
\end{equation*}
This is a true distance between persistence modules
up to isomorphisms taking values in $[0,+\infty]$.

Let $\bB := \{ (I_k,m_k) \}$ and $\bB' := \{ (J_l,m'_l) \}$
be two barcodes that we see as multisets of intervals.
Given an interval $I=(a,b]$ or $(a,+\infty)$, we set $I^\delta := (a-\delta,b+\delta]$
or $(a-\delta,+\infty)$ respectively.
Given $\delta\geq 0$, a $\delta$-matching between
the barcodes $\bB$ and $\bB'$
is a bijection of multisets $\mu:\bB_0 \to \bB'_0$
where $\bB_0$ and $\bB'_0$ are some submultisets of
$\bB$ and $\bB'$ containing (at least) every interval of length
$\geq 2\delta$ and such that
$\mu(I)\subset I^\delta$ and $I\subset \mu(I)^\delta$
for every $I\in\bB_0$.
When such a $\delta$-matching exists, it is said that
$\bB$ and $\bB'$ are $\delta$-matched.
The bottleneck distance between two barcodes $\bB$ and
$\bB'$ is defined by
\begin{equation*}
    d_{\mathrm{bottleneck}}(\bB,\bB') :=
    \inf\{ \delta\geq 0\ |\ \bB \text{ and } \bB'
    \text{ are $\delta$-matched} \}.
\end{equation*}
This is a true distance taking values in $[0,+\infty]$.

The isometry theorem between the bottleneck distance and the interleaving distance
states that given any persistence modules of finite type $V$ and $W$,
\begin{equation*}
    d_{\mathrm{int}}(V,W) = d_{\mathrm{bottleneck}}(\bB(V),\bB(W)),
\end{equation*}
(see for instance \cite{BL14}).

\subsection{Generating functions of $\C$-equivariant diffeomorphisms}
\label{se:gf}

In this section, we summarize definitions widely discussed in
\cite[Section~5]{periodicCPd}.
This use of generating functions was much inspired by the dynamical
viewpoint of Chaperon \cite{Cha84} and the work of Givental \cite{Giv90}
and Théret \cite{The98}.

Given a Hamiltonian map $(h_s):[0,1]\times\CP^n\to\R$,
let $(H_s)$ be the lifted Hamiltonian of $\C^{d+1}$ that is
2-homogeneous and $S^1$-invariant as defined in the beginning of
Section~\ref{se:outline} and let $(\Phi_s)$ be the associated Hamiltonian flow.
The maps $\Phi_s$ are smooth diffeomorphisms of $\C^{d+1}\setminus 0$
that extend to homeomorphisms of $\C^{d+1}$ by setting $\Phi_s(0)=0$.
We refer to these maps as $\C$-equivariant Hamiltonian diffeomorphisms because
\begin{equation*}
    \Phi_s(\lambda z) = \lambda\Phi_s(z),\quad
    \forall \lambda\in\C,\forall z\in\C^{d+1}.
\end{equation*}

In general, a generating function
will denote a map $F:\C^{d+1}\times\C^k\to\R$, $(x;\xi)\mapsto F(x;\xi)$,
such that $\partial_\xi F$ is a submersion.
We say that $x\in\C^{d+1}$ is the main variable of $F$ and $\xi\in\C^k$ is the
auxiliary variable.
Let $\Sigma_F\subset \C^{d+1}\times\C^k$ be the submanifold
\begin{equation*}
    \Sigma_F := \left\{ 
    (x;\xi)\in \C^{d+1}\times\C^k \ |\ \partial_\xi F(x;\xi) = 0 \right\}.
\end{equation*}
We say that the generating function $F$ is generating the Hamiltonian
diffeomorphism $\Phi$ of $\C^{d+1}$ if
\begin{equation*}\label{eq:egf}
    \forall z\in\C^{d+1}, \exists ! (x;\xi)\in\Sigma_F,\quad
    x=\frac{z+\Phi(z)}{2} \quad \text{ and } \quad
    \partial_x F(x;\xi)= i(z-\Phi(z)).
\end{equation*}
The critical points of a generating function of $\Phi$ are in
one-to-one correspondence with the fixed points of $\Phi$,
the bijection being $(x;\xi)\mapsto x$.
Given generating functions $F:\C^{d+1}\times\C^k\to\R$
and $G:\C^{d+1}\times\C^l\to\R$ of $\Phi$ and $\Psi$ respectively,
the fiberwise sum of $F$ and $G$ denotes the map
\begin{equation}\label{eq:fiberwise}
    (F+G)(x;\xi,\eta) := F(x;\xi) + G(x;\eta).
\end{equation}
Although this is not a generating function of $\Phi\circ\Psi$,
the critical points of $F+G$ are also in bijection with the fixed
points of $\Phi\circ\Psi$ \emph{via} $(x;\xi,\eta)\mapsto
x-i\partial_xG(x;\eta)/2$
(these statements are easy consequences of the definitions).

Let $\Phi$ be a 
$\C$-equivariant Hamiltonian diffeomorphism of $\C^{d+1}$
which can be decomposed in $\Phi = \sigma_n\circ\cdots\circ\sigma_1$
where every Hamiltonian diffeomorphism $\sigma_k$ is sufficiently $C^1$-close
to $\id$ such that they admit generating functions without auxiliary variable
$f_k:\C^{d+1}\to\R$.
We call such generating functions
elementary generating functions.
They are uniquely defined up to an additive constant,
we will set $f_k(0)=0$.
For all $n\in\N^*$,
we will say that the $n$-tuple $\boldsymbol{\sigma}=(\sigma_1,\dotsc,\sigma_n)$
is associated to the Hamiltonian flow $(\Phi_s)$ if
there exists real numbers $0=t_0\leq t_1\leq \cdots\leq t_n = 1$
such that $\sigma_k = \Phi_{t_k}\circ\Phi_{t_{k-1}}^{-1}$.
For all $k\in\N$, we denote $\bepsilon^k$ the $k$-tuple 
\begin{equation*}
    \bepsilon^k := (\id,\ldots,\id)
\end{equation*}
and we define the $k$-th power $\bsigma^k$ of an $n$-tuple $\bsigma$
as the $kn$-tuple
\begin{equation*}
    \bsigma^k := (\bsigma,\ldots,\bsigma).
\end{equation*}
A continuous family of such tuples $(\bsigma_s)$ will denote
a family of tuples of the same size $n\geq 1$,
$\boldsymbol{\sigma}_s =: (\sigma_{1,s},\dotsc,\sigma_{n,s})$ such that
the maps $s\mapsto \sigma_{k,s}$ are $C^1$-continuous.
Every $\C$-equivariant
Hamiltonian flow $(\Phi_s)_{s\in[0,1]}$ admit a continuous family of associated tuple
$(\boldsymbol{\sigma}_s)$ that is $\boldsymbol{\sigma}_s$ is associated to $\Phi_s$
for all $s\in[0,1]$ (and the size can be taken as large as wanted).
We denote by $F_{\boldsymbol{\sigma}}$ the following function $(\C^{d+1})^n\to\R$:
\begin{equation*}
    F_{\boldsymbol{\sigma}} (v_1,\dotsc,v_n) := \sum_{k=1}^{n}
        f_k\left(\frac{v_k + v_{k+1}}{2}\right) +
    \frac{1}{2}\la v_k,iv_{k+1}\ra,
\end{equation*}
with convention $v_{n+1} = v_1$, each $f_k:\C^{d+1}\to\R$ being the elementary
generating function associated with $\sigma_k$.
If $n$ is odd,
this function is a generating function of $\Phi$ with main variable $v_1$
and auxiliary variable $(v_2,\ldots,v_n)$
that will be referred as
the generating function associated with $\bsigma$.
If the opposite is not explicitly supposed,
$\bsigma$ will have an odd size.

Let $(\Phi_s)$ be the lift of the Hamiltonian flow $(\varphi_s)$ of $\CP^d$
under the Hamiltonian map $(h_s)$.
In this case, we take 2-homogeneous and $S^1$-equivariant $f_k$'s, that is
\begin{equation*}
    f_k(\lambda w) = |\lambda|^2 f_k(w),\quad
    \forall \lambda\in\C,\forall w\in\C^{d+1},
\end{equation*}
so that the resulting $F_\bsigma$ is also 2-homogeneous and $S^1$-invariant.
Hence, a $\C$-line of critical points $\C\mathbf{v}$ of $F_\bsigma$
corresponds to
a $\C$-line of fixed points $\C v_1$ of $\Phi_1$.
The Euler identity for homogeneous maps implies that these $\C$-lines
of critical points have value $0$.
Given a 2-homogeneous and $S^1$-invariant map $F:\C^{N+1}\to\R$,
we denote by $\widehat{F}:\CP^N \to \R$ the projectivization of $F$,
that is the unique map for which the restriction of $F$ to the unit sphere
$\sphere{2N+1}\subset\C^{N+1}$ is a lift under the quotient map
$\sphere{2N+1}\to\CP^N$.
For instance, $h_s = \widehat{H}_s$.
Critical $\C$-lines of $F_\bsigma$ are in one-to-one correspondence with
critical points of $\widehat{F}_\bsigma$ \emph{with value $0$}.
Generically, $\widehat{F}_\bsigma$ has no critical point of value $0$ at all.
Indeed, for all $[z]\in\CP^d$ with lift $z\in\C^{d+1}\setminus 0$,
\begin{equation*}\label{eq:fixphifixPhi}
    \varphi_1([z])=[z] \text{ with action } t\in\R/\Z \quad
    \Leftrightarrow \quad e^{-2i\pi t}\Phi_1(z) = z,
\end{equation*}
where $t$ is defined by (\ref{eq:action}) (it is well-defined up
to an integer if one does not choose a preferred capping);
this is due to Théret \cite[Proposition~5.8]{The98}.
Therefore, for a fixed $m\in\N^*$, we need to define a family $(\bsigma_{m,t})$ of
tuples generating $(e^{-2i\pi t}\Phi_1)$, $-m\leq t\leq m$.

Let $(\delta_t)$ be the family of small $\C$-equivariant diffeomorphisms
$\delta_t(z):=e^{-2i\pi t}z$, $t\in (-1/2,1/2)$.
The associated elementary generating function is
$w\mapsto -\tan(\pi t)\|w\|^2$.
Let us fix once for all an even number $n_0\geq 4$ and
let $(\bdelta_t^{(1)})$ be the family of
$n_0$-tuples $(\delta_{t/n_0},\ldots,\delta_{t/n_0})$
generating $z\mapsto e^{-2i\pi t}z$ for $t\in (-2,2)$.
For all $m\in\N^*$, let $(\bdelta_t^{(m)})$ be a family
of $mn_0$-tuples generating $z\mapsto e^{-2i\pi t}z$ for
$t\in(-m-1,m+1)$ and satisfying 
\begin{equation}\label{eq:bdeltam}
    \bdelta_t^{(m+1)} = \left(\bdelta_t^{(m)},\bepsilon^{n_0}\right),\quad
    \forall t\in [-m,m].
\end{equation}
More precisely, let $\chi:\R\to\R$ be an odd smooth non-decreasing map
such that $\chi_m\equiv\id$ on $[-m-1/4,m+1/4]$ and $\chi_m\equiv m+1/2$ on
$[m+3/4,+\infty)$. We set
\begin{equation*}
    \bdelta_t^{(m+1)} = \left(\bdelta_{\chi_m(t)}^{(m)}
    ,\bdelta^{(1)}_{t-\chi_m(t)}\right),\quad
    \forall t\in (-m-2,m+2).
\end{equation*}

Finally, we can set
\begin{equation*}
    \bsigma_{m,t} := \left(\bsigma,\bdelta_t^{(m)}\right),\quad
    \forall t\in [-m,m].
\end{equation*}
Since $\tan$ is increasing on $[-\pi/2,\pi/2]$,
we deduce that $\partial_t F_{\bsigma_{m,t}} \leq 0$
by a straightforward computation.

\subsection{Homology of sublevel sets and local homology
of a fixed point}
\label{se:homologysublevel}

The study of the homology of sublevel sets of generating functions
was introduced by Viterbo \cite{Vit} who introduced spectral invariants
of Hamiltonian diffeomorphisms of $\R^{2d}$ with compact support.
This work led to the definition of homology groups of these diffeomorphisms
by Traynor \cite{Tray} (which are in fact isomorphic
to their Floer theoretic analogue \cite{Vit96}).
Here, we show how to define similar homology groups for Hamiltonian
diffeomorphisms of $\CP^d$.
We refer to \cite[Section~5]{periodicCPd} for comprehensive
proofs and references of the results stated in this section.

Throughout this paper, $H_*(X)$ and $H^*(X)$ denote respectively
the singular homology and the singular cohomology of a topological
space or pair $X$ over an indeterminate ring $R$.
If one needs to specify the ring $R$, one writes $H_*(X;R)$
and $H^*(X;R)$ instead.
Let $\bsigma$ be an $n$-tuple of small $\C$-equivariant Hamiltonian
diffeomorphisms.
We denote by $Z(\bsigma)\subset \CP^{n(d+1)-1}$ the sublevel set
\begin{equation*}
    Z(\bsigma) := \left\{ \widehat{F}_{\bsigma} \leq 0 \right\}.
\end{equation*}
We denote by $HZ_*(\bsigma)$ the shifted homology group
\begin{equation*}
    HZ_*(\bsigma) := H_{*+(n-1)(d+1)}(Z(\bsigma)),
\end{equation*}
and if $Z(\bsigma')\subset Z(\bsigma)$,
with $\bsigma'$ an $n$-tuple, we set
\begin{equation*}
    HZ_*(\bsigma,\bsigma') := H_{*+(n-1)(d+1)}(Z(\bsigma),Z(\bsigma')).
\end{equation*}

For $m\in\N^*$ and $a\leq b$ in $[-m,m]$,
one has $F_{\bsigma_{m,b}} \leq F_{\bsigma_{m,a}}$ so
$Z(\bsigma_{m,a})\subset Z(\bsigma_{m,b})$ and
we can set
\begin{equation*}
    G_*^{(a,b)}(\bsigma,m) :=
    HZ_*(\bsigma_{m,b},\bsigma_{m,a}),
\end{equation*}
when $a$ and $b$ are not action values of $\bsigma$.
In \cite[Section~5.4]{periodicCPd},
we remarked that this homology group can be naturally identified to
the homology of sublevel sets
of a map, that is
\begin{equation*}
    G_*^{(a,b)}(\bsigma,m) \simeq H_{*+(n-1)(d+1)}
    \left(\left\{\action \leq b\right\},\left\{\action\leq a\right\}\right),
\end{equation*}
for some $C^1$-map $\action:M\to\R$ 
defined on some manifold $M$
that is smooth in the neighborhood
of its critical points.
The function $\action$ is some kind of finite-dimensional action:
critical points of $\action$ are in one-to-one correspondence with
capped fixed points of $\varphi_1$ with action value inside $[-m,m]$.
This correspondence sends critical value to action value and
Morse index up to a $(n-1)(d+1)$ shift in degree to
the Conley-Zehnder index.
More generally, the local homology of $\action$ (up to the same shift
in degree) is isomorphic to the local Floer homology of the corresponding
capped orbit.
Let us denote by $\lochom_*(f;x)$ the local homology
of the critical point $x$ of a map $f$.
According to the above discussion, we can define up to isomorphism
\begin{equation*}
    \lochom_*(\bsigma;z,t) \simeq \lochom_*\left(\widehat{F}_{\bsigma_{m,t}};\zeta\right)
    \simeq \lochom_{*+(n-1)(d+1)}\left(\action;(\zeta,t)\right),
\end{equation*}
where $\zeta\in\CP^{n(d+1)-1}$ is the critical point of
$\widehat{F}_{\bsigma_{m,t}}$ associated with the fixed point $z\in\CP^d$
of action $t\in[-m,m]$ (see \cite[Section~5.5 and 5.7]{periodicCPd}
for details).
The independence on $m$ of this definition can also easily be deduced from
the isomorphism induced by $\theta_m^{m+1}$ (defined in Section~\ref{se:direct})
on the local homologies.
Local homologies $\lochom_*(\bsigma;z,t)$ and $\lochom_*(\bsigma;z,t+1)$ are isomorphic
up to a $2(d+1)$ shift in degree so we will not specify the choice
of representative $t\in\R$ when the grading is irrelevant.
One can prove it without using the isomorphism with local Floer homology
as was done in \cite[Proposition~5.10]{periodicCPd} or by using the
isomorphism induced by (\ref{iso:periodic}) in local homology.
We can now define precisely (\ref{eq:N}) for a choice of tuple $\bsigma$
by
\begin{equation*}
    N(\bsigma;\F) := \sum_{z\in\mathrm{Fix}(\varphi_1)}
    \dim \lochom_*(\bsigma;z;\F) \in \N.
\end{equation*}

We recall that an integer $k\in\N^*$ is said to be admissible for $\varphi$
at a fixed point $z$ if $\lambda^k\neq 1$ for all eigenvalues $\lambda\neq 1$
of $\ud\varphi(z)$.
Until the end of the section, $\varphi$ is associated with a tuple $\bsigma$
and the periodic points of $\varphi$ are isolated in order to simplify
the statements.
The following proposition was proved by Ginzburg-Gürel in \cite{GG10}
for the local Floer homology.
\begin{prop}\label{prop:gromollmeyer}
    Let $k\in\N^*$ be an admissible iteration of $\varphi$ at
    the fixed point $z$. Then as graded $R$-modules,
    \begin{equation*}
        \lochom_*(\bsigma^k;z) \simeq \lochom_{*-i_k}(\bsigma;z),
    \end{equation*}
    for some shift in degree $i_k\in\Z$.
\end{prop}
In our setting, one can prove this result 
by directly applying
the shifting theorem of Gromoll-Meyer \cite[\S3]{GM69b}.
Let us give the proof when $k$ is odd.

\begin{proof}
One can assume that the fixed point $z$ has action $0$
so that $\lochom_*(\bsigma^k;z)$ is isomorphic
to the local homology of $\widehat{F}_{\bsigma^k}$ at
$[\mathbf{v}_0:\cdots:\mathbf{v}_0]$,
where $[\mathbf{v}_0]\in\CP^N$ is the critical point of
$\widehat{F}_\bsigma$ associated with $z$.
Let $M\subset \CP^N$ be a characteristic submanifold for
$\widehat{F}_\bsigma$ at $[\mathbf{v}_0]$.
Then the image of $M$ under the embedding
$\iota:[\mathbf{v}]\mapsto[\mathbf{v}:\cdots:\mathbf{v}]$
is a characteristic submanifold for $\widehat{F}_{\bsigma^k}$
if and only if $k$ is admissible
(see \cite[Equation~(11) and above]{periodicCPd}).
According to the shifting theorem,
the local homology of a function at a given point
is isomorphic to the local homology of the restriction of
this function to a characteristic submanifold at the given point
up to a shift in degree.
Since $\widehat{F}_{\bsigma^k}\circ\iota = \widehat{F}_\bsigma$, the conclusion
follows.
\end{proof}

This proof needs small changes when $k$ is even
(see Section~\ref{se:case2} for an idea of the case $k=2$), but it
will not be needed to prove the following corollary.

\begin{cor}\label{cor:gromollmeyer}
    For every fixed point $z$ of $\varphi$, there exists $B>0$
    such that, for all prime $p$
    \begin{equation*}
        \dim\lochom_*(\bsigma^p;z;\F_p) < B.
    \end{equation*}
\end{cor}

\begin{proof}
    Applying Proposition~\ref{prop:gromollmeyer} with $R:=\Z$,
    the $\lochom_*(\bsigma^p;z;\Z)$'s are isomorphic up to a shift in degree
    for sufficiently large prime numbers $p$.
    Since these $\Z$-modules are finitely generated, the conclusion
    is a straightforward application of the universal coefficient theorem.
\end{proof}

\section{Generating functions homology of Hamiltonian diffeomorphisms of $\CP^d$}
\label{se:gfhomology}

\subsection{Composition of sublevel sets of generating functions}

Given an odd number $n\in\N$, let $A_n$ be the linear automorphism
of $(\C^{d+1})^n$ such that $\mathbf{w} = A_n\mathbf{v}$
with $w_k = \frac{v_k+v_{k+1}}{2}$.
We will often omit $A_n$ in our changes of variables and talking about $w$-variables.
Let $Q_n : (\C^{d+1})^n \to \R$ be the $S^1$-invariant quadratic form
\begin{equation*}
    Q_n(\mathbf{v}) :=
    F_{\bepsilon^n}(\mathbf{v}) =
    \frac{1}{2}\sum_{k=1}^n \la v_k,iv_{k+1} \ra =
    2\sum_{k=1}^n \sum_{l=1}^{k-1} (-1)^{k+l}\la w_k,i w_l \ra,
\end{equation*}
with conventions $v_{n+1} := v_1$ and $w_{n+1} := w_1$.

The following proposition is a direct consequence of the fact that
$Q_n$ is both a quadratic form and a generating function of the identity.
\begin{prop}\label{prop:Qn}
    The quadratic form $Q_n$ has nullity $2(d+1)$.
    Moreover
    \begin{equation*}
        Q_n(v_1,v_2,\ldots,v_n) = - Q_n(v_1,v_n,v_{n-1},\ldots,v_2)
    \end{equation*}
    so that
    \begin{equation*}
        \ind Q_n = \coind Q_n = (n-1)(d+1).
    \end{equation*}
\end{prop}

We denote by $\widetilde{B}_{n,m}:(\C^{d+1})^n \times (\C^{d+1})^m \to (\C^{d+1})^{m+n+1}$
the $\C$-linear map
\begin{equation*}
    \widetilde{B}_{n,m}(\mathbf{w},\mathbf{w}') :=
\left(\mathbf{w},\sum_{k=1}^m (-1)^{k+1} w'_k ,\mathbf{w}'\right),
\end{equation*}
and we denote by $B_{n,m}:\CP^{(d+1)n-1}*\CP^{(d+1)m-1}\to \CP^{(d+1)(n+m+1)-1}$
the associated projective map,
where $*$ denotes the projective join (see Appendix~\ref{se:pj}).
A straightforward computation gives the following proposition.

\begin{prop}
    Given odd integers $n,m\in\N$, for all $\mathbf{w}\in(\C^{d+1})^n$
    and $\mathbf{w}'\in(\C^{d+1})^m$, one has in $w$-variables,
\begin{equation*}
    Q_{n+m+1}\left(\widetilde{B}_{n,m}(\mathbf{w},\mathbf{w}')\right)
    = Q_n(\mathbf{w}) + Q_m(\mathbf{w}').
\end{equation*}
\end{prop}

\begin{cor}\label{cor:Bmn}
    Given tuples $\bsigma$, $\bsigma'$ of odd respective sizes $n$ and $m$,
    one has
    \begin{equation*}
        F_{(\bsigma,\bepsilon,\bsigma')}
        \left(\widetilde{B}_{n,m}(\mathbf{w},\mathbf{w}')\right)
        = F_\bsigma (\mathbf{w}) + 
        F_{\bsigma'} (\mathbf{w}').
    \end{equation*}
    Therefore, $\widetilde{B}_{n,m}$ induces an injective map
    \begin{equation*}
        \left\{ F_\bsigma \leq a \right\} \times \left\{ F_{\bsigma'} \leq b \right\}
        \to \left\{ F_{(\bsigma,\id,\bsigma')} \leq a+b \right\}
    \end{equation*}
    by restriction and $B_{n,m}$
    induces a map
    \begin{equation*}
        Z(\bsigma)*Z(\bsigma') \to Z(\bsigma,\bepsilon,\bsigma').
    \end{equation*}
\end{cor}

\begin{proof}
    This is a direct consequence of the form that takes $F_\bsigma$ in $w$-variables:
    \begin{equation*}
        F_\bsigma(\mathbf{w}) = \sum_{k=1}^n f_k(w_k) + Q_n(\mathbf{w}).
    \end{equation*}
    Therefore,
    \begin{eqnarray*}
        F_{(\bsigma,\bepsilon,\bsigma')}
        \left(\widetilde{B}_{n,m}(\mathbf{w},\mathbf{w}')\right)
        &=& \sum_{k=1}^n f_k(w_k) + \sum_{l=1}^m f'_l(w'_l) +
        Q_{n+m+1}\left(\widetilde{B}_{n,m}(\mathbf{w},\mathbf{w}')\right) \\
        &=& \sum_{k=1}^n f_k(w_k) + Q_n(\mathbf{w}) 
        + \sum_{l=1}^m f'_l(w'_l) + Q_m(\mathbf{w}') \\
        &=& F_\bsigma (\mathbf{w}) + 
        F_{\bsigma'} (\mathbf{w}').
    \end{eqnarray*}
\end{proof}

By making use of the homology projective join defined in Appendix~\ref{se:pj},
we are now in position to define a special map in homology involving two
different Hamiltonian flows and their composition.
Let us fix 2 tuples $\bsigma$, $\bsigma'$ of odd respective sizes $n$ and $n'$.
According to Corollary~\ref{cor:Bmn}, the map $B_{n,n'}$ induces a natural
morphism $H_*(Z(\bsigma)*Z(\bsigma')) \to H_*(Z(\bsigma,\bepsilon,\bsigma'))$.
In Appendix~\ref{se:pj}, we define a morphism
\begin{equation*}
    \pj_*: H_*(A\times B) \to H_{*+2}(A*B)
\end{equation*}
called the homology projective join.
The composition of this morphism with the homology projective join defines
a natural morphism
\begin{equation*}
    HZ_*(\bsigma)\otimes HZ_*(\bsigma') \to HZ_{*-2d}((\bsigma,\bepsilon,\bsigma')).
\end{equation*}
It generalizes to the relative case $Z(\bsigma'')\subset Z(\bsigma')$:
\begin{equation*}
    HZ_*(\bsigma)\otimes HZ_*(\bsigma',\bsigma'') \to
    HZ_{*-2d}((\bsigma,\bepsilon,\bsigma'),(\bsigma,\bepsilon,\bsigma''));
\end{equation*}
the relative $HZ_*$ could be in the left hand side of the tensor product as
well, as long as one of the two $HZ_*$'s is an absolute homology group.
In symbols, we will write this map as $\alpha\otimes\beta \mapsto
\alpha \compi \beta$.
The naturality of these morphisms under inclusion maps and boundary maps
follows directly from the naturality of the homology projective join.

In particular, the following diagram commutes
\begin{equation}\label{dia:compipj}
    \begin{gathered}
        \xymatrix{
            HZ_*(\bsigma)\otimes HZ_*(\bsigma') \ar[r]^-{\compi} \ar[d]
            & HZ_{*-2d}((\bsigma,\bepsilon,\bsigma')) \ar[d]\\
            H_{*+r}(\CP^N)\otimes H_{*+r'}(\CP^{N'}) \ar[r]^-{\pj_*}
            & H_{*+r+r'+2}(\CP^{N''})
        }
    \end{gathered},
\end{equation}
where the vertical arrows are induced by inclusions,
$r:=(n-1)(d+1)$, $N:=n(d+1)-1$, $r':= (n'-1)(d+1)$,
$N':=n'(d+1)-1$, $N'':=(n+n'+1)(d+1)-1$ and
we see $\CP^N$ and $\CP^{N'}$ as the disjoint subspaces included in $\CP^{N''}$
\emph{via} $[\mathbf{w}]\mapsto [\mathbf{w}:0]$ and
$[\mathbf{w}']\mapsto [0:\mathbf{w}']$.
The commutativity of this diagram follows from the naturality of $\pj_*$
and the fact that $B_{n,n'}$ is homotopic to
$[\mathbf{w}:\mathbf{w}']\mapsto [\mathbf{w}:0:\mathbf{w}']$.

\begin{prop}\label{prop:associativityComposition}
    Composition morphisms are associative,
    that is the following diagram commutes:
    \begin{equation*}
        \begin{gathered}
            \xymatrix{
                HZ_*(\bsigma)\otimes HZ_*(\bsigma')\otimes HZ_*(\bsigma'')
                \ar[d] \ar[r]
                & HZ_*(\bsigma)\otimes HZ_{*-2d}((\bsigma',\bepsilon,\bsigma''))
                \ar[d] \\
                HZ_{*-2d}((\bsigma,\bepsilon,\bsigma'))\otimes HZ_*(\bsigma'')
                \ar[r]
                & HZ_{*-4d}((\bsigma,\bepsilon,\bsigma',\bepsilon,\bsigma''))
            }
        \end{gathered}.
    \end{equation*}
    In symbols, given $\alpha\in HZ_*(\bsigma)$,
    $\beta\in HZ_*(\bsigma')$ and $\gamma\in HZ_*(\bsigma'')$,
    \begin{equation*}
        (\alpha\compi\beta)\compi\gamma = \alpha\compi(\beta\compi\gamma).
    \end{equation*}
    This is also true for the relative case where
    one of the initial groups (\emph{e.g.} $HZ_*(\bsigma'')$) is
    replaced by a relative homology group (\emph{e.g.} $HZ_*(\bsigma'',\bsigma^{(3)})$)
    while the other groups
    are still absolute homology groups.
\end{prop}

\begin{proof}
    We first remark that
    \begin{multline*}
        \widetilde{B}_{n+n'+1,n''}\left(\widetilde{B}_{n,n'}
        (\mathbf{w},\mathbf{w}'),\mathbf{w}''\right)
        =
        \left(\mathbf{w},\sum_{k=1}^{n'} (-1)^{k+1}w'_k ,
            \mathbf{w}',\sum_{l=1}^{n''} (-1)^{l+1}w''_l,
        \mathbf{w}''\right) \\
        = 
        \widetilde{B}_{n,n'+n''+1}\left(
        \mathbf{w},\widetilde{B}_{n',n''}(\mathbf{w}',\mathbf{w}'')\right),
        \quad \forall \mathbf{w},\mathbf{w}',\mathbf{w}'',
    \end{multline*}
    so that $B_{n+n'+1,n''}\circ (B_{n,n'}*\id) = B_{n,n'+n''+1}\circ (\id * B_{n',n''})$.
    Here, we denote by $f*g$ the map
    $(f*g)[a:b]:=[f(a):g(b)]$.
    Therefore,
    for all $\alpha\in HZ_*(\bsigma)$, $\beta\in HZ_*(\bsigma')$ and
    $\gamma\in HZ_*(\bsigma'')$ (or $\gamma\in HZ_*(\bsigma'',\bsigma^{(3)})$),
    \begin{multline*}
        (B_{n+n'+1,n''})_*\pj_*((B_{n,n'})_*(\pj_*(\alpha\times\beta)\times\gamma))\\
        = 
        (B_{n+n'+1,n''})_*(B_{n,n'}*\id)_*\pj_*(\pj_*(\alpha\times\beta)\times\gamma)\\
        =
        (B_{n,n'+n''+1})_*(\id * B_{n',n''})_*\pj_*(\pj_*(\alpha\times\beta)\times\gamma)\\
        =
        (B_{n,n'+n''+1})_*(\id * B_{n',n''})_*\pj_*(\alpha\times\pj_*(\beta\times\gamma))\\
        =
        (B_{n,n'+n''+1})_*\pj_*(\alpha\times (B_{n',n''})_*\pj_*(\beta\times\gamma)),
    \end{multline*}
    where we use the naturality of the homology projective join (\ref{dia:natural})
    to get the first and last identity, the previous remark to get the second equality
    and the associativity of the homology projective join
    (Proposition~\ref{prop:pjassociativity}) to get the third equality.
\end{proof}

\subsection{The direct system of $G^{(a,b)}_*(\bsigma)$}
\label{se:direct}

For a fixed $m\in\N$, the long exact sequence of the triple induces
inclusion and boundary morphisms fitting into a long exact sequence:
\begin{equation*}
    \cdots \xrightarrow{\partial_{*+1}}
    G_*^{(a,b)}(\bsigma,m) \to G_*^{(a,c)}(\bsigma,m) \to G_*^{(b,c)}(\bsigma,m)
    \xrightarrow{\partial_*} G_{*-1}^{(a,b)}(\bsigma,m) \to \cdots
\end{equation*}
where $-m\leq a\leq b\leq c\leq m$.
In order to precisely define these maps without reference of $m$ anymore,
we will define an isomorphism
\begin{equation}\label{iso:Gm}
    \theta_{m}^{m+1}:G_*^{(a,b)}(\bsigma,m) \to G_*^{(a,b)}(\bsigma,m+1),
\end{equation}
for $-m\leq a\leq b\leq m$, that commutes with the above mentioned inclusion 
and boundary morphisms
and define $G^{(a,b)}_*(\bsigma)$ as the direct limit of
the direct system induced by $(\theta_m^{m+1})_m$:
\begin{equation*}
    G_*^{(a,b)}(\bsigma) := \varinjlim G_*^{(a,b)}(\bsigma,m).
\end{equation*}
We will then have inclusion and boundary morphisms
\begin{equation*}
    \cdots \xrightarrow{\partial_{*+1}} 
    G_*^{(a,b)}(\bsigma) \to G_*^{(a,c)}(\bsigma) \to G_*^{(b,c)}(\bsigma)
    \xrightarrow{\partial_*} G_{*-1}^{(a,b)}(\bsigma) \to \cdots
\end{equation*}
for all $a\leq b\leq c$ that are not action values; one can thus set
\begin{equation*}
    G_*^{(-\infty,b)}(\bsigma) := \varprojlim G_*^{(a,b)}(\bsigma),
    \quad a\to -\infty,
\end{equation*}
and one can then define $G_*^{(-\infty,+\infty)}(\bsigma)$ by taking a direct limit in
a similar way. The inclusion and boundary maps thus extend to the extended real
numbers line
$\overline{\R} := \R\cup\{-\infty,+\infty\}$.

In order to define the isomorphism (\ref{iso:Gm}),
let us remark that for an odd $n\in\N$,
the space $Z(\bepsilon^n)$ retracts on the projectivization of the
maximal non-positive linear subspace
of $Q_n$ which is a $\CP^{N-1}$ with $N=(d+1)(n+1)/2$ according
to Proposition~\ref{prop:Qn}.
Therefore,
\begin{equation*}
    HZ_*(\bepsilon^n) =
    \bigoplus_{k=-(d+1)(n-1)/2}^d R a_k^{(n)}
    \simeq H_{*+(n-1)(d+1)}\left(\CP^{(d+1)(n+1)/2-1}\right),
\end{equation*}
where $a_k^{(n)}$ is the generator of degree $2k$ identified with
the class $[\CP^l]$ of appropriate degree $2l=2k+(n-1)(d+1)$ under the isomorphism
induced by the inclusion of a maximal complex projective subspace
of $Z(\bepsilon^n)$.
With the help of the composition defined in the previous section,
we now define (\ref{iso:Gm}) by
\begin{equation*}
    \theta_m^{m+1}(\alpha) := \alpha\compi a_d^{(n_0-1)}
    \in G_*^{(a,b)}(\bsigma,m+1),
    \quad
    \forall \alpha\in G_*^{(a,b)}(\bsigma,m).
\end{equation*}
This is formally well-defined since
\begin{equation*}
    HZ_*((\bsigma_{m,b},\bepsilon^{n_0}),(\bsigma_{m,a},\bepsilon^{n_0}))
    = G_*^{(a,b)}(\bsigma,m+1),
\end{equation*}
according to (\ref{eq:bdeltam}).

\begin{prop}\label{prop:compad}
    For an odd $n\in\N$,
    the morphism
    \begin{equation*}
        \left\{
        \begin{array}{c c c}
            HZ_*(\bsigma_{m,b},\bsigma_{m,a}) &\to &
            HZ_*((\bsigma_{m,b},\bepsilon^{n+1}),(\bsigma_{m,a},\bepsilon^{n+1}))\\
            \alpha &\mapsto& \alpha\compi a_d^{(n)}
        \end{array}
        \right.
    \end{equation*}
    is an isomorphism
    (and the same is true for $\alpha\mapsto a_d^{(n)}\compi \alpha$).
\end{prop}

\begin{cor}
    The morphism $\theta_m^{m+1}$ is an isomorphism.
\end{cor}

\begin{proof}
    Let $A:= Z(\bsigma_{m,b},\bsigma_{m,a})$,
    $A' := Z((\bsigma_{m,b},\bepsilon^{n+1}),(\bsigma_{m,a},\bepsilon^{n+1}))$,
    $B:= Z(\bepsilon^n)$,
    $n'$ be the size of $\bsigma_{m,b}$
    and let us denote by $\theta$ the morphism in question.
    Up to a shift in degree,
    the morphism $\theta$ can be written explicitly as
    the composition
    \begin{equation*}
        H_*(A) \xrightarrow{\pj_*(\cdot\times a_d^{(n)})}
        H_{*+i_0}(A*B) \xrightarrow{(B_{n',n})_*}
        H_{*+i_0}(A'),
    \end{equation*}
    for some $i_0\in\N$.
    According to Corollary~\ref{cor:pjstabilization},
    the first morphism is an isomorphism.
    It remains to show that $(B_{n',n})_*$ is
    also an isomorphism.

    Let
    $C$ be the following
    automorphism of $(\C^{d+1})^{n'+n+1}$:
    \begin{equation*}
        C(\mathbf{v},\mathbf{v}') := (\mathbf{v},\mathbf{v}'+(v_1,v_{n'},v_1,
        v_{n'},\ldots,v_1,v_{n'})),
    \end{equation*}
    where $\mathbf{v}\in(\C^{d+1})^{n'}$ and
    $\mathbf{v}'\in(\C^{d+1})^{n+1}$.
    By a direct computation
    \begin{equation*}
        F_{(\bsigma_{m,t},\bepsilon^{n+1})}\circ C(\mathbf{v},\mathbf{v}') = 
        F_{\bsigma_{m,t}}(\mathbf{v}) + Q_{n+2}(0,\mathbf{v}'),
    \end{equation*}
    (this is an explicit version of \cite[Proposition~5.2]{periodicCPd}).
    According to Givental \cite[Proposition~B.1]{Giv90},
    $\{ \widehat{F}_{(\bsigma_{m,t},\bepsilon^{n+1})}\circ C \leq 0 \}$ retracts on
    $\{ \widehat{F}_{\bsigma_{m,t}} \leq 0 \} * \{ \widehat{Q}_{n+2}(0,\cdot) \leq 0\}$.
    The quadratic form $Q_{n+2}(0,\cdot)$ is non-degenerate with index $(n+1)(d+1)$,
    so that the sublevel set
    $\{ \widehat{Q}_{n+2}(0,\cdot)\leq 0\}$ retracts on a complex projective
    space of $\C$-dimension $(d+1)(n+1)/2 -1$.

    The injective linear map $J:(\C^{d+1})^n \to (\C^{d+1})^{n+1}$,
    $J\mathbf{v} := (\mathbf{v},v_1)$ satisfies
    \begin{equation*}
        Q_{n+2}(0,J\mathbf{v}) = Q_n(\mathbf{v}).
    \end{equation*}
    Since the sum of the index and the nullity of $Q_n$ is $(n+1)(d+1)$,
    $\{ \widehat{Q}_n\leq 0\}$ retracts on a projective space of $\C$-dimension
    $(d+1)(n+1)/2 -1$ in such a way that $J$ induces an isomorphism in homology
    \begin{equation*}
        J_*:
        H_*\left(\left\{ \widehat{Q}_n\leq 0\right\}\right) \xrightarrow{\simeq}
        H_*\left(\left\{ \widehat{Q}_{n+2}(0,\cdot)\leq 0\right\} \right).
    \end{equation*}

    Let $P:(\C^{d+1})^{n+1}\to (\C^{d+1})^n$ be the surjective linear map
    \begin{equation*}
        P(v_1,\ldots,v_{n+1}):=(v_{n+1},v_2,v_3,\ldots,v_n)
    \end{equation*}
    so that $PJ=\id$.
    Let $P':(\C^{d+1})^{n+n'+1} \to (\C^{d+1})^{n+n'}$
    be $P'(\mathbf{v},\mathbf{v}')=(\mathbf{v},P\mathbf{v}')$
    and let $J':(\C^{d+1})^{n+n'} \to (\C^{d+1})^{n+n'+1}$
    be $J'(\mathbf{v},\mathbf{v}')=(\mathbf{v},J\mathbf{v}')$ so
    that $P'J'=\id$.
    In $v$-variables, $\widetilde{B}_{n',n}$ takes the form
    \begin{equation*}
        (\mathbf{v},\mathbf{v}') \mapsto (\mathbf{v},v_1,\mathbf{v}'
        + (v_1'-v_1,v_1-v_1',v_1'-v_1,\ldots,v_1-v_1')).
    \end{equation*}
    A direct computation then shows that the endomorphism
    $\tilde{f}:=P'C^{-1}\widetilde{B}_{n',n}$ is invertible.
    More precisely, $\tilde{f}(\mathbf{v},\mathbf{v}')=(\mathbf{v},\tilde{g}(\mathbf{v}')
    +\tilde{h}(\mathbf{v}))$ where $\tilde{g}$ and $\tilde{h}$ are $\C$-linear and
    $\tilde{g}$ is invertible.
    Let $f:A*B\to A*B$ be the induced projective map.
    
    We then have the following commutative diagram:
    \begin{equation*}
        \xymatrixcolsep{4pc}
        \begin{gathered}
        \xymatrix{
            H_*(A*B) \ar[r]^-{(B_{n',n})_*} \ar[d]^-{f_*}_-{\simeq} &
            H_*(A') \\
            H_*(A*B) \ar@/^/[r]^-{J'_*}_-{\simeq} &
            H_*\left(A*\left\{ \widehat{Q}_{n+2}(0,\cdot) \leq 0\right\}
            \right) \ar@/^/[l]^-{P'_*}
            \ar[u]^-{C_*}_-{\simeq}
        }
    \end{gathered},
    \end{equation*}
    By the above discussion, the induced maps $f_*$, $C_*$,
    $J'_*$ and $P'_*$ are isomorphisms.
    Therefore, $(B_{n',n})_*$ is an isomorphism
    and so is $\theta$.
\end{proof}

By construction of the map $\theta_m^{m+1}$, the following diagram
commutes.
\begin{equation}\label{dia:thetaCPN}
    \begin{gathered}
        \xymatrixcolsep{7pc}
        \xymatrix{
            G_*^{(a,b)}(\bsigma,m) \ar[r]^-{\theta_m^{m+1}}_-{\simeq}
            & G_*^{(a,b)}(\bsigma,m+1) \\
        HZ_*(\bsigma_{m,b})
        \ar[r]^-{\cdot \compi a_d^{(n_0-1)}} \ar[u] \ar[d]
            & HZ_*(\bsigma_{m+1,b})
            \ar[u] \ar[dd] \\
            H_{*+r}(\CP^N) \ar[d] \\
            H_{*+r}(\CP^{N+n_0(d+1)})
            \ar[r]^-{\pj_*(\cdot\times[\CP^{(d+1)n_0/2}])}
            & H_{*+r+n_0(d+1)}(\CP^{N+n_0(d+1)})
        }
    \end{gathered},
\end{equation}
where the vertical arrows are induced by inclusions,
$r:=(n_1+mn_0-1)(d+1)$, $n_1$ is the size of $\bsigma$,
$N:=(n_1+mn_0)(d+1)-1$.

Applying the commutativity of (\ref{dia:compipj}) to
$\bsigma=\bepsilon^n$ and $\bsigma'=\bepsilon^{n'}$,
one has
\begin{equation*}
a^{(n)}_k \compi a^{(n')}_l = a^{(n+n'+1)}_{k+l-d},
\end{equation*}
for $-(d+1)(n-1)/2 \leq k\leq d$ and $-(d+1)(n'-1)/2\leq l\leq d$.
By associativity of $\compi$ (Proposition~\ref{prop:associativityComposition}),
we deduce that the isomorphism
$\theta_m^{m+k} := \theta^{m+k}_{m+k-1}\circ\cdots\circ\theta_m^{m+1}$
is $\theta_m^{m+k}(\alpha) = \alpha\compi a_d^{(kn_0-1)}$.

By using the same construction as for $\theta_m^{m+k}$,
one can define an isomorphism
\begin{equation*}
    \eta_k : G_*^{(a,b)}(\bsigma,m)
    \xrightarrow{\simeq} G_*^{(a,b)}((\bepsilon^k,\bsigma),m),
\end{equation*}
sending $\alpha$ to $a_d^{(k-1)}\compi \alpha$
for each even $k\in\N$.
This isomorphism commutes with the direct system induced by the $\theta_m^{m+1}$'s
and the inclusion maps
and makes a diagram similar to (\ref{dia:thetaCPN}) commute where
in particular
$\bsigma_{m+1,b}$ is replaced by $(\bepsilon^k,\bsigma)_{m,b}$.
The commutation with the direct system induces a natural final isomorphism
\begin{equation*}
    \eta_k :
    G_*^{(a,b)}(\bsigma) \xrightarrow{\simeq} 
     G_*^{(a,b)}((\bepsilon^k,\bsigma)).
\end{equation*}

\subsection{Interpolation isomorphisms}
\label{se:interpolation}

We start this section with a general statement that is easily
deduced from Morse theory.
\begin{prop}\label{prop:isotopy}
    Let $X$ be a closed manifold and $m>0$.
    Let $f_{s,t}:X\to\R$, $s\in[0,1]$, $t\in [-m,m]$, be a $C^1$-family
    of maps. We suppose that for all $s\in[0,1]$, $t\in(-m,m)$ and
    $x\in X$, $\frac{\ud}{\ud t} f_{s,t}(x) \leq 0$.
    If $a,b\in(-m,m)$, $a\leq b$, satisfy that $0$ is a regular value of
    $f_{s,a}$ and $f_{s,b}$ for all $s\in[0,1]$, then the inclusion
    $X\hookrightarrow [0,1]\times X$, $x\mapsto (s,x)$,
    induces the following isomorphism in homology for all $s\in[0,1]$
    \begin{equation*}
        H_*\left(\{ f_{s,b} \leq 0 \},\{ f_{s,a} \leq 0\}\right) \xrightarrow{\simeq}
        H_*\left(\{ (r,x)\ |\ f_{r,b}(x) \leq 0 \},
        \{ (r,x)\ |\ f_{r,a}(x) \leq 0\}\right),
    \end{equation*}
    where $(r,x)$ are describing $[0,1]\times X$ in the right hand side of the
    arrow.
    The analogous non-relative statement is also true:
    let $f_s : X\to\R$, $s\in [0,1]$, be a $C^1$-family of maps
    with $0$ as a regular value.
    Then the inclusion
    $X\hookrightarrow [0,1]\times X$, $x\mapsto (s,x)$,
    induces the following isomorphism in homology for all $s\in[0,1]$
    \begin{equation*}
        H_*\left(\{ f_s \leq 0 \}\right) \xrightarrow{\simeq}
        H_*\left(\{ (r,x)\ |\ f_r(x) \leq 0 \}\right).
    \end{equation*}
\end{prop}

\begin{proof}
    For any interval $I\subset [0,1]$, let $f_{I,t}:I\times X\to \R$
    be the map $f_{I,t}(r,x) := f_{r,t}(x)$.
    Let $a\leq b$ be real numbers as above.
    By compacity of $[0,1]$, there exists an $\varepsilon > 0$ such that
    $[-\varepsilon,2\varepsilon]$ contains only regular values of $f_{s,a}$
    and $f_{s,b}$ for all $s\in [0,1]$.
    By compacity, there also exists a $\delta>0$ such that 
    $\| f_{s,c} - f_{r,c} \|_\infty < \varepsilon$ for $|s-r|\leq \delta$
    and $c\in\{ a,b\}$.

    We recall that if topological pairs $A\subset B\subset C\subset D$ satisfy that the
    maps induced by inclusions $H_*(A)\to H_*(C)$ and $H_*(B)\to H_*(D)$ are isomorphisms, then
    the map induced by inclusion $H_*(B)\to H_*(C)$ is also an isomorphism.
    We apply this result to $A:=(\{ f_{I,b}\leq -\varepsilon \}, \{ f_{I,a}\leq -\varepsilon \})$,
    $B:=(I\times \{ f_{s,b}\leq 0 \}, I\times \{ f_{s,a}\leq 0 \})$,
    $C:=(\{ f_{I,b}\leq \varepsilon \}, \{ f_{I,a}\leq \varepsilon \})$ and
    $D:=(I\times \{ f_{s,b}\leq 2\varepsilon \}, I\times\{ f_{s,a}\leq 2\varepsilon \})$
    for $I\subset [0,1]$ an interval of length less than $\delta$ and $s\in I$.
    Indeed, these topological pairs are increasing for $\subset$ by definition of $\delta$
    and the needed isomorphisms come from the Morse deformation lemma which
    can be applied by definition of $\varepsilon$.
    We thus have the following commutative diagram:
    \begin{equation*}
        \begin{gathered}
            \xymatrix{
                H_*(I\times \{ f_{s,b}\leq 0 \}, I\times \{ f_{s,a}\leq 0 \}) \ar[r]^-{\simeq} &
                H_*(\{ f_{I,b}\leq \varepsilon \}, \{ f_{I,a}\leq \varepsilon \}) \\
                H_*(\{s\}\times \{ f_{s,b}\leq 0 \}, \{s\}\times\{ f_{s,a}\leq 0 \})
                \ar[r]^-{\simeq} \ar[u]^-{\simeq} &
                H_*(\{ f_{I,b}\leq 0\}, \{ f_{I,a}\leq 0\}) \ar[u]^-{\simeq}
            }
        \end{gathered},
    \end{equation*}
    where every map is induced by inclusion. The top arrow is an isomorphism by the above
    general fact. The left hand side arrow is an isomorphism because $I$ retracts on $\{s\}$.
    The right hand side arrow is an isomorphism by the Morse deformation lemma,
    since $[0,\varepsilon]$ contains only regular values of $f_{I,a}$ and $f_{I,b}$.
    Therefore, the bottom arrow is an isomorphism.
    
    According to the Mayer-Vietoris long exact sequence,
    given topological pairs $A$ and $B$,
    if the inclusion maps $H_*(A\cap B) \to H_*(A)$ and
    $H_*(A\cap B)\to H_*(B)$ are isomorphisms, then the inclusion map
    $H_*(A\cap B) \to H_*(A\cup B)$ is an isomorphism.
    We apply this result to 
    $A:=(\{ f_{I,b}\leq 0\}, \{ f_{I,a}\leq 0\})$ and
    $B:=(\{ f_{J,b}\leq 0\}, \{ f_{J,a}\leq 0\})$ for increasing
    length of $I$ and $\length(J)\leq \delta$ to show inductively
    the result.
\end{proof}

Another way to proceed is to remark that 
$(\{ f_{I,b}\leq 0\} \setminus \{ f_{I,a} <0 \},
\{ f_{I,a} = 0 \})$
retracts on
$(\{ f_{s,b}\leq 0\} \setminus \{ f_{s,a} <0 \},
\{ f_{s,a} = 0 \})$
relative to boundaries through the gradient
flow of the restriction of the projection $I\times X\to I$,
which has no critical points under the hypothesis made on $a$ and $b$.

Let $s\mapsto \bsigma^{(s)}$ be a $C^1$-family 
(or more generally $C^1$-piecewise) of tuples associated with
the same $\C$-equivariant Hamiltonian diffeomorphism $\Phi$
without $\C$-lines of fixed points.
We apply Proposition~\ref{prop:isotopy} to the following family of
maps:
\begin{equation*}
    f_s := \widehat{F}_{\bsigma^{(s)}} : \CP^N\to\R.
\end{equation*}
The assumptions of Proposition~\ref{prop:isotopy} are
fulfilled and we define $\Delta$ so that the following diagram
commutes:
\begin{equation*}
    \begin{gathered}
        \xymatrix{
            HZ_*(\bsigma^{(0)}) \ar[r]^-{\simeq}
            \ar[d]^-{\Delta}_-{\simeq}
                & H_{*+i_0} (A) \\
                HZ_*(\bsigma^{(1)}) \ar[ru]^-{\simeq}
            }
    \end{gathered},
\end{equation*}
where the non-vertical arrows are the inclusions maps and
\begin{equation*}
A := \left\{ (r,x)\ |\ f_r (x)\leq 0 \right\}.
\end{equation*}
Since these isomorphisms are defined with inclusion maps the above way,
they clearly commute with inclusion and boundary morphisms.
In the same way, let $(\eta_t^{(s)})$ be a $C^1$-family
of tuples so that for each $s$ the family
$s\mapsto \eta_t^{(s)}$ is associated with the same
$\C$-equivariant Hamiltonian diffeomorphism.
If the associated map $f_{s,t}:\CP^N\to\R$ satisfies the assumption
of Proposition~\ref{prop:isotopy}, then we can define the associated isomorphism
\begin{equation*}
    \Delta : HZ_*(\boldeta_b^{(0)},\boldeta_a^{(0)}) 
    \xrightarrow{\simeq}
    HZ_*(\boldeta_b^{(1)},\boldeta_a^{(1)}).
\end{equation*}
As an important example, 
let $s\mapsto \bsigma^{(s)}$ be a $C^1$-family 
(or more generally $C^1$-piecewise) of tuples associated with
the same $\C$-equivariant Hamiltonian diffeomorphism $\Phi$.
If $a\leq b$ are not action values of $\bsigma$,
then $\eta_t^{(s)} := \bsigma^{(s)}_{m,t}$ satisfies the above
assumption and $\Delta$ is an isomorphism
\begin{equation*}
    \Delta : G_*^{(a,b)}(\bsigma^{(0)},m)
    \xrightarrow{\simeq} G_*^{(a,b)}(\bsigma^{(1)},m).
\end{equation*}
We will call the isomorphism $\Delta$ the interpolation isomorphism associated
with $(\bsigma^{(s)})$ and write in symbols $\Delta \longleftrightarrow
s\mapsto \bsigma^{(s)}$.

\begin{prop}\label{prop:interpolationComposition}
    Let $\Delta$, $\Delta'$ and $\Delta''$ be the interpolation isomorphisms
    associated with $(\bsigma^{(s)})$, $(\boldeta^{(s)}_t)$ and
    $(\bsigma^{(s)},\bepsilon,\boldeta^{(s)}_t)$ respectively.
    The following diagram commutes:
    \begin{equation*}
        \begin{gathered}
            \xymatrix{
                HZ_*\left(\bsigma^{(0)}\right)\otimes 
                HZ_*\left(\boldeta^{(0)}_1,\boldeta^{(0)}_0\right) \ar[r]^-{\compi}
                \ar[d]^-{\Delta\otimes\Delta'}_-{\simeq}
                & HZ_{*-2d}\left(
                \left(\bsigma^{(0)},\bepsilon,\boldeta^{(0)}_1\right),
                \left(\bsigma^{(0)},\bepsilon,\boldeta^{(0)}_0\right)\right)
                \ar[d]^-{\Delta''}_-{\simeq}\\
                HZ_*\left(\bsigma^{(1)}\right)\otimes 
                HZ_*\left(\boldeta^{(1)}_1,\boldeta^{(1)}_0\right) \ar[r]^-{\compi}
                & HZ_{*-2d}\left(\left(\bsigma^{(1)},\bepsilon,\boldeta^{(1)}_1\right),
                \left(\bsigma^{(1)},\bepsilon,\boldeta^{(1)}_0\right)\right)
            }
        \end{gathered}.
    \end{equation*}
\end{prop}

\begin{proof}
    One can assume that either $(\bsigma^{(s)})$ or
    $(\boldeta^{(s)}_t)$ is independent of $s$.
    The proof of the proposition is a consequence of the naturality of
    $(B_{n,n'})_*$
    and a slightly generalized version of $\pj_*$ to projective bundles.
    Let $I:=[0,1]$ and let us extend $\pj_*$ to subsets of $I\times\CP^N$
    the following way.
    Let $A\subset \CP^n$ and $B\subset I\times \CP^m$, we set
    $B_s := B\cap \{s\}\times \CP^m$
    and define $A*B\subset I\times\CP^{m+n+1}$ by
    $A*B := \bigcup_s \{s\}\times (A * B_s)$.
    Let $E_{A,B}$ be the set of those $(a,(s,b),(t,c))$'s with $a\in A$, $(s,b)\in B$
    and $(t,c)\in A*B$ such that $s=t$ and $c\in (ab)$;
    let $p_1 : E_{A,B} \to A\times B$ and $p_2 : E_{A,B} \to A*B$
    be associated projection maps.
    Now $p_1$ is a $\CP^1$-fiber bundle and
    $\pj_* : H_*(A\times B) \to H_{*+2}(A*B)$ is defined by $(p_2)_*\circ (p_1)^*$.
    Since $E_{A,B_s}$ is the restriction of the fiber bundle $E_{A,B}$ to
    $A\times B_s$,
    by naturality of the morphisms involved, the following diagram
    commutes for all $s\in I$:
    \begin{equation*}
        \begin{gathered}
            \xymatrix{
                H_*(A\times B_s) \ar[d] \ar[r]^-{\pj_*}
            & H_{*+2} (A*B_s) \ar[d] \\
            H_*(A\times B) \ar[r]^-{\pj_*}
            & H_{*+2} (A*B)
        }
    \end{gathered},
\end{equation*}
where the vertical arrows are inclusion morphisms induced by
$\{s\}\hookrightarrow I$.
By giving to $Z(\boldeta^I_t)$ the meaning of
$\bigcup_s \{s\}\times Z(\boldeta^{(s)}_t)$ and then extending the
definition of $HZ_*$ accordingly, we deduce that the following diagram
commutes for all $s\in I$:
\begin{equation*}
    \begin{gathered}
        \xymatrix{
            HZ_*\left(\bsigma\right)\otimes 
            HZ_*\left(\boldeta^{(s)}_1,\boldeta^{(s)}_0\right) \ar[r]^-{\compi}
            \ar[d]^-{\simeq}
                & HZ_{*-2d}\left(
                    \left(\bsigma,\bepsilon,\boldeta^{(s)}_1\right),
                \left(\bsigma,\bepsilon,\boldeta^{(s)}_0\right)\right)
                \ar[d]^-{\simeq}\\
                HZ_*\left(\bsigma\right)\otimes 
                HZ_*\left(\boldeta^I_1,\boldeta^I_0\right) \ar[r]^-{\compi}
                & HZ_{*-2d}\left(\left(\bsigma,\bepsilon,\boldeta^I_1\right),
                \left(\bsigma,\bepsilon,\boldeta^I_0\right)\right)
            }
        \end{gathered},
    \end{equation*}
    where the vertical arrows are inclusion morphisms.
    The conclusion follows.
\end{proof}

In particular, the interpolation isomorphism
$\Delta : G_*^{(a,b)}(\bsigma^{(0)},m)\to
G_*^{(a,b)}(\bsigma^{(1)},m)$ commutes with the direct system
$(\theta_m^{m+1})$,
so it ultimately defines the interpolation isomorphism
\begin{equation*}
    \Delta : G_*^{(a,b)}(\bsigma^{(0)})
    \xrightarrow{\simeq} G_*^{(a,b)}(\bsigma^{(1)})
\end{equation*}
that commutes with inclusion and boundary morphisms.

We are now in position to prove that $G^{(a,b)}_*(\bsigma)$ and
its inclusion and boundary morphisms are independent, up to isomorphism, of the choice
of continuous family of $n$-tuples of small Hamiltonian flows $(\bsigma^s)$
generating the $\C$-equivariant Hamiltonian flow $(\Phi_s)$ lifting $(\varphi_s)$
with $\bsigma^0 = \bepsilon^n$ and $\bsigma^1 = \bsigma$.
Indeed, let $(\bsigma^s)$ and $((\bsigma')^s)$ be a $n$-tuple and a $n'$-tuple
of small Hamiltonian flows generating $(\Phi_s)$ with $n\geq n'$.
One can define an isomorphism $G_*^{(a,b)}(\bsigma) \xrightarrow{\simeq}
G_*^{(a,b)}(\bsigma')$ by composition of reduction maps $\eta_k$
and interpolation maps $\Delta$ in the following way:
\begin{equation*}
    \begin{gathered}
        \xymatrix{
            G_*^{(a,b)}(\bsigma) \ar[r]^{\simeq}
            & G_*^{(a,b)}(\bsigma') \\
            G_*^{(a,b)}((\bepsilon^{2n},\bsigma)) \ar[d]^-{\Delta}_-{\simeq}
            \ar[u]^-{\eta_{2n}}_-{\simeq}
            & G_*^{(a,b)}((\bepsilon^{3n-n'},\bsigma')) 
            \ar[u]^-{\eta_{3n-n'}}_-{\simeq} \\
            G_*^{(a,b)}((\bsigma,\bepsilon^{2n})) \ar[r]^-{\Delta'}_-{\simeq}
            & G_*^{(a,b)}((\bsigma,\bsigma^{-1},\bepsilon^{n-n'},\bsigma'))
            \ar[u]^-{\Delta''}_-{\simeq}
        }
    \end{gathered},
\end{equation*}
where $\Delta$, $\Delta'$ and $\Delta''$ are interpolation isomorphisms
associated with isotopies of $3n$-tuples generating the same
Hamiltonian diffeomorphism $\Phi$ in the following way
\begin{eqnarray*}
    \Delta & \longleftrightarrow & s\mapsto
    \begin{cases}
        (\bsigma^{2s},\bsigma^{-2s},\bsigma),& 0\leq s\leq 1/2,\\
        (\bsigma,\bsigma^{-2(1-s)},\bsigma^{2(1-s)}),& 1/2 \leq s\leq 1,
    \end{cases} \\
    \Delta' &\longleftrightarrow &s\mapsto
    (\bsigma,\bsigma^{-s},\bepsilon^{n-n'},(\bsigma')^s),\\
    \Delta'' &\longleftrightarrow& s\mapsto
    (\bsigma^{1-s},\bsigma^{s-1},\bepsilon^{n-n'},\bsigma').
\end{eqnarray*}

\subsection{Composition of generating functions homologies}

Let us fix 2 tuples $\bsigma$, $\bsigma'$ of odd respective sizes $n$ and $n'$,
$a,b,c\in\R$ that are not action values of $\bsigma$ or $\bsigma'$.
Let $m,m'\in\N$ such that $m>2m'>4n$.
The composition map has the form
\begin{equation*}
    HZ_*(\bsigma'_{m',c})\otimes G_*^{(a,b)}(\bsigma,m) \xrightarrow{\compi}
    HZ_*(\boldsymbol{\eta}_b,\boldsymbol{\eta}_a),
\end{equation*}
where
\begin{equation*}\label{simeq:C}
\boldsymbol{\eta}_t := \left(\bsigma',\bdelta_c^{(m')},\bepsilon,
\bsigma,\bdelta^{(m)}_t\right).
\end{equation*}
Let $(\boldsymbol{\eta}_t^s)_s$ be a homotopy of tuples of small Hamiltonians
from $\boldsymbol{\eta}_t^0 = \boldsymbol{\eta}_t$ to
$\boldsymbol{\eta}_t^1 =
(\bsigma',\bepsilon,\bsigma)_{t+c,m+m'}$, for $2|t|<m$,
generating the same diffeomorphism for  a fixed value of $t$.
The condition $m'>2m>4n'$ makes the construction of such a homotopy
possible, we sketch the successive stages of it:
\begin{multline*}
\boldsymbol{\eta}_t = 
\left(\bsigma',\bdelta_c^{(m')},\bepsilon,\bsigma,\bdelta^{(m)}_t\right)
\leadsto 
\left(\bsigma',\bepsilon^{m'n_0+1},\bsigma,\bdelta^{(m)}_{t+c}\right)
\\ \leadsto
\left(\bsigma',\bepsilon,\bsigma,\bsigma^{-1},
\bepsilon^k,\bsigma,\bdelta^{(m)}_{t+c}\right)
\leadsto
\left(\bsigma',\bepsilon,\bsigma,
\bepsilon^{m'n_0},\bdelta^{(m)}_{t+c}\right)
\\ \leadsto
\left(\bsigma',\bepsilon,\bsigma,\bdelta^{(m+m')}_{t+c}\right)
= (\bsigma',\bepsilon,\bsigma)_{t+c,m+m'}.
\end{multline*}
According to the previous section, this homotopy induces
an interpolation isomorphism
\begin{equation*}
    \Delta:
    HZ_*(\boldsymbol{\eta}_b,\boldsymbol{\eta}_a)
    \xrightarrow{\simeq} G_*^{(a+c,b+c)}(
    (\bsigma,\bepsilon,\bsigma'),m+m').
\end{equation*}
The composition of the above composition morphism with $\Delta$
gives this generating functions homology version of the composition morphism:
\begin{equation*}
    HZ_*(\bsigma'_{m',c})\otimes G_*^{(a,b)}(\bsigma,m) \to
    G_{*-2d}^{(a+c,b+c)}((\bsigma',\bepsilon,\bsigma),m+m').
\end{equation*}
We define the same way composition morphism of absolute homology groups:
\begin{equation*}
        HZ_*(\bsigma_{m,t}) \otimes HZ_*(\bsigma'_{m',t'}) \to
        HZ_{*-2d}((\bsigma,\bepsilon,\bsigma')_{m+m',t+t'}).
\end{equation*}
We will denote these maps $\alpha\otimes\beta\mapsto \alpha\compii\beta$
so that in symbols
\begin{equation*}
    \alpha\compii\beta := \Delta(\alpha\compi\beta).
\end{equation*}

Since interpolation isomorphisms commute with inclusion in the total space,
the commutativity of (\ref{dia:compipj}) implies the commutativity of the
analogous square
\begin{equation}\label{dia:compiipj}
    \begin{gathered}
        \xymatrix{
            HZ_*(\bsigma_{m,t})\otimes HZ_*(\bsigma'_{m',t'}) \ar[r]^-{\compii} \ar[d]
            & HZ_{*-2d}((\bsigma,\bepsilon,\bsigma')_{m+m',t+t'}) \ar[d]\\
            H_{*+r}(\CP^N)\otimes H_{*+r'}(\CP^{N'}) \ar[r]^-{\pj_*}
            & H_{*+r+r'+2}(\CP^{N''})
        }
    \end{gathered}.
\end{equation}
This new form of the composition morphism
is also associative.
\begin{cor}
    The following diagram of composition morphisms commutes:
    \begin{equation*}
        \begin{gathered}
        \resizebox{0.98\displaywidth}{!}{            
            \xymatrix{
                HZ_*(\bsigma_{m,c})\otimes HZ_*(\bsigma'_{m',c'})\otimes
                G_*^{(a,b)}(\bsigma'',m'')
                \ar[d] \ar[r]
                & HZ_*(\bsigma_{m,c})\otimes
                G^{(a+c',b+c')}_{*-2d}((\bsigma',\bepsilon,\bsigma''),m'+m'')
                \ar[d] \\
                HZ_{*-2d}((\bsigma,\bepsilon,\bsigma')_{m+m',c+c'})\otimes
                G_*^{(a,b)}(\bsigma'',m'')
                \ar[r]
                & G^{(a+c+c',b+c+c')}_{*-4d}
                ((\bsigma,\bepsilon,\bsigma',\bepsilon,\bsigma''),m+m'+m'')
        } }
        \end{gathered}.
    \end{equation*}
    In symbols, given $\alpha\in HZ_*(\bsigma_{m,c})$,
    $\beta\in HZ_*(\bsigma'_{m',c'})$ and
    $\gamma\in G_*^{(a,b)}(\bsigma'',m'')$,
    \begin{equation*}
        (\alpha\compii\beta)\compii\gamma =
        \alpha\compii(\beta\compii\gamma).
    \end{equation*}
\end{cor}

\begin{proof}
    According to Proposition~\ref{prop:interpolationComposition},
    \begin{equation*}
        (\alpha\compii\beta)\compii\gamma =
        \Delta_2(\Delta_1(\alpha\compi\beta)\compi\gamma)
        = \Delta_2\circ\widetilde{\Delta}_1((\alpha\compi\beta)\compi\gamma),
    \end{equation*}
    where the interpolation isomorphisms are associated with
    homotopies in the following way:
    \begin{eqnarray*}
        \Delta_1 &\longleftrightarrow & (\bsigma_{m,c},\bepsilon,\bsigma'_{m',c'})
        \leadsto (\bsigma,\bepsilon,\bsigma')_{m+m',c+c'},\\
        \widetilde{\Delta}_1 &\longleftrightarrow & 
        ((\bsigma_{m,c},\bepsilon,\bsigma'_{m',c'}),\bepsilon,\bsigma''_{m'',t})
        \leadsto ((\bsigma,\bepsilon,\bsigma')_{m+m',c+c'},\bepsilon,\bsigma''_{m'',t}),\\
        \Delta_2 &\longleftrightarrow &
        ((\bsigma,\bepsilon,\bsigma')_{m+m',c+c'},\bepsilon,\bsigma''_{m'',t})
        \leadsto (\bsigma,\bepsilon,\bsigma',\bepsilon,\bsigma'')_{m+m'+m'',c+c'+t}.
    \end{eqnarray*}
    In the same way,
    \begin{equation*}
        \alpha\compii(\beta\compii\gamma) =
        \Delta_2'(\alpha\compi\Delta_1'(\beta\compi\gamma))
        = \Delta_2'\circ\widetilde{\Delta}_1'(\alpha\compi(\beta\compi\gamma)),
    \end{equation*}
    with convenient interpolation isomorphisms $\Delta_1'$, $\widetilde{\Delta}_1'$
    and $\Delta_2'$.
    According to the associativity of $\compi$
    (Proposition~\ref{prop:associativityComposition}),
    it is enough to prove that
    $\Delta_2\circ\widetilde{\Delta}_1=\Delta_2'\circ\widetilde{\Delta}_1'$.
    These two interpolation isomorphisms are associated with
    homotopies that are themselves homotopic through homotopies
    preserving the associated family of diffeomorphisms.
    It is then a simple consequence of the definition of the interpolation
    isomorphisms.
\end{proof}

As a consequence, the following diagram commutes
\begin{equation*}
    \begin{gathered}
        \xymatrix{
            HZ_*(\bsigma'_{m',c}) \otimes G_*^{(a,b)}(\bsigma,m) \ar[r]^-{\compii}
            \ar[d]^-{\id\otimes\theta_m^{m+1}}_-{\simeq}
            & G_{*-2d}^{(a+c,b+c)}((\bsigma',\bepsilon,\bsigma),m+m') 
            \ar[d]^-{\theta_{m+m'}^{m+m'+1}}_-{\simeq} \\
            HZ_*(\bsigma'_{m',c}) \otimes G_*^{(a,b)}(\bsigma,m+1) \ar[r]^-{\compii}
            & G_{*-2d}^{(a+c,b+c)}((\bsigma',\bepsilon,\bsigma),m+m'+1)
        }
    \end{gathered}.
\end{equation*}
It ultimately defines a morphism
\begin{equation*}\label{eq:compositionUltimate}
    HZ_*(\bsigma'_{m',c}) \otimes G_*^{(a,b)}(\bsigma) \xrightarrow{\compii}
    G_{*-2d}^{(a+c,b+c)}((\bsigma',\bepsilon,\bsigma)),
\end{equation*}
for almost all $a,b\in\overline{\R}$.

By naturality of the composition morphism, 
given $t\geq 0$ and $m\in\N^*$, the following diagram commutes
\begin{equation}\label{dia:compositionInclusion}
    \begin{gathered}
        \xymatrix{
            G^{(a,b)}_*(\bsigma)\otimes HZ_*(\bepsilon_{m,0}) \ar[r]^-{\compii}
            \ar[d]^-{\id\otimes\mathrm{inc}_*}
            & G^{(a,b)}_{*-2d}((\bsigma,\bepsilon^2)) \ar[d]^-{\mathrm{inc}_*}\\
            G^{(a,b)}_*(\bsigma)\otimes HZ_*(\bepsilon_{m,t}) \ar[r]^-{\compii}
            & G^{(a+t,b+t)}_{*-2d}((\bsigma,\bepsilon^2))
        }
    \end{gathered}.
\end{equation}

\subsection{Properties of the generating functions homology}
\label{se:propertiesG}

Let $\bsigma$ and $\bsigma'$ be two different tuples of
$(h_s)$.
We proved that the graded modules $G_*^{(a,b)}(\bsigma)$
and $G_*^{(a,b)}(\bsigma')$ are isomorphic and that
there exists a family of isomorphism compatible with the
inclusion maps so it makes sense to define $G_*^{(a,b)}(h_s)$
in Section~\ref{se:outline}.
Nevertheless, we will keep track of the specific choice of $\bsigma$
in our statements for the sake of being precise.

Let us first focus on the special case $\bsigma = \bepsilon$.
Let us denote by $T_{m,t}$
the family of generating functions associated with $(\bepsilon_{m,t})_t$.
Since the elementary generating function of $\delta_s$ is
$u\mapsto -\tan(\pi s)\| u\|^2$, the map $T_{m,t}$ is a quadratic
form.
Since $T_{m,t}$ is a generating function, its kernel as a quadratic form
has dimension $2(d+1)$.
We had already remarked that $T_{m,0}$ is equivalent to $-T_{m,0}$
(they both generates the identity) which implies that
$\ind T_{m,0} = mn_0(d+1)$ (Proposition~\ref{prop:Qn}).
The variation of index is governed by the Maslov index of
$(e^{-2i\pi t})$ so that
\begin{equation*}
    \ind T_{m,t} - \ind T_{m,0} = 2(d+1)\lfloor t\rfloor,
\end{equation*}
(See \cite[Section~3 and Lemma~5.5]{periodicCPd}).
Therefore, there exists an increasing sequence of
complex projective subspaces $P_{m,-m}\subset P_{m,-m+1}
\subset \cdots \subset P_{m,m}$
such that $P_{m,k} \simeq \CP^{(d+1)(k+mn_0/2)}$ and
$Z(\bepsilon_{m,t})$ retracts on $P_{m,\lfloor t\rfloor}$
inducing an equivalence between the persistence modules
$(H_*(P_{m,\lfloor t\rfloor}))$
and $(H_*(Z(\bepsilon_{m,t})))$, $-m\leq t\leq m$.
Thus, as a graded $R$-module,
\begin{equation*}
    HZ_*(\bepsilon_{m,t}) = \bigoplus_{k=-(d+1)mn_0/2}^{d+(d+1)\lfloor t\rfloor}
    Ra_k^{(mn_0+1)}(t),
\end{equation*}
where $a_k^{(mn_0+1)}(t)$ is the generator of degree $2k$ identified with
the class $[\CP^l]$ of appropriate degree $2l=2k+(d+1)mn_0$ under the isomorphism
induced by $P_{m,\lfloor t\rfloor}\hookrightarrow Z(\bepsilon_{m,t})$.
The inclusion morphism $HZ_*(\bepsilon_{m,t}) \to HZ_*(\bepsilon_{m,s})$
maps each $a_k^{(mn_0+1)}(t)$ to $a_k^{(mn_0+1)}(s)$
(for $-m\leq t\leq s\leq m$).
Hence,
\begin{equation*}
    G_*^{(a,b)}(\bepsilon,m) =
    \bigoplus_{k=d+(d+1)\lfloor a\rfloor}^{d+(d+1)\lfloor b\rfloor}
    R\alpha_k^{(m)}(a,b),
\end{equation*}
for $-m < a\leq b< m$,
where $\alpha_k^{(m)}(a,b)$ is the image of $a_k^{(mn_0+1)}(b)$
under the inclusion morphism $HZ_*(\bepsilon_{m,b})
\to G_*^{(a,b)}(\bepsilon,m)$.
According to the commutativity of (\ref{dia:thetaCPN}),
one has $\theta_m^{m+1}\alpha_k^{(m)}(a,b) = \alpha_k^{(m+1)}(a,b)$.
We set $\alpha_k(a,b) := \theta_m^\infty \alpha_k^{(m)}(a,b)$.
For $a<b<c$, if $\alpha_k(b,c)$ is well-defined,
then $\alpha_k(a,c)$ is also well-defined and sent to
the former.
We deduce that there exists a well-defined $\alpha_k(-\infty,c)
\in G_{2k}^{(-\infty,c)}(\bepsilon)$ sent to $\alpha_k(a,c)$
for all $a\leq c$.
Let $\alpha_k$ be the image of $\alpha_k(-\infty,c)$ under
$G_{2k}^{(-\infty,c)}(\bepsilon) \to G_{2k}^{(-\infty,+\infty)}(\bepsilon)$.
Finally,
\begin{equation*}
    G_*^{(-\infty,+\infty)}(\bepsilon) = \bigoplus_{k\in\Z} R\alpha_k,
\end{equation*}
we will show in Theorem~\ref{thm:spectral} that this is also the case for
any $\bsigma$.

In order to show the ``periodicity'' of the persistence module
of $\bsigma$, let us define a natural isomorphism
\begin{equation}\label{iso:periodic}
    G_*^{(a,b)}(\bsigma) \xrightarrow{\simeq}
    G_{*+2(d+1)}^{(a+1,b+1)}(\bsigma).
\end{equation}
In order to simplify the exposition, let us set
$a_d := a_d^{(mn_0+1)}(0)\in HZ_{2d}(\bepsilon_{m,0})$
and $a_{2d+1} := a_{2d+1}^{(mn_0+1)}(1)\in HZ_{2(2d+1)}(\bepsilon_{m,1})$.
According to Proposition~\ref{prop:compad},
the morphism $G_*^{(a+1,b+1)}(\bsigma) \to G_*^{(a+1,b+1)}((\bepsilon^2,
\bsigma))$, $\alpha\mapsto a_d\compii \alpha$, is an isomorphism;
let us write $\alpha\mapsto a_d^{-1}\compii \alpha$ its inverse morphism.
We define the morphism (\ref{iso:periodic}) by
$\alpha\mapsto a_d^{-1}\compii a_{2d+1} \compii \alpha$.

\begin{prop}\label{prop:periodicity}
    The morphism (\ref{iso:periodic}) is an isomorphism
    commuting with inclusion and boundary morphisms.
\end{prop}

\begin{proof}
    The naturality of this morphism comes from the naturality of
    $\alpha\mapsto a_d\compii\alpha$ and $\alpha\mapsto a_{2d+1}\compii\alpha$.
    It remains to prove that $\alpha\mapsto a_{2d+1}\compii\alpha$
    is an isomorphism.
    Let us set $a_{-1} := a_{-1}^{(mn_0+1)}(-1) \in HZ_{-2}(\bepsilon_{m,-1})$.
    According to the commutativity of (\ref{dia:compiipj}),
    $a_{2d+1}\compii a_{-1} = a_{-1}\compii a_{2d+1}
    = a_d$ where $a_d$ is identified with $a_d^{(2mn_0+3)}(0)$.
    Thus, the following diagram commutes
\begin{equation*}
    \begin{gathered}
        \xymatrixcolsep{4pc}
        \xymatrix{
            G_*^{(a,b)}(\bsigma) \ar[r]^-{a_{2d+1}\compii\cdot}
            \ar[rd]_-{a_d\compii\cdot}^-{\simeq}
            & G_{*+2(d+1)}^{(a+1,b+1)}((\bepsilon^2,\bsigma)) 
            \ar[d]^-{a_{-1}\compii\cdot} \ar[rd]_-{a_d\compii\cdot}^-{\simeq}\\
            & G_*^{(a,b)}((\bepsilon^4,\bsigma)) \ar[r]_-{a_{2d+1}\compii\cdot}
            & G_{*+2(d+1)}^{(a+1,b+1)}((\bepsilon^6,\bsigma))
        }
    \end{gathered}
\end{equation*}
so every arrow in it is an isomorphism.
\end{proof}

\begin{thm}\label{thm:spectral}
    Let $\bsigma$ be a tuple of small $\C$-equivariant Hamiltonian
    diffeomorphisms
    associated with the Hamiltonian diffeomorphism $\varphi$ of $\CP^d$.
    As a graded $R$-module,
    \begin{equation*}
        G_*^{(-\infty,+\infty)}(\bsigma) = \bigoplus_{k\in\Z} R\alpha_k
    \end{equation*}
    for some non-zero $\alpha_k$'s with $\deg\alpha_k = 2k$.
    For all $k\in\Z$, let
    \begin{equation*}
        c_k(\bsigma) := \inf \left\{ t\in \R\ |\ \alpha_k \in
            \im \left( G_*^{(-\infty,t)}(\bsigma) \to
        G_*^{(-\infty,+\infty)}(\bsigma) \right)\right\}.
    \end{equation*}
    Then for all $k\in\Z$, $c_k(\bsigma)\in\R$ is an action value
    of $\bsigma$ and
    $c_{k+d+1}(\bsigma)=c_k(\bsigma)+1$. Moreover \[c_k(\bsigma) \leq c_{k+1}(\bsigma)\] for all $k \in \Z,$ and if there exists $k\in\Z$ such that $c_k(\bsigma)=c_{k+1}(\bsigma)$,
    then $\varphi$ has infinitely many fixed points of action $c_k(\bsigma)$.
    If $d+1$ consecutive $c_k(\bsigma)$'s are equal then $\varphi=\id$.
\end{thm}

The following proposition is contained in the proof of Theorem~\ref{thm:spectral}.
It makes precise the fact that one could informally think of $\alpha_k$ as ``the class
$[\CP^{k+\infty}]\in H_{2(k+\infty)}(\CP^{2\infty})$''.

\begin{prop}\label{prop:spectral}
    Let $\bsigma$ be an $n_1$-tuple of small $\C$-equivariant Hamiltonian diffeomorphisms
    and let $k\in\Z$.
    Let $m\in\N$, $K\in\R$ and $t\in\R$ such that $-m<-K<c_k(\bsigma)<t<m$.
    We set $r:=(n_1+mn_0-1)(d+1)$ and $N:=(n_1+mn_0)(d+1)-1$.
    Let $\alpha'_k\in G_{2k}^{(-\infty,t)}(\bsigma)$ be a class sent to $\alpha_k$
    under the inclusion map $G_*^{(-\infty,t)}(\bsigma)\to G_*^{(-\infty,+\infty)}(\bsigma)$.
    Then the image $\alpha_k''\in G_{2k}^{(-K,t)}(\bsigma)$ under the inclusion map
    is non-zero and there exists $b_k'\in G_{2k}^{(-K,t)}(\bsigma,m)$ such that
    $\theta_m^\infty(b_k') = \alpha_k''$.
    The class $b_k'$ is the image of a class $b_k\in
    HZ_{2k}(\bsigma_{m,t})$ that is sent to
    $[\CP^{k+r/2}]\in H_{2k+r}(\CP^N)$ under the map induced by inclusion
    $Z(\bsigma_{m,t}) \hookrightarrow \CP^N$.
\end{prop}

\begin{proof}[Proof of Theorem~\ref{thm:spectral}]
    In the beginning of this section, we have proved the theorem
    for $\bsigma=\bepsilon$; hence for $\bsigma=\bepsilon^n$
    for all odd $n\in\N$.
    Let us show that the persistence module of any $n$-tuple
    $\bsigma$ and the persistence module of $\bepsilon^n$ are
    $\delta$-interleaved for some $\delta>0$.

    Let $\bsigma$ be an $n$-tuple and let us denote
    $f_1,\ldots,f_n : \C^{d+1} \to \R$ the elementary generating functions
    of $\sigma_1,\ldots,\sigma_n$ respectively.
    Let $\tau\in (0,1/2)$ be such that
    \begin{equation*}
        \max_{z\in B,j\in\{ 1,\ldots,n\}} |f(z)| \leq \tan(\pi\tau),
    \end{equation*}
    where $B\subset \C^{d+1}$ denotes the closed unit ball.
    Then
    \begin{equation*}
        F_{(\delta_\tau,\ldots,\delta_\tau)_{m,t}} \leq
        F_{\bsigma_{m,t}} \leq
        F_{(\delta_{-\tau},\ldots,\delta_{-\tau})_{m,t}}.
    \end{equation*}
    Composing interpolation morphisms associated with
    $(\delta_{\pm\tau},\ldots,\delta_{\pm\tau})_{m,t}
    \leadsto (\bepsilon^n)_{m,t\pm n\tau}$,
    \begin{equation*}
        s\mapsto 
        (\delta_{\pm(1-s)\tau},\ldots,\delta_{\pm(1-s)\tau})_{m,t+sn\tau},
    \end{equation*}
    with inclusion morphisms induced by
    \begin{equation*}
        Z({(\delta_{-\tau},\ldots,\delta_{-\tau})_{m,t}})
        \subset
        Z(\bsigma_{m,t})
        \subset
        Z({(\delta_{\tau},\ldots,\delta_{\tau})_{m,t}}),
    \end{equation*}
    one gets a natural
    (\emph{i.e.} that commutes with inclusions and the direct system)
    $n\tau$-interleaving between $G_*^{(-K,t)}(\bsigma,m)$
    and $G_*^{(-K,t)}(\bepsilon^n,m)$ with $K\in (-m,1-m)$ fixed,
    almost every $t\in (-K,m)$ and large $m$.
    These natural interleavings induced a $n\tau$-interleaving
    between $G_*^{(-\infty,t)}(\bsigma)$ and $G_*^{(-\infty,t)}(\bepsilon^n)$.
    The fact that $G_*^{(-\infty,+\infty)}(\bsigma)$ is isomorphic
    as a graded $R$-module to $G_*^{(-\infty,+\infty)}(\bepsilon^n)$
    is now a direct consequence of the existence of an interleaving
    between their associated persistence modules.
    The characterisation of $c_k(\bsigma)$ and $\alpha_k$ given by
    Proposition~\ref{prop:spectral} is true for $\bsigma = \bepsilon^n$
    by the above discussion.
    Since the $n\tau$-interleaving is ultimately induced by inclusions,
    Proposition~\ref{prop:spectral} is still true for
    $\bsigma$ that is ``$\alpha_k$ corresponds to $[\CP^{N(m)+k}]$''
    seen in $G_*^{(a,b)}(\bsigma,m)$ for $m$, $-a$ and $b$ large enough.

    The fact that $c_{k+d+1}(\bsigma) = c_k(\bsigma) +1$ is a direct consequence
    of Proposition~\ref{prop:periodicity} applied to $a=-\infty$ and
    $b= c_k(\bsigma) + \varepsilon$ for $\varepsilon>0$.
    The last statements of the Theorem~\ref{thm:spectral} are consequences
    of the Lyusternik-Schnirelmann theory, as one can see that
    $G_*^{(a,b)}(\bsigma,m)$ is naturally isomorphic to the relative homology
    of sublevel sets of one single map (see \cite[Section~5]{The98} or
    \cite[Section~5.4]{periodicCPd} for
    details).
\end{proof}

The dual statement holds for the generating functions cohomology groups
with the additional structure of $R$-algebra induced by the cup-product:
the $R$-algebra $G^*_{(-\infty,+\infty)}(\bsigma)$ is generated by
a class $e$ of degree 2 that is invertible and of infinite order:
\begin{equation*}
    G^*_{(-\infty,+\infty)}(\bsigma) = \bigoplus_{k\in\Z} R e^k.
\end{equation*}
Informally one could think of $e^k$ as ``the class $u^{k+\infty}
\in H^{2(k+\infty)}(\CP^{2\infty})$'' where $u$ is the generator
of $H^*(\CP^\infty)$.
Therefore, one can define alternatively the spectral values $c_k(\bsigma)$ by
\begin{equation*}
        c_k(\bsigma) = \inf \left\{ t\in \R\ |\ e^k \not\in
        \ker \left( G^*_{(-\infty,+\infty)}(\bsigma) \to
        G^*_{(-\infty,t)}(\bsigma) \right)\right\}.
\end{equation*}

Given an $n$-tuple $\bsigma=(\sigma_1,\ldots,\sigma_n)$,
we denote by $\bsigma^{-1}$ the $n$-tuple $(\sigma_n^{-1},\ldots,\sigma_1^{-1})$.
We recall that if $f:\C^{d+1}\to\R$ is an elementary generating function of $\sigma$
then $-f$ is an elementary generating function of $\sigma^{-1}$.
Therefore, one has
\begin{equation}\label{eq:Finverse}
    F_{\bsigma^{-1}}(v_1,v_2,\ldots,v_n) = -F_\bsigma(v_1,v_n,v_{n-1},\ldots,v_2)
\end{equation}
(this identity has already been stated in the special case $\bsigma=\bepsilon^n$ in
Proposition~\ref{prop:Qn}).

Given an $n$-tuple $\bsigma$ and an $m$-tuple $\bsigma'$, one has
\begin{equation}\label{eq:Fcyclic}
    F_{(\bsigma,\bsigma')}(\mathbf{v},\mathbf{v}') =
    F_{(\bsigma',\bsigma)}(\mathbf{v}',\mathbf{v}),\quad
    \forall \mathbf{v} \in (\C^{d+1})^n,
    \forall \mathbf{v}' \in (\C^{d+1})^m.
\end{equation}

\begin{prop}[Poincaré duality]\label{prop:poincareduality}
    Let $\bsigma$ be a tuple of small $\C$-equivariant Hamiltonian diffeomorphisms
    of $\C^{d+1}\setminus 0$.
    There exists a duality isomorphism between generating function homology and cohomology
    \begin{equation*}
        PD: G^*_{(a,b)}(\bsigma) \xrightarrow{\simeq}
        G_{2d-*}^{(-b,-a)}(\bsigma^{-1}),
    \end{equation*}
    with $-\infty\leq a\leq b\leq +\infty$ and $a,b$ not action values.
    This isomorphism is natural: it commutes with inclusion and boundary maps.
\end{prop}

\begin{proof}
    Let us recall a version of the classical Poincaré duality
    (see for instance \cite[Theorem~3.43]{Hatcher}).
    Let $M$ be a compact $R$-orientable $n$-dimensional manifold
    whose boundary $\partial M$ is the union of two disjoint manifolds $A$
    and $B$. Then the cap-product with the fundamental class
    $[M]\in H_n(M,\partial M)$ gives a natural isomorphism
    $H^*(M,A) \to H_{n-*}(M,B)$.
    This general statement can be applied to sublevel sets of
    a $C^1$ map of a closed $R$-orientable $n$-manifold $f:W\to\R$ in the following way.
    Let $a<b$ be regular values of $f$ so that $A:= \{ f=a\}$
    and $B:=\{ f=b\}$ are the disjoint pieces of the boundary of the
    compact $R$-orientable $n$-manifold $M:= \{ a \leq f\leq b\}$.
    Now, by excision (which can be used because the boundary of
    $M$ admits a collar neighborhood)
    \begin{equation*}
        H^*(M,A) \simeq H^*\big(\{ f\leq b\},\{ f\leq a\}\big)   \text{ and } 
        H_*(M,B) \simeq H_*\big(\{ -f\leq -a \},\{ -f\leq -b\}\big).
    \end{equation*}
    Finally, one has the duality isomorphism
    \begin{equation*}
        PD : H^*\big(\{ f\leq b\},\{ f\leq a\}\big)
        \xrightarrow{\simeq}
        H_{n-*}\big(\{ -f\leq -a \},\{ -f\leq -b\}\big).
    \end{equation*}

    Let us apply the same idea in $W:= \CP^N$ with $N:=(n+mn_0)(d+1)-1$:
    \begin{align*}
        G^*_{(a,b)}(\bsigma,m) &=
        H^{(n+mn_0-1)(d+1)+*}\left(
            \left\{ \widehat{F}_{\bsigma_{m,b}} \leq 0 \right\},
        \left\{ \widehat{F}_{\bsigma_{m,a}} \leq 0 \right\} \right)\\
                               &\simeq
        H_{2N-((n+mn_0-1)(d+1)+*)}\left(
            \left\{ \widehat{F}_{\bsigma_{m,a}} \geq 0 \right\},
        \left\{ \widehat{F}_{\bsigma_{m,b}} \geq 0 \right\} \right),
    \end{align*}
    where we have used that $0$ is not a critical value of either
    $\widehat{F}_{\bsigma_{m,a}}$ or $\widehat{F}_{\bsigma_{m,b}}$. 
    This last homology group is isomorphic to
    \begin{equation*}
        H_{(n+mn_0-1)(d+1)+(2d-*)}\left(
            \left\{ \widehat{F}_{(\bsigma^{-1})_{m,-a}} \leq 0 \right\},
        \left\{ \widehat{F}_{(\bsigma^{-1})_{m,-b}} \leq 0 \right\} \right)
        = G_{2d-*}^{(-b,-a)}(\bsigma^{-1},m).
    \end{equation*}
    Indeed, according to identity (\ref{eq:Finverse}),
    this homology group is naturally isomorphic to the homology group of a pair
    of sublevel sets of functions $-\widehat{F}_{(\bsigma_{m,t})^{-1}}$.
    The conclusion follows by applying
    Proposition~\ref{prop:isotopy} to an interpolation
    from $(\bsigma_{m,t})^{-1}$ to $(\bsigma^{-1})_{m,-t}$.

    Let us precise the interpolation argument.
    One has $(\bsigma_{m,t})^{-1} = ((\bdelta^{(m)}_t)^{-1},\bsigma^{-1})$.
    According to (\ref{eq:Fcyclic}), the sublevel set induced by this tuple
    is homeomorphic to the one induced by $(\bsigma^{-1},(\bdelta^{(m)}_t)^{-1})$
    so it is enough to find an interpolation from $(\bdelta^{(m)}_t)^{-1}$
    to $\bdelta^{(m)}_{-t}$.
    By definition, 
    \begin{equation*}
        (\bdelta^{(m)}_t)^{-1} = (\bepsilon^{qn_0},\delta_s^{-1},\delta_1^{-1},
        \ldots,\delta_1^{-1}) 
    \end{equation*}
    with $s=t-\lfloor t\rfloor$.
    We remark that $\delta_t^{-1} = \delta_{-t}$.
    It is then easy to find the wanted homotopy
    among tuples of the form $(\delta_{t_1},\delta_{t_2},
    \ldots)$ with $t_1 + t_2 +\cdots = -t$.
\end{proof}

Using both equivalent definitions of the spectral values $c_k(\bsigma)$,
the Poincaré duality implies the following identity.

\begin{cor}\label{cor:spectralinverse}
    Let $\bsigma$ be a tuple of small $\C$-equivariant
    Hamiltonian diffeomorphisms, then
    \begin{equation*}
        c_k (\bsigma^{-1}) = - c_{d-k} (\bsigma),\quad
        \forall k\in\Z.
    \end{equation*}
\end{cor}

Finally, let us apply the composition morphisms to spectral classes
in order to prove the sub-additivity of the spectral values.

\begin{prop}\label{prop:compositionSpectral}
    Let $m,m'\in\N$ and $k,l\in\Z$ be such that $c_k(\bsigma)\in I_m$ and 
    $c_l(\bsigma')\in I_{m'}$
    and let $t\in (c_k(\bsigma),m)$ and $t'\in (c_l(\bsigma'),m')$.
    Let $b_k\in HZ_{2k}(\bsigma_{m,t})$ and $b'_l\in HZ_{2l}(\bsigma'_{m',t'})$
    be classes associated with the respective spectral classes $\alpha_k$ and $\alpha_l'$
    in the way expressed in Proposition~\ref{prop:spectral}.
    Then the composition morphism
    \begin{equation*}
        HZ_*(\bsigma_{m,t}) \otimes HZ_*(\bsigma'_{m',t'}) \to
        HZ_{*-2d}((\bsigma,\bepsilon,\bsigma')_{m+m',t+t'})
    \end{equation*}
    maps the class $b_k\otimes b'_l$ to a class $b''_{k+l-d}$
    that is sent to the $[\CP^r]\in H_*(\CP^N)$ of appropriate degree
    under the inclusion morphism.
\end{prop}

\begin{proof}
    This is a direct consequence of the commutativity of (\ref{dia:compiipj}).
\end{proof}

\begin{cor} Given any tuples $\bsigma$ and $\bsigma'$ of small
    $\C$-equivariant Hamiltonian diffeomorphisms, one has
    \begin{equation*}
        c_{k+l-d}((\bsigma,\bepsilon,\bsigma')) \leq
        c_k(\bsigma) + c_l(\bsigma'),\quad
        \forall k,l\in\Z.
    \end{equation*}
\end{cor}

As explained in Section~\ref{se:outline},
we can now associate to every $\bsigma$ a persistence module
$(G_*^{(-\infty,t)}(\bsigma))_t$ that satisfies
the ``periodicity'' property
$G_*^{(-\infty,t+1)}(\bsigma) \simeq G_{*+2(d+1)}^{(-\infty,t)}(\bsigma)$,
the isomorphism being an isomorphism of persistence module
according to the naturality of (\ref{iso:periodic}).
While discussing barcodes properties,
we assume that the persistence module is over a field
and the number of associated fixed points in $\CP^d$
is finite.
Since this periodicity property shifts the degree
by a constant positive integer $2(d+1)$, it
induces a permutation of the bars of the barcode 
sending a bar $[a,b)$ on a bar $[a+1,b+1)$ that
generates a free $\Z$-action on the bars.
A family of representatives of the bars is given
by the union of the barcodes of
$(G_k^{(-\infty,t)}(\bsigma))_t$ for $0\leq k\leq 2d+1$.
For instance, Figure~\ref{fig:barcode} represents a part of
the barcode of some $\bsigma$ associated with a Hamiltonian
diffeomorphism of $\CP^1$.
This barcode has 2 $\Z$-orbits of finite bars
and $d+1=2$ $\Z$-orbits of infinite bars corresponding
to the spectral values $c_1+\Z$ and $c_2+\Z$
where $c_k:=c_k(\bsigma)$.

\begin{figure}
\begin{center}
\begin{small}
\def\svgwidth{1\textwidth}
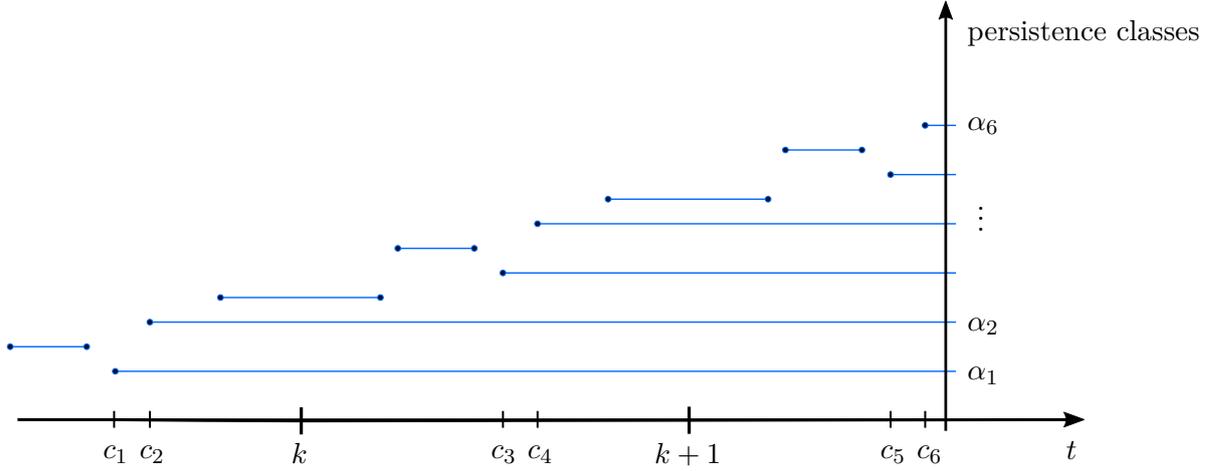
\end{small}
\caption{Barcode of a Hamiltonian diffeomorphism of $\CP^1$
    in the neighborhood of $[k,k+1]$ for some $k\in\Z$
(bars of degree less than 2 are missing).}
\label{fig:barcode}
\end{center}
\end{figure}

\begin{lem}\label{lem:extremities}
    Let $\bsigma$ be a tuple of $\C$-equivariant Hamiltonian
    diffeomorphisms with a finite number of fixed $\C$-lines
    and let $a<b$ that are not action values of $\bsigma$.
    For every field $\F$,
    the number of bars of the barcode of $\bsigma$
    over $\F$ that
    intersect $t=a$ or $t=b$ but not both
    is equal to $\dim G_*^{(a,b)}(\bsigma;\F)$.
\end{lem}

\begin{proof}
    Let us consider the long exact sequence
    \begin{multline}\label{eq:lesGF}
        \cdots \xrightarrow{\partial_{*+1}}
        G_*^{(-\infty,a)}(\bsigma;\F) \to
        G_*^{(-\infty,b)}(\bsigma;\F) \to
        G_*^{(a,b)}(\bsigma;\F)\\
        \xrightarrow{\partial_*}
        G_{*-1}^{(-\infty,a)}(\bsigma;\F) \to
        G_{*-1}^{(-\infty,b)}(\bsigma;\F) \to \cdots.
    \end{multline}
    Applying the normal form theorem of persistence modules,
    one can find bases
    $((\alpha_i),(\delta_j^-))$ and $((\beta_k),(\delta_j^+))$
    of the $\F$-vector spaces
    $G_*^{(-\infty,a)}(\bsigma;\F)$ and
    $G_*^{(-\infty,b)}(\bsigma;\F)$ that are in a canonical
    bijection with bars of the barcode intersecting $t=a$
    and $t=b$ respectively,
    $\delta_j^-$ and $\delta_j^+$ being associated with the
    same bar for each $j$ while the bars associated with
    the $\alpha_i$'s do not intersect $t=b$ and
    the bars associated with the $\beta_k$'s do not intersect $t=a$
    (see Figure~\ref{fig:betatot}).
    In other words, the following diagram commutes
    \begin{equation*}
        \begin{gathered}
            \xymatrix{
                G^{(-\infty,a)}_*(\bsigma;\F) =
                \bigoplus_i \F\alpha_i \oplus \bigoplus_j \F\delta_j^- 
                \qquad\qquad\
                \ar@<-15ex>[d]  \ar@<8ex>[d]^-{\simeq}\\
                G^{(-\infty,b)}_*(\bsigma;\F) =
                \qquad\qquad\ \bigoplus_j \F\delta_j^+ \oplus
                \bigoplus_k \F\beta_k
            }
        \end{gathered},
    \end{equation*}
    where the left arrow is the inclusion morphism
    and the right arrow sends $\delta_j^-$ to $\delta_j^+$ for all $j$
    and the $\alpha_i$'s to $0$.
    Let us recall that, according to the finiteness assumption on the
    number of fixed points, the number of $\alpha_i$'s and
    $\beta_k$'s is finite (here $a$ and $b$ are finite).
    With the above diagram, one can extract
    a short exact sequence of finite dimensional vector spaces
    from the long exact sequence (\ref{eq:lesGF})
    \begin{equation*}
        0 \to \bigoplus_k \F\beta_k \to G_*^{(a,b)}(\bsigma;\F)
        \to \bigoplus_i \F\alpha_i \to 0.
    \end{equation*}
    Hence the result.
    \begin{figure}
        \begin{center}
            \begin{small}
                \def\svgwidth{0.6\textwidth}
\begingroup%
  \makeatletter%
  \providecommand\color[2][]{%
    \errmessage{(Inkscape) Color is used for the text in Inkscape, but the package 'color.sty' is not loaded}%
    \renewcommand\color[2][]{}%
  }%
  \providecommand\transparent[1]{%
    \errmessage{(Inkscape) Transparency is used (non-zero) for the text in Inkscape, but the package 'transparent.sty' is not loaded}%
    \renewcommand\transparent[1]{}%
  }%
  \providecommand\rotatebox[2]{#2}%
  \newcommand*\fsize{\dimexpr\f@size pt\relax}%
  \newcommand*\lineheight[1]{\fontsize{\fsize}{#1\fsize}\selectfont}%
  \ifx\svgwidth\undefined%
    \setlength{\unitlength}{261.21957813bp}%
    \ifx\svgscale\undefined%
      \relax%
    \else%
      \setlength{\unitlength}{\unitlength * \real{\svgscale}}%
    \fi%
  \else%
    \setlength{\unitlength}{\svgwidth}%
  \fi%
  \global\let\svgwidth\undefined%
  \global\let\svgscale\undefined%
  \makeatother%
  \begin{picture}(1,0.72591898)%
    \lineheight{1}%
    \setlength\tabcolsep{0pt}%
    \put(0,0){\includegraphics[width=\unitlength,page=1]{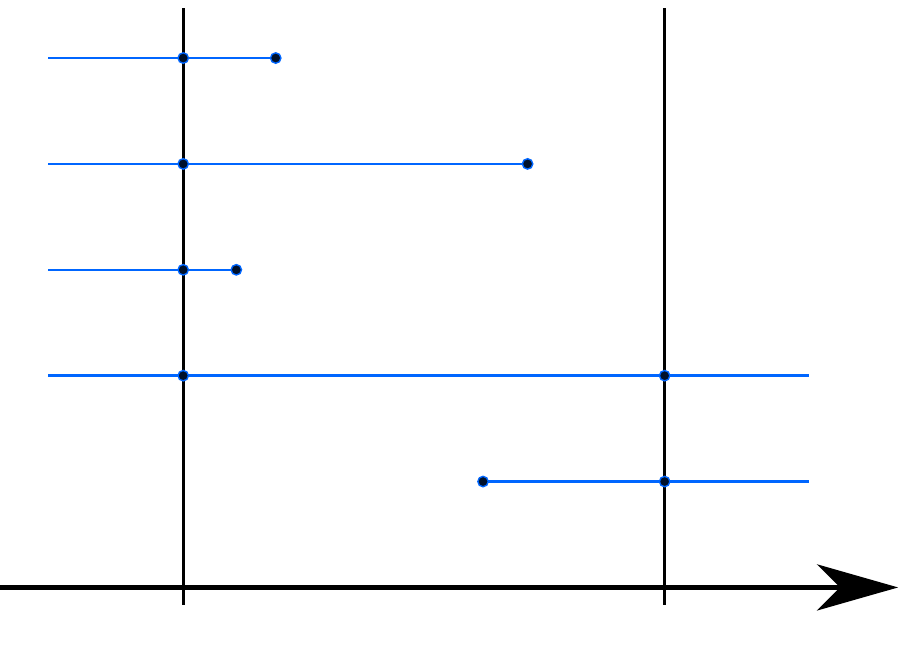}}%
    \put(0.18771684,0.00269918){\makebox(0,0)[lt]{\lineheight{1.25}\smash{\begin{tabular}[t]{l}$a$\end{tabular}}}}%
    \put(0.7183369,0.00269918){\makebox(0,0)[lt]{\lineheight{1.25}\smash{\begin{tabular}[t]{l}$b$\end{tabular}}}}%
    \put(0.76689871,0.22551526){\makebox(0,0)[lt]{\lineheight{1.25}\smash{\begin{tabular}[t]{l}$\beta_k$\end{tabular}}}}%
    \put(0.23422522,0.34398785){\makebox(0,0)[lt]{\lineheight{1.25}\smash{\begin{tabular}[t]{l}$\delta_j^-$\end{tabular}}}}%
    \put(0.76689871,0.34398785){\makebox(0,0)[lt]{\lineheight{1.25}\smash{\begin{tabular}[t]{l}$\delta_j^+$\end{tabular}}}}%
    \put(0.23422522,0.46246045){\makebox(0,0)[lt]{\lineheight{1.25}\smash{\begin{tabular}[t]{l}$\alpha_1$\end{tabular}}}}%
    \put(0.23422522,0.58093313){\makebox(0,0)[lt]{\lineheight{1.25}\smash{\begin{tabular}[t]{l}$\alpha_2$\end{tabular}}}}%
    \put(0.23422522,0.69940573){\makebox(0,0)[lt]{\lineheight{1.25}\smash{\begin{tabular}[t]{l}$\alpha_3$\end{tabular}}}}%
  \end{picture}%
\endgroup%

            \end{small}
            \caption{Relationship between the barcode of $\bsigma$ in the interval
                $(a,b)$ and its persistence module (there are infinitely
                many bars that do not appear in the figure).
                The value $\dim G_*^{(a,b)}(\bsigma;\F)$
        gives the number of $\alpha_i$'s and $\beta_k$'s.}
            \label{fig:betatot}
        \end{center}
    \end{figure}
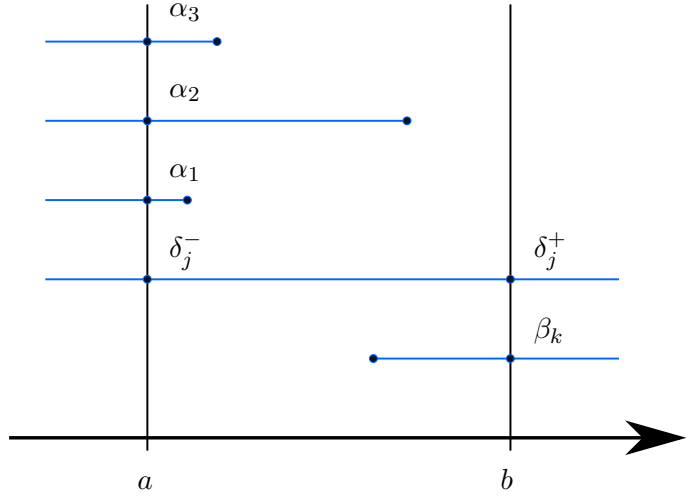   
\end{proof}

\begin{prop}\label{prop:N}
    Given a tuple $\bsigma$ of $\C$-equivariant Hamiltonian
    diffeomorphisms of $\C^{d+1}\setminus 0$
    with a finite number of fixed $\C$-lines,
    for every field $\F$,
    \begin{equation*}
        N(\bsigma;\F) = d+1 + 2 K(\bsigma;\F),
    \end{equation*}
    where $K(\bsigma;\F)$ is the number of $\Z$-orbits of finite bars of the
    persistence module
    of $\bsigma$ over the field $\F$.
    In other words, $N(\bsigma;\F)$ is the number of (finite) extremities of
    a set of representative bars.
\end{prop}

\begin{proof}
    According to the $\Z$-symmetry of the barcode,
    it boils down to proving that the number
    of extremities of the barcode lying inside $[0,1)$
    is equal to $N(\bsigma;\F)$.
    Let $0\leq t_1 < t_2 <\cdots < t_n <1$ be the action
    values of $\bsigma$ in $[0,1)$,
    that is the points where extremities of the barcode could
    appear.
    Let $t_j^\pm := t_j \pm \varepsilon$ where $
    \varepsilon>0$ is strictly less than the minimum
    distance between two action values
    so that $t_j$ is the only action value inside
    $[t_j^-,t_j^+]$.
    According to Lemma~\ref{lem:extremities},
    $\dim G_*^{(t_j^-,t_j^+)}(\bsigma;\F)$ equals
    the number of extremities at $t=t_j$.
    Therefore, we just have to prove that
    \begin{equation*}
        N(\bsigma;\F) =
        \sum_{j=1}^n \dim G_*^{(t_j^-,t_j^+)}(\bsigma;\F).
    \end{equation*}

    Let us denote by $\varphi$ the Hamiltonian diffeomorphism
    associated with $\bsigma$.
    Since $t_j$ is the only action value in $[t_j^-,t_j^+]$,
    an excision argument gives
    \begin{equation*}
        G_*^{(t_j^-,t_j^+)}(\bsigma) \simeq
        \bigoplus_z \lochom_*(\bsigma;z,t_j),
    \end{equation*}
    where the direct sum is over the fixed points $z\in\CP^d$
    of $\varphi$ with action value $t_j$
    (see \cite[Section~5.5]{periodicCPd}).
    By taking these isomorphisms over the field $\F$ for all $j$,
    the conclusion follows.
\end{proof}

\section{Uniform bound on $\betamax$}\label{se:betamax}

\begin{thm}\label{thm:betamax}
    For every tuple of small $\C$-equivariant Hamiltonian diffeomorphisms
    $\bsigma$ generating a Hamiltonian diffeomorphism
    of $\CP^d$ with finitely many fixed points,
    the longest finite bar of its barcode is not greater than 1:
    \begin{equation*}
        \betamax(\bsigma) \leq 1.
    \end{equation*}
\end{thm}

As a matter of fact, the proof allows us to give
the more precise bound:
\begin{equation*}
    \betamax(\bsigma) \leq c_{d+k}(\bsigma) - c_k(\bsigma),
\end{equation*}
for all $k\in \Z$ (in particular, one can always replace $\leq 1$ by $<1$).

In the proof, we will essentially prove that the persistence modules
of $\bsigma$ and $\bepsilon$ are $(c_d(\bsigma),-c_0(\bsigma))$-interleaved.
The isometry theorem between the interleaving distance and
the barcode distance states that the distance between the two associated barcodes
is not more than $c_d(\bsigma)-c_0(\bsigma)$.
Since the barcode of $\bepsilon$ does not have any finite bar,
the conclusion follows.

In order to simplify the proof, we will use a slightly weaker result
than the isometry theorem.
We recall that the maximal length of a finite bar $\betamax(V^t)\geq 0$ in the persistence
module $(V^t)$ can alternatively be defined by
\begin{equation*}
    \betamax(V^t) = \sup \left\{ \beta \geq 0\ |\ \exists t\in\R,\
    \ker(V^t \to V^{t+\beta}) \neq \ker(V^t \to V^{+\infty}) \right\}.
\end{equation*}

\begin{lem}\label{lem:betabound}
    Let $((V^t),\pi)$ be a persistence module and
    $((W^t),\kappa)$ be a persistence module without any finite bar.
    If there exist $\delta,\delta'\in\R$ with $\delta+\delta'\geq 0$ and
    $f:(V^t) \to (W^{t+\delta})$ and $g : (W^t)\to (V^{t+\delta'})$ that
    are morphisms of persistence modules such that
    $g_{t+\delta}f_t = \pi_t^{t+\delta+\delta'}$ for all $t\in\R$,
    then $\betamax(V^t)\leq \delta+\delta'$.
\end{lem}

\begin{proof}[Proof of Lemma~\ref{lem:betabound}]
    The proof can be summed up by the following diagram:
\begin{equation*}
    \begin{gathered}
        \xymatrix{
            V^t \ar[rr] \ar[rd]^-{f_t}
            && V^{t+\delta+\delta'} \ar[r] &
            V^s \ar[rd]^-{f_s}\\
            & W^{t+\delta} \ar[ur]^-{g_{t+\delta}} \ar[rrr]
            && 
            & W^{s+\delta}
        }
    \end{gathered},
\end{equation*}
    (we do not assume that the right ``square'' is commutative).
    By contradiction, let us assume that
    there exists $t\in\R$ and
    $v\in V^t$ such that $\pi_t^{t+\delta+\delta'}v\neq 0$ and
    $\pi_t^{+\infty}v = 0$.
    Since $\pi_t^{+\infty}v = 0$, there exists $s\geq t+\delta+\delta'$ such that
    $\pi_t^sv=0$.
    By hypothesis, $\pi_t^{t+\delta+\delta'} = g_{t+\delta}f_t$
    so $w:=f_tv \neq 0$.
    Since $(W^t)$ does not have any finite bars,
    $\kappa_{t+\delta}^sw \neq 0$.
    Since $f$ is a morphism of persistence modules,
    $f_s\pi_t^s = 
    \kappa_{t+\delta}^{s+\delta}f_t$
    so $f_s\pi_t^sv
    = \kappa_{t+\delta}^{s+\delta}w \neq 0$.
    A contradiction with $\pi_t^sv = 0$.
\end{proof}

\begin{proof}[Proof of Theorem~\ref{thm:betamax}]
    Let $\varepsilon>0$ and 
let $\eta := c_d(\bsigma) +\varepsilon/2$ and
$\eta' := c_d(\bsigma^{-1}) +\varepsilon/2$.
Let $m>\max(|\eta|,|\eta'|,|c_d(\bsigma)|,|c_d(\bsigma^{-1})|)$.
One can assume that the size $n$ of $\bsigma$ satisfies
$2n = mn_0$ for some integer
$m>\max(|\eta|,|\eta'|,|c_d(\bsigma)|,|c_d(\bsigma^{-1})|)$,
by concatenation of $\bsigma$
with some $\bepsilon^k$.
Let $b_d\in HZ_{2d}(\bsigma_{m,\eta})$ 
and $b'_d\in HZ_{2d}(\bsigma^{-1}_{m,\eta'})$ be classes associated with
the spectral class $\alpha_d$ in the sense of Corollary~\ref{cor:spectralinverse}
(well defined by definition of $\eta$, $\eta'$ and $m$).
In order to simplify the exposition, let us set
$a_d := a_d^{(2n+1)}(0)\in HZ_{2d}(\bepsilon^{2n+1})$
and $a'_d := a_d^{(2n+1)}(\eta+\eta')\in HZ_{2d}(\bepsilon_{m,\eta+\eta'})$
(we recall that $\bepsilon^{2n+1} = \bepsilon^{mn_0+1} = \bepsilon_{m,0}$
by our assumption on $n$).
Let us denote by $\Delta_1$ the interpolation isomorphism
associated with $(\bsigma^{s-1},\bepsilon,\bsigma^{1-s})_{s\in [0,1]}$
and by $\Delta_2$ the interpolation isomorphism
associated with $(\bsigma^{1-s},\bepsilon,\bsigma^{s-1})_{s\in [0,1]}$.
Let us denote by $\widetilde{\Delta}_1$ the interpolation isomorphism
associated with $(\bsigma,\bepsilon,\bsigma^{s-1},\bepsilon,\bsigma^{1-s})_{s\in [0,1]}$
and by $\widetilde{\Delta}_2$ the one associated with
$(\bsigma^{1-s},\bepsilon,\bsigma^{s-1},\bepsilon,\bsigma)_{s\in [0,1]}$.

We define the following morphisms of persistence modules:
\begin{gather*}
    f_t:\left\{
        \begin{array}{c c c}
    G_*^{(-\infty,t)}(\bsigma) &\to&
    G_*^{(-\infty,t+\eta')}(\bepsilon^{2n+1}), \\
    \alpha & \mapsto & 
    \Delta_1 (b'_d \compii \alpha),
\end{array}\right.
 \\
    g_t:\left\{
        \begin{array}{c c c}
            G_*^{(-\infty,t)}(\bepsilon^{2n+1}) &\to&
    G_*^{(-\infty,t+\eta)}(\bsigma), \\
    \alpha &\mapsto & 
    a_d^{-1} \compii \widetilde{\Delta}_2\circ
    \widetilde{\Delta}_1^{-1}(b_d \compii \alpha).
\end{array}\right.
\end{gather*}
Indeed, these morphisms commute with inclusion morphisms
by naturality of the morphisms involved in their definitions.
Let us denote by $\pi_t^s : G^{(-\infty,t)}_*(\bsigma) \to
G_*^{(-\infty,s)}(\bsigma)$ the inclusion morphism for $t\leq s$.

In order to apply Lemma~\ref{lem:betabound}, one
just needs to show that $g_{t+\eta'}\circ f_t = \pi_t^{t+\eta+\eta'}$
for all $t\in\R$.
For all $\alpha\in G_*^{(-\infty,t)}(\bsigma)$, one has
\begin{eqnarray*}
    g_{t+\eta'}\circ f_t(\alpha)
    &=&
    a_d^{-1}\compii 
    \widetilde{\Delta} ( b_d\compii \Delta_1 ( b'_d \compii \alpha ))\\
    &=&
    a_d^{-1}\compii \widetilde{\Delta}_2
    (b_d \compii b'_d \compii \alpha)\\
    &=&
    a_d^{-1}\compii \Delta_2
    (b_d \compii b'_d) \compii \alpha\\
    &=&
    a_d^{-1}\compii a'_d \compii \alpha\\
    &=& \pi_t^{\eta+\eta'}\alpha,
\end{eqnarray*}
where the second and third equality comes from
Proposition~\ref{prop:interpolationComposition}
and the identity $a_d^{-1}\compii a'_d = \pi_t^{\eta+\eta'}$ is a direct
consequence of the commutativity of (\ref{dia:compositionInclusion}).
The conclusion follows from Lemma~\ref{lem:betabound}
since the persistence module of $\bepsilon^{2n+1}$
does not have any finite bar (see the beginning of
Section~\ref{se:propertiesG}) and
\begin{equation*}
    \eta+\eta' = c_d (\bsigma) + c_d(\bsigma^{-1}) + \varepsilon
    = c_d(\bsigma) - c_0(\bsigma) + \varepsilon
    \leq 1 + \varepsilon,
\end{equation*}
where the second equality comes from Corollary~\ref{cor:spectralinverse}.
\end{proof}

\section{Smith inequality}
\label{se:smith}

In this section, we show how the classical Smith inequality (\ref{eq:smithineq})
can be applied to the sublevel sets of generating functions to prove
inequality (\ref{eq:smithbetatot}).
\c C\. inel\. i-Ginzburg used the same kind of argument to prove a Smith inequality between
the dimension of the local homology of a Hamiltonian orbit and its $p$-iterate
for $p$ prime \cite{CG20}.

\subsection{$\Z/(p)$-action of a $p$-iterated generating function}

Let us fix a prime number $p\geq 3$.
Let us fix $t\in\R$ and study the generating function
of $e^{-2i\pi t}\Phi_1$. In order to fix the notation,
we recall that
\begin{equation*}
    F_{\bsigma_{m,t}} (\mathbf{v}) := \sum_{k=1}^{n}
        f_k\left(\frac{v_k + v_{k+1}}{2}\right) +
    \frac{1}{2}\la v_k,iv_{k+1}\ra,
\end{equation*}
where $\mathbf{v}:=(v_1,\ldots,v_n)\in (\C^{d+1})^n$
and the $f_k : \C^{d+1} \to\R$ are $S^1$-invariant and 2-homogeneous.
Thus $F_{\bsigma_{m,t}^p}:(\C^{n(d+1)})^p\to\R$ is invariant under the action of $\Z/(p)$
by cyclic permutation of coordinates generated by
\begin{equation*}
    (\mathbf{v}_1,\mathbf{v}_2,\ldots,\mathbf{v}_p)
    \mapsto
    (\mathbf{v}_p,\mathbf{v}_1,\ldots,\mathbf{v}_{p-1}),
\end{equation*}
(here $\bsigma^p_{m,t}$ means $(\bsigma_{m,t})^p$).
The induced $\widehat{F}_{\bsigma_{m,t}^p}:\CP^{pn(d+1)-1}\to\R$ is then
invariant under the $\Z/(p)$-action by permutation of homogeneous coordinates
induced by
\begin{equation*}
    [\mathbf{v}_1:\mathbf{v}_2:\cdots:\mathbf{v}_p]
    \mapsto
    [\mathbf{v}_p:\mathbf{v}_1:\cdots:\mathbf{v}_{p-1}].
\end{equation*}
Fixed points $(\CP^N)^{\Z/(p)}$ of this action are the disjoint union $\bigsqcup_q P_q$
of the $p$ following $(n(d+1)-1)$-complex projective subspaces:
\begin{equation*}
    P_q := \left\{ \big[\mathbf{v} : \zeta^q \mathbf{v} : \zeta^{2q} \mathbf{v} : \cdots
    : \zeta^{(p-1)q} \mathbf{v} \big] \ |\ [\mathbf{v}]\in \CP^{n(d+1)-1} \right\},
    \quad \zeta := e^{\frac{2i\pi}{p}}.
\end{equation*}

Using the fact that the $f_k$'s are $S^1$-invariant and 2-homogeneous,
\begin{multline*}
    \frac{1}{p}
    F_{\bsigma^p_{m,t}} (\mathbf{v},\zeta^q \mathbf{v},\ldots,\zeta^{q(p-1)} \mathbf{v}) 
    = \sum_{k=1}^{n-1} \left[
        f_k\left(\frac{v_k + v_{k+1}}{2}\right) +
    \frac{1}{2}\la v_k,iv_{k+1}\ra \right] \\
        + f_n\left(\frac{v_n + \zeta^q v_1}{2}\right) +
    \frac{1}{2}\la v_n,i\zeta^q v_1\ra.
\end{multline*}
We apply the linear change of variables $\mathbf{v}\mapsto\mathbf{u}$
given by $u_k := v_k + (-1)^k\frac{1-\zeta^q}{2} v_1$ so that
\begin{equation*}
    \left\{
        \begin{array}{c c c}
            u_1 + u_2 &=& v_1 + v_2, \\
            u_2 + u_3 &=& v_2 + v_3, \\
                      &\vdots& \\
            u_{n-1} + u_n &=& v_{n-1} + v_n, \\
            u_n + u_1 &=& v_n + \zeta^q v_1.
        \end{array}
    \right.
\end{equation*}
A direct computation gives
\begin{equation*}
    \sum_{k=1}^{n-1} 
    \la v_k,iv_{k+1}\ra  + \la v_n,i\zeta^q v_1\ra
    = \sum_{k=1}^n \la u_k,i u_{k+1} \ra - 2\tan\left(\frac{q\pi}{p}\right)\| u_1\|^2,
\end{equation*}
for all integer $q\in\left\{ \frac{1-p}{2},\ldots,\frac{p-1}{2}\right\}$,
so that
\begin{equation*}
    F_{\bsigma^p_{m,t}} (\mathbf{v},\zeta^q \mathbf{v},\ldots,\zeta^{q(p-1)} \mathbf{v})
    = p\left[F_{\bsigma_{m,t}} (\mathbf{u}) - \tan\left(\frac{q\pi}{p}\right)\| u_1\|^2
    \right]
    =: pG_{t,q}(\mathbf{u}),
\end{equation*}
for $q\in\left\{ \frac{1-p}{2},\ldots,\frac{p-1}{2}\right\}$.
This last function $G_{t,q}$ is the fiberwise sum of a generating function
of $e^{-2i\pi t}\Phi_1$ and a generating function of $e^{-2i\pi q/p}$.
We recall that in this case, $\C$-lines of critical points of
this function are in one-to-one correspondence with $\C$-lines of fixed points
of the composed diffeomorphism $e^{-2i\pi (t+q/p)}\Phi_1$
(see the paragraph surrounding Equation~(\ref{eq:fiberwise})).
Let $(f_{s,t})$ be the family of function
\begin{equation*}
    f_{s,t}(\mathbf{u}) := F_{\bsigma_{m,t+(1-s)q/p}}(\mathbf{u}) - 
    \tan\left(s\frac{q\pi}{p}\right)\|u_1\|^2,\quad
    s\in [0,1].
\end{equation*}
The function $f_{s,t}$ is the fiberwise sum of a generating function
of $e^{-2i\pi (t+(1-s)q/p)}\Phi_1$ and a generating function of
$e^{-2i\pi sq/p}$ so
$0$ is a regular value of $f_{s,t}$ if and only if
$e^{-2i\pi (t+q/p)}\Phi_1$ does not have any $\C$-line of fixed points,
that is if and only if $t+q/p$ is not an action value of $\bsigma$.
According to Proposition~\ref{prop:isotopy},
\begin{equation}\label{iso:smith}
    H_*\left( \left\{ \widehat{G}_{b,q} \leq 0\right\},
    \left\{ \widehat{G}_{a,q} \leq 0\right\}\right)
    \simeq G_{*-i_0}^{(a+q/p,b+q/p)}(\bsigma,m)
    \simeq G_{*-i_0}^{(a+q/p,b+q/p)}(\bsigma),
\end{equation}
where $-m\leq a+q/p \leq b+q/p \leq m$, $i_0$ is some integer
and $a+q/p$ and $b+q/p$ are not action values of $\bsigma$.

\subsection{Application of Smith inequality}

According to Smith inequality,
\begin{equation}\label{eq:smithineq}
    \dim H_*(X;\F_p) \geq \dim H_*(X^{\Z/(p)};\F_p),
\end{equation}
where $X$ is a locally compact space or pair such that $H_*(X;\F_p)$ is
finitely generated, a space
on which acts the group $\Z/(p)$ (see for instance
\cite[Chapter IV, \S4.1]{Bor60}).
Here $\dim H_*$ means the total dimension $\sum_k \dim H_k$.

\begin{prop}\label{prop:smith}
    Given any tuple $\bsigma$ of small $\C$-equivariant Hamiltonian diffeomorphisms,
    for every prime number $p$ and every $a\leq b$ such that
    $a+q/p$ and $b+q/p$ are not action values of $\bsigma$
    and $pa$ and $pb$ are not action values of $\bsigma^p$,
\begin{equation*}\label{eq:GSmith}
    \dim G_*^{(pa,pb)}(\bsigma^p;\F_p) \geq
    \sum_{(1-p)/2\leq q\leq (p-1)/2} \dim G_*^{(a+q/p,b+q/p)}(\bsigma;\F_p).
\end{equation*}
\end{prop}

\begin{proof}
    Let us assume that $p\geq 3$ and refer the reader to Section~\ref{se:case2}
    for the modifications specific to $p=2$.
    By concatenating $\bsigma$ with some $\bepsilon^k$ if needed,
    it is not difficult to find a homotopy interpolating
    $\bsigma^p_{m,t}:=(\bsigma_{m,t})^p$ and $(\bsigma^p)_{pm,pt}$.
    Therefore, the associated interpolation isomorphism gives us
    \begin{equation}\label{iso:sigmap}
        HZ_*(\bsigma^p_{m,b},\bsigma^p_{m,a})
        \simeq
        G_*^{(pa,pb)}(\bsigma^p,pm)
        \simeq G_*^{(pa,pb)}(\bsigma^p).
    \end{equation}

    Now, we apply the Smith inequality (\ref{eq:smithineq}) to the couple
    \begin{equation*}
        X := \left(\left\{ \widehat{F}_{\bsigma_b^p} \leq 0 \right\},
        \left\{ \widehat{F}_{\bsigma_a^p} \leq 0 \right\} \right).
    \end{equation*}
    According to the last section,
    \begin{equation*}
        X^{\Z/(p)} \simeq \bigsqcup_{(1-p)/2\leq q\leq (p-1)/2} 
        \left( \left\{ \widehat{G}_{b,q} \leq 0\right\},
        \left\{ \widehat{G}_{a,q} \leq 0\right\}\right).
    \end{equation*}
    Therefore, Smith inequality (\ref{eq:smithineq}),
    (\ref{iso:sigmap}) and (\ref{iso:smith}) bring the
    conclusion.
\end{proof}

\subsection{Computation of $\betatot$}
\label{se:betatot}

The arguments of this section follow the proof of
Theorem~\ref{thm: beta tot} given by Shelukhin
in Appendix~\ref{se:She} in the realm of Floer theory
that applies to every closed monotone symplectic manifold.

\begin{prop}\label{prop:betatotint}
    Let $\bsigma$ be a tuple of small $\C$-equivariant Hamiltonian
    diffeomorphisms of $\C^{d+1}\setminus 0$ with a finite
    number of associated fixed points in $\CP^d$.
    For all $a\in\R$, all integers $n\in\N^*$ and all fields $\F$,
    \begin{equation*}
        \betatot(\bsigma;\F) = \frac{1}{2} \left(
        \int_0^1 \dim G_*^{(a+t,a+t+n)}(\bsigma;\F)
    \,\ud t - n(d+1) \right).
    \end{equation*}
\end{prop}

\begin{proof}
    Let $I_1,\ldots,I_n \subset \R$ be representatives of each
    $\Z$-orbit of finite bars of the persistence module
    $(G_*^{(-\infty,t)}(\bsigma;\F))_t$ and let $J_1,\ldots,J_{d+1}\subset\R$
    be representatives of each $\Z$-orbit of infinite bars of this persistence
    module.
    Given an interval $I\subset\R$, let $\chi_I : \R\to\{0,1\}$
    denote its characteristic map.
    According to Lemma~\ref{lem:extremities}, for $a<b$ which are neither infinite
    nor action values of $\bsigma$, one has
    \begin{multline*}
        \dim G_*^{(a,b)}(\bsigma;\F) = \\
        \sum_{k\in\Z} \left[
            \sum_{r=1}^n \big| \chi_{I_r}(b+k) - \chi_{I_r}(a+k) \big|
            + 
            \sum_{s=1}^{d+1} \big| \chi_{J_s}(b+k) - \chi_{J_s}(a+k) \big|
        \right].
    \end{multline*}
    Therefore, in order to prove the statement, it is enough to prove that for all
    $I\in\{ I_1,\ldots, I_n\}$,
    \begin{equation}\label{eq:intI}
        \sum_{k\in\Z} \int_0^1 \big|\chi_I(a+t+k+n)-\chi_I(a+t+k)\big| \ud t
        = 2\length I,
    \end{equation}
    while for all $J\in \{ J_1,\ldots, J_{d+1} \}$,
    \begin{equation}\label{eq:intJ}
        \sum_{k\in\Z} \int_0^1 \big|\chi_J(a+t+k+n)-\chi_J(a+t+k)\big| \ud t
        = n.
    \end{equation}

    In order to prove (\ref{eq:intI}), let us write $I=(u,v)$.
    According to Theorem~\ref{thm:betamax},
    $\length I = v-u \leq 1$.
    If the integer parts $\lfloor u - a\rfloor$ and $\lfloor v - a\rfloor$
    are equal, then
    only the terms $k=\lfloor u - a\rfloor$ and $k=\lfloor u - a\rfloor +n$
    are non-zero, both equal to $\length I$.
    Otherwise, one must compute the four non-zero terms and gets
    \begin{equation*}
        (1 - \{ u - a \}) + \{ v-a \} + (1- \{ u-a\}) + \{ v-a \}
        = 2(v-u) = 2\length I,
    \end{equation*}
    where $\{ x\}$ denotes the fractional part $x-\lfloor x\rfloor$ for $x\in\R$
    (at the first equality, we have used $\lfloor v-a \rfloor =
    \lfloor u-a\rfloor +1$).

    Identity (\ref{eq:intJ}) is proven by a similar straightforward computation.
\end{proof}

The following corollary is a generating functions analogue of
\cite[Theorem~D]{She19} in the special case $M=\CP^d$.

\begin{cor}
    \label{cor:betatotSmith}
    For all tuples of small $\C$-equivariant Hamiltonian
    diffeomorphisms of $\C^{d+1}\setminus 0$ with a finite
    number of associated fixed points in $\CP^d$,
    for all prime numbers $p$,
    \begin{equation*}
        \betatot(\bsigma^p;\F_p) \geq p \betatot(\bsigma;\F_p).
    \end{equation*}
\end{cor}

\begin{proof}
    Let us take the integral over almost all $t\in[0,1]$ of the
    inequality stated in Proposition~\ref{eq:GSmith} for $a=t$ and $b=1+t$.
    On the left hand side,
    \begin{eqnarray*}
        \int_0^1 \dim G_*^{(pt,p+pt)}(\bsigma^p;\F_p) \ud t
        &=& \frac{1}{p} \int_0^p \dim G_*^{(s,p+s)}(\bsigma^p;\F_p) \ud s \\
        &=& \frac{1}{p} \sum_{k=0}^{p-1}
        \int_0^1 \dim G_*^{(k+s,k+s+p)}(\bsigma^p;\F_p) \ud s \\
        &=& \frac{1}{p}\sum_{k=0}^{p-1} \big[
        2\betatot(\bsigma^p;\F_p) + p(d+1) \big] \\
        &=& 2\betatot(\bsigma^p;\F_p) + p(d+1),
    \end{eqnarray*}
    where we have applied Proposition~\ref{prop:betatotint} at the
    third line.
    On the right hand side, by applying Proposition~\ref{prop:betatotint}
    once again,
    \begin{multline*}
        \sum_{(1-p)/2\leq q\leq (p-1)/2} 
        \int_0^1 \dim G_*^{(q/p,q/p+t)}(\bsigma;\F_p)
        \ud t \\
        = \sum_{(1-p)/2\leq q\leq (p-1)/2} \big[ 2\betatot(\bsigma;\F_p) + d+1 \big]
        = 2p\betatot(\bsigma;\F_p) + p(d+1).
    \end{multline*}
    Therefore,
    \begin{equation*}
        2\betatot(\bsigma^p;\F_p) + p(d+1) \geq
        2p\betatot(\bsigma;\F_p) + p(d+1)
    \end{equation*}
    and the conclusion follows.
\end{proof}

\begin{prop}\label{prop:universalcoef}
    For every tuple of small $\C$-equivariant Hamiltonian
    diffeomorphisms $\bsigma$ of $\C^{d+1}\setminus 0$ with a finite
    number of associated fixed points in $\CP^d$,
    there exists an integer $N\in\N$ such that
    for all prime numbers $p\geq N$,
    \begin{equation*}
        \betatot(\bsigma;\F_p) = \betatot(\bsigma;\Q).
    \end{equation*}
\end{prop}

\begin{proof}
    According to Proposition~\ref{prop:betatotint},
    it is enough to prove that for some $N\in\N$
    every prime number such that $p\geq N$ satisfies
    \begin{equation}\label{eq:dimuniversal}
        \dim G_*^{(t,t+1)}(\bsigma;\F_p) =
        \dim G_*^{(t,t+1)}(\bsigma;\Q), \quad
        \forall t\in [0,1].
    \end{equation}
    If there is no action value of $\bsigma$ in $[a,b]$
    then $\dim G_*^{(a,a+1)}(\bsigma;\F) =
    \dim G_*^{(b,b+1)}(\bsigma;\F)$ for all field $\F$.
    Since there is a finite number of action values in $[0,1]$,
    it is enough to prove (\ref{eq:dimuniversal}) for a finite
    number of value $t$ (one in between each critical value
    of $[0,1]$).
    For each $t$ (that is not an action value),
    $G_*^{(t,t+1)}(\bsigma;\F) \simeq H_{*+i_0}(A_t,B_t;\F)$
    for a topological pair $(A_t,B_t)$ of some complex projective space
    independent of $\F$ with a finitely generated homology group with integral
    coefficients.
    According to the universal coefficient theorem, there exists $N_t\in\N$
    such that $\dim H_*(A_t,B_t;\Q) = \dim H_*(A_t,B_t;\F_p)$
    for all prime number $p\geq N_t$.
    The conclusion follows by taking the maximum among the $N_t$'s
    for our finite set of $t$'s.
\end{proof}

\subsection{The special case $p=2$}
\label{se:case2}

Here, we briefly explain how to modify the above arguments in the
special case $p=2$ -- that is only useful to prove Theorem~\ref{thm:main}
when $\F$ has characteristic 2.

In order to study the $\Z/(2)$-symmetry of a generating function associated
with $(e^{-2i\pi t}\Phi_1)^2$, one cannot take the generating function of $\bsigma^2_{m,t}$
since it is an even tuple.
Instead, we take the generating function of $(\bsigma_{m,t},\bepsilon,\bsigma_{m,t})$
which is invariant under the following action of $\Z/(2)$ written in $w$-variables:
\begin{equation*}
    (\mathbf{w}^1,w^2,\mathbf{w}^3) \mapsto
    \left(\mathbf{w}^3,-w^2+\sum_{k=1}^n(-1)^k (w^1_k + w^3_k),\mathbf{w}^1\right).
\end{equation*}
Indeed, $Q_{2n+1}$ is invariant under this action and, in $w$-variables,
\begin{equation*}
    F_{(\bsigma_{m,t},\bepsilon,\bsigma_{m,t})}(\mathbf{w^1},w^2,\mathbf{w}^3)
    = F'(\mathbf{w^1}) + F'(\mathbf{w^3}) + Q_{2n+1}(\mathbf{w}^1,w^2,\mathbf{w}^3),
\end{equation*}
where $F'$ is the direct sum of the elementary generating functions of $\bsigma_{m,t}$.
The set of fixed points of the induced action on $\CP^{(2n+1)(d+1)-1}$ is the disjoint
union of the complex projective spaces $P_0$ and $P_1$ defined by
\begin{gather*}
    P_0 := \left\{ \left[\mathbf{w}:\sum_{k=1}^n (-1)^k w_k : \mathbf{w}\right] \ |
    \ \mathbf{w} \in (\C^{d+1})^n \right\},\\
    P_1 := \left\{ \left[\mathbf{w}:w': -\mathbf{w}\right] \ |
    \ (\mathbf{w},w') \in (\C^{d+1})^n\times \C^{d+1} \right\}.
\end{gather*}
The restriction of the generating function to $P_0$ gives us back the generating
function $\widehat{F}_{\bsigma_{m,t}}$ whereas, still in $w$-variables,
\begin{equation*}
    \frac{1}{2}F_{(\bsigma_{m,t},\bepsilon,\bsigma_{m,t})}(\mathbf{w},w',-\mathbf{w})
    = F_{\bsigma_{m,t}}(\mathbf{w}) + 2\la \sum_{k=1}^n (-1)^{k+1} w_k, iw' \ra,
\end{equation*}
By the change of variables $A_n\mathbf{v} =\mathbf{w}$ and $\xi = 2w'$,
one gets, in $v$-variables, the function
\begin{equation*}
    (\mathbf{v},\xi) \mapsto
    F_{\bsigma_{m,t}}(\mathbf{v}) + \la v_1, i\xi \ra
\end{equation*}
which is the fiberwise sum of a generating function of $e^{-i\pi t}\Phi_1$
with the generating function $(x;\xi)\mapsto \la x , i\xi\ra$ that generates
$-\id$.
This time, we can take $(f_{s,t})$ to be the family of function
\begin{equation*}
    f_{s,t}(\mathbf{v},v_{n+1}) :=
    F_{\bsigma_{m,t+(1-s)/2}}(\mathbf{v}) + \sin\left(s\frac{\pi}{2}\right)
    \la v_1
    - i\cos\left(s\frac{\pi}{2}\right) \frac{\xi}{2}, i\xi\ra
    ,\ s\in[0,1],
\end{equation*}
that interpolates the latter with
$(\mathbf{v},\xi)\mapsto F_{\bsigma_{m,t+1/2}}(\mathbf{v})$.
That being said, it is not difficult to conclude.

\section{Proof of Theorem~\ref{thm:main}}
\label{se:proof}

\begin{proof}[Proof of Theorem~\ref{thm:main}]
    Let $\bsigma$ be any tuple of $\C$-equivariant Hamiltonian diffeomorphisms
    associated with $\varphi$, so that $N(\bsigma;\F) = N(\varphi;\F)$.
    Let us denote by $K(\bsigma;\F)$ the number of $\Z$-orbits of
    finite bars of the barcode associated with $\bsigma$ over the field $\F$.
    According to the universal coefficient theorem,
    one can assume that $\F=\Q$ if $\F$ has characteristic 0
    and $\F=\F_p$ if it has characteristic $p\neq 0$.

    Let us assume that $\F=\Q$.
    According to Proposition~\ref{prop:N},
    $N(\bsigma;\Q)>d+1$ implies that $K(\bsigma;\Q)>0$
    so $\betamax(\bsigma;\Q) > 0$.
    According to Corollary~\ref{cor:betatotSmith},
    for all prime number $p\geq 3$,
    \begin{equation*}
        K(\bsigma^p;\F_p)\betamax(\bsigma^p;\F_p) \geq
        \betatot(\bsigma^p;\F_p) \geq p\betatot(\bsigma;\F_p).
    \end{equation*}
    Thus, by Proposition~\ref{prop:universalcoef},
    for all sufficiently large prime $p$,
    \begin{equation*}
        K(\bsigma^p;\F_p)\betamax(\bsigma^p;\F_p) \geq
        p\betatot(\bsigma;\Q) \geq p\betamax(\bsigma;\Q),
    \end{equation*}
    that is to say that $K(\bsigma^p;\F_p)\betamax(\bsigma^p;\F_p)$
    grows at least linearly with prime numbers $p$.
    According to Theorem~\ref{thm:betamax},
    $\betamax(\bsigma^p;\F_p) \leq 1$ so
    $K(\bsigma^p;\F_p)$ must diverge to $+\infty$
    with prime numbers $p$ and so must $N(\bsigma^p;\F_p)$
    by Proposition~\ref{prop:N}.
    Let $z_1,\ldots, z_n\in\CP^d$ be the fixed points of $\varphi$.
    According to Corollary~\ref{cor:gromollmeyer}, there exists $B>0$
    such that $\dim \lochom_*(\bsigma^p;z_k;\F_p) < B$ for all $k$
    and all prime $p$.
    Let $A\in\N$ be such that for all prime $p\geq A$,
    $N(\bsigma^p;\F_p) > nB$.
    Then, for all prime $p\geq A$, there must be at least one
    fixed point of $\varphi^p$ that is not one of the $z_k$'s,
    that is there must be at least one $p$-periodic point
    that is not a fixed point.
    Hence, the conclusion for the case $\F$ of characteristic 0.

    Let us assume that $\F=\F_p$ for some prime number $p$.
    By contradiction, let us assume that $\varphi$ has only finitely
    many periodic point of period $p^k$ for some $k\in\N$.
    According to Corollary~\ref{cor:betatotSmith},
    \begin{equation*}
        \betatot(\bsigma^{p^k};\F_p) \geq p^k\betatot(\bsigma;\F_p),
        \quad \forall k\in\N,
    \end{equation*}
    in particular, $N(\bsigma^{p^k};\F_p)>d+1$ for all $k\in\N$.
    Thus, by taking a sufficiently large $p^k$-iterate of $\varphi$,
    one can assume that every periodic point of $\varphi$ of period $p^k$
    for some $k$ is an admissible fixed point of $\varphi$
    (see the end of Section~\ref{se:homologysublevel} for the definition
    of an admissible fixed point).
    According to Proposition~\ref{prop:gromollmeyer}, it implies
    that $N(\bsigma^{p^k};\F_p) = N(\bsigma;\F_p)$ for all $k\in\N$.
    But Corollary~\ref{cor:betatotSmith} together with
    Proposition~\ref{prop:N} imply that the left-hand side of this equation
    must diverge to $+\infty$ as $k$ grows,
    a contradiction.
\end{proof}

\appendix

\section{Projective join}\label{se:pj}

We describe an operation on the homology of subsets of a projective space
relating the homology of two subsets $A$ and $B$ to the homology
of their projective join $A*B$.
The analogous operation for the topological join was already defined by
Whitehead in \cite{Whi56}. Granja-Karshon-Pabiniak-Sandon already
defined a homology projective join in the real case
in \cite{GKPS} for a purpose similar to ours.
However, their direct construction seems difficult to extend in the complex
case.

Since the proof of some fundamental properties of this operation are only technical
and does not shed much light on their applications, we have put
these proofs in a specific section.
More precisely, proofs of Proposition~\ref{prop:pjuk} and \ref{prop:fpjtimesbeta}
are postponed to Section~\ref{se:pjproofs}

\subsection{Definition and properties}

Let $m,n\in\N$ and let $\pi:\C^{m+n+2}\setminus 0\to \CP^{m+n+1}$ be the
quotient map. We projectively endow $\CP^m$ and $\CP^n$ in $\CP^{m+n+1}$
by identifying $\CP^m$ with $\pi(\C^{m+1}\times 0\setminus 0)$ and $\CP^n$ with
$\pi(0\times \C^{n+1}\setminus 0)$ so that $\CP^m$ and $\CP^n$ do not
intersect.  This is equivalent to considering two projective subspaces of
respective $\C$-dimension $m$ and $n$ in general position.  Let $A\subset
\CP^m$ and $B\subset \CP^n$ be non-empty sets. Then the projective join
$A*B\subset \CP^{m+n+1}$ is the union of every projective line intersecting
$A$ and $B$.  In other words, $A*B = A\cup B\cup
\pi(\widetilde{A}\times\widetilde{B})$ where $\widetilde{A}$ and
$\widetilde{B}$ are the lifts of $A$ and $B$ to $\C^{n+1}\setminus 0$ and
$\C^{m+1}\setminus 0$ respectively.  One can remark that $\CP^m *\CP^n =
\CP^{m+n+1}$ and that if $[a:b]\in\CP^{m+n+1}$, with $a\in\C^{m+1}$ and
$b\in\C^{n+1}$, does not belong to $\CP^m$ and $\CP^n$, then only one
projective line intersecting these two subspaces contains $[a:b]$, namely the
line joining $\alpha:=[a:0]$ to $\beta:=[0:b]$ denoted by $(\alpha\beta)$.

We need to define a projective join in the level of homology which would be a map
$\pj_*:H_*(A\times B) \to H_{*+2}(A*B)$, or dually in the level of cohomology
$\pj^*:H^*(A*B)\to H^{*+2}(A\times B)$.
One can perhaps proceed directly by defining a projective join at the level of chains,
sending two chains $\alpha\in C_i(A)$ and $\beta\in C_j(B)$ to a chain
$\alpha *\beta \in C_{i+j+2}(A*B)$
that triangulates the projective join of their images.
We will proceed in an indirect way by remaining in the level of homology.

Let $E_{A,B}\subset \CP^n\times\CP^m\times\CP^{m+n+1}$ be the set
\begin{equation*}
    E_{A,B} := \left\{ (a,b,c) \in A\times B\times (A*B) \ |\ c\in
    (ab) \right\}.
\end{equation*}
Let $p_1$ and $p_2$ be the canonical projection of $\CP^n\times\CP^m\times\CP^{m+n+1}$
on the factor $\CP^n\times\CP^m$ and $\CP^{m+n+1}$ respectively.
Then $p_1|_{E_{A,B}}$ defines a $\CP^1$-fiber bundle on $A\times B$,
the fiber of any $(a,b)\in A\times B$ being $a\times b\times (ab)\simeq (ab)$.
As $\CP^1$ can be identified with the 2-sphere $\sphere{2}$,
the Gysin long exact sequence holds:
\begin{equation}\label{les:Gysin}
    \cdots \xrightarrow{\cdot\smile e} H^*(A\times B) \xrightarrow{(p_1)^*} H^*(E_{A,B})
    \xrightarrow{(p_1)_*} H^{*-2}(A\times B) \xrightarrow{\cdot\smile e} \cdots,
\end{equation}
where $e\in H^3(A\times B)$ denotes the Euler class of the $\sphere{2}$-bundle
$E_{A,B}$.

\begin{definition}
    The cohomology projective join $\pj^*:H^*(A*B)\to H^{*-2}(A\times B)$
    denotes the map $\pj^* := (p_1)_*\circ(p_2)^*$,
    where $(p_1)_*$ is defined by (\ref{les:Gysin}) and $(p_2)^*$
    is induced by $p_2:E_{A,B}\to A*B$.
    The homology projective join $\pj_* = (p_2)_*\circ (p_1)^* :
    H_*(A\times B) \to H_{*+2}(A*B)$ is
    defined dually.
\end{definition}

We extend this definition to topological couples $(A,B)\subset\CP^m$,
$(C,D)\subset\CP^n$ the following way.
Let $(A,B)*(C,D):=(A*C,A*D\cup B*C)$,
the map $p_1$ defines a relative $\CP^1$-fiber bundle $(E_{A,C},E_{A,D}\cup E_{B,C})$
on $(A,B)\times (C,D)$ while $p_2$ maps this bundle on $(A,B)*(C,D)$.
Hence, one can set $\pj^* := (p_1)_*\circ(p_2)^*$ and $\pj_*:=(p_2)_*\circ (p_1)^*$ as before.
By naturality of the maps induced by $p_1$ and $p_2$,
this extension is natural: projective join commutes with long exact sequences of
topological pairs or triples.

These maps are also natural in the following way: let $A,C\subset\CP^m$
and $B,D\subset\CP^n$ and assume that $f:A*B\to C*D$ is the restriction of
a projective map satisfying $f|_{A}:A\to C$ and $f|_{B}:B\to D$,
then the $\CP^1$-fiber bundle $E_{A,B}$
is the pull-back of $E_{C,D}$ by $f|_A\times f|_B$, so that
the following diagram commutes:
\begin{equation}\label{dia:natural}
    \begin{gathered}
    \xymatrix{
        H_*(A\times B) \ar[d]^-{(f|_A\times f|_B)_*} \ar[r]^-{\pj_*} & H_{*+2}(A*B)
        \ar[d]^-{f_*} \\
        H_*(C\times D) \ar[r]^-{\pj_*} & H_{*+2}(C*D)
    }
\end{gathered}.
\end{equation}
This statement extends to topological pairs in the obvious way.

\begin{prop}\label{prop:pjassociativity}
    The homology projective join is associative:
    given $A$, $B$ and $C$ included in $\CP^n$,
    \begin{equation*}
        \forall (\alpha,\beta,\gamma) \in
        H_*(A)\times H_*(B)\times H_*(C),\quad
        \pj_*(\pj_*(\alpha\times\beta)\times\gamma)
        = \pj_*(\alpha\times\pj_*(\beta\times\gamma)).
    \end{equation*}
\end{prop}

As $R$-algebras, one has 
\begin{equation*}
H^*(\CP^{m+n+1}) = R[u]/u^{m+n+2}
\quad \text{and} \quad
H^*(\CP^m\times\CP^n) = R[u_1,u_2]/(u_1^{m+1},u_2^{n+1})
\end{equation*}
where $u$, $u_1$ and $u_2$ restrict to orientation classes of $\CP^1$
(with $\CP^1\subset \CP^m$ for $u_1$ and $\CP^1\subset \CP^n$ for
$u_2$).

\begin{prop}\label{prop:pjuk}
    Let $\pj^*$ be the cohomology projective join on $\CP^m\times\CP^n$,
    with the above notation one has
    \begin{equation*}
        \pj^*u^k = \sum_{i+j=k-1} u_1^i u_2^j,\quad \forall k\in\N^*.
    \end{equation*}
    Dually, the homology projective join $\pj_*$ on $\CP^m\times\CP^n$
    satisfies
    \begin{equation*}
        \pj_*\left([\CP^i]\times[\CP^j]\right) = [\CP^{i+j+1}],\quad
        \forall i\in\{ 0,\ldots,m\}, \forall j\in\{ 0,\ldots,n\}.
    \end{equation*}
\end{prop}

We recall that the cohomological length $\ell(A)$ of a subspace $A\subset\CP^N$
is the rank of the morphism $H^*(\CP^N;\Z)\to H^*(A;\Z)$ induced by the inclusion
(\emph{e.g.} $\ell(\CP^n) = n+1$). This is also the rank of the morphism
$H_*(A;\Z)\to H_*(\CP^N;\Z)$.

Given two subsets $A\neq\emptyset$ and $B$ as above, let $\ell := \ell(B)$.
The restriction of $u^\ell \in H^*(\CP^{m+n+1})$ to
$H^*(B)$ is zero and is non-zero in $H^*(A*B)$.
Let $v_B\in H^*(A*B,B)$ be one of its inverse image.
We recall that there is a well-defined cap-product
\begin{equation*}
    H_k(A*B,B)\times H^l(A*B,B) \xrightarrow{\frown} H_{k-l}(A*B\setminus B)
\end{equation*}
which is defined by the following commutative diagram:
\begin{equation*}
    \begin{gathered}
            \xymatrix{
                H_*(A*B,B)\times H^*(A*B,B)
                \ar[]!<9ex,2ex>;[rd]!<0ex,-0.6ex>^-{\frown} 
                \ar@<-8ex>[d]^-{\simeq} \\
                H_*(A*B,T)\times H^*(A*B,T)
                \ar@<-8ex>[u]^-{\simeq}
                \ar@<8ex>[d]^-{\simeq} 
                & H_*(A*B\setminus B) \\
                H_*(A*B\setminus B,T\setminus B)\times
                H^*(A*B\setminus B,T\setminus B)
                \ar@<8ex>[u]^-{\simeq} 
                \ar[]!<15ex,2ex>;[ru]^-{\frown} 
        }
    \end{gathered},
\end{equation*}
where $T\subset A*B$ is the restriction of a tubular neighborhood of $\CP^n$
to $A*B$,
the bottom diagonal arrow is the usual cap-product
and vertical maps are isomorphisms induced by inclusion maps
(the isomorphisms come from retractions at the top and from excision
at the bottom).
Let $p_A : A*B\setminus B\to A$ be the map
$p_A[a:b] := [a:0]$.
Let $f:H_*(A*B)\to H_{*-2\ell}(A)$ be the map $f(\alpha):=(p_A)_*(\alpha\frown v_B)$.
These definitions extend to the case where $A$ is a topological
pair $(A_1,A_0)$ with $A_1\neq A_0$
by taking $v_B\in H^{\ell(B)}(A_1*B,B)$ and by using the cap-product
\begin{equation*}\label{eq:defcap}
    H_k(A_1*B,A_0*B)\times H^l(A_1*B,B) \xrightarrow{\frown}
    H_{k-l}(A_1*B\setminus B,A_0*B\setminus B),
\end{equation*}
defined the same way as above.

\begin{cor}[{\cite[Corollary~A.2]{Giv90}}] \label{cor:homlength}
    For all non-empty subsets $A\subset \CP^m$ and $B\subset\CP^n$,
    one has
    \begin{equation*}
        \ell (A*B) = \ell (A) + \ell (B).
    \end{equation*}
\end{cor}

\begin{proof}
    Let $\alpha\in H_{2\ell(A)-2}(A)$ and  $\beta\in H_{2\ell(B)-2}(B)$
    be classes that are sent to the class $[\CP^{\ell(A)-1}]\in H_*(\CP^m)$ and
    $[\CP^{\ell(B)-1}]\in H_*(\CP^n)$ respectively.
    According to Proposition~\ref{prop:pjuk} and naturality (\ref{dia:natural}),
    $\pj_*(\alpha\times\beta)$ is sent to $[\CP^{\ell(A)+\ell(B)-1}]$
    in $H_*(\CP^{m+n+1})$. Hence $\ell(A*B)\geq \ell(A)+\ell(B)$.
    The converse inequality comes from the commutativity of the following diagram:
    \begin{equation*}
        \begin{gathered}
            \xymatrixcolsep{4pc}
            \xymatrix{
                H_*(A*B) \ar[r]^-{f} \ar[d] & H_{*-2\ell(B)}(A) \ar[d] \\
                H_*(\CP^{m+n+1}) \ar[r]^-{\cdot\frown u^{\ell(B)}} &
                H_{*-2\ell(B)}(\CP^{m+n+1})
            }
        \end{gathered}
        .
    \end{equation*}
\end{proof}

\begin{prop}\label{prop:fpjtimesbeta}
    Let $A\subset\CP^m$, $B\subset\CP^n$ be non-empty sets and $\ell := \ell(B)$.
    Let $\beta\in H_*(B)$ be a class that is sent to
    $[\CP^{\ell-1}]\in H_*(\CP^n)$.
    The following diagram commutes:
    \begin{equation*}
        \begin{gathered}
            \xymatrix{
                H_{*+2\ell-2}(A\times B) \ar[r]^-{\pj_*} &
                H_{*+2\ell}(A*B) \ar[d]^-{f} \\
                H_*(A) \ar[u]^-{\cdot\times\beta} \ar[r]^-{\id}_-{=} &
                H_*(A)
            }
        \end{gathered},
    \end{equation*}
    where $f(\alpha) := (p_A)_*(\alpha\frown v_B)$ as defined above.
    This result also holds when $A$ is a topological pair $(A_1,A_0)$
    with $A_1\neq A_0$.
\end{prop}

\begin{cor}\label{cor:pjstabilization}
    Let $(A_1,A_0)$ be a topological pair included in $\CP^m$,
    the map $\alpha\mapsto \pj_*(\alpha\times [\CP^n])$ gives
    an isomorphism
    \begin{equation*}
        H_*(A_1,A_0) \to H_{*+2(n+1)}(A_1*\CP^n,A_0*\CP^n)
    \end{equation*}
    which is the inverse of the isomorphism $f:\alpha\mapsto
    (p_{A_1})_*\alpha\frown v_{\CP^n}$.
\end{cor}

\begin{proof}
    According to \cite[Proposition~4.1]{periodicCPd}, the inclusion morphism
    $H_*(\CP^n)\to H_*(A_i*\CP^n)$ is an isomorphism in degree $*\leq 2n+1$
    and $f_i:H_*(A_i *\CP^n)\to H_{*-2(n+1)}(A_i)$,
    $f_i(\alpha):=(p_{A_i})_*\alpha\frown v_{\CP^n}$ is an isomorphism in degree
    $*\geq 2(n+1)$, for $i\in\{0,1\}$ (it is stated in cohomology but the proof
    also holds in this dual setting).
    Using the long exact sequences of the pairs $(A_1,A_0)$ and
    $(A_1*\CP^n,A_0*\CP^n)$, we deduce that
    $f:H_*(A_1*\CP^n,A_0*\CP^n)\to H_{*-2(n+1)}(A_1,A_0)$ is an isomorphism.
    The result is now a direct consequence of Proposition~\ref{prop:fpjtimesbeta}.
\end{proof}

\subsection{Technical proofs}\label{se:pjproofs}

We will denote $E_{\CP^m,\CP^n}$ by $E_{m,n}$.
The bundle $E_{A,B}$ is the restriction of the bundle $E_{m,n}$ to $A\times B$,
hence $e$ is the pullback of the Euler class of $E_{m,n}$ which lies
in $H^3(\CP^m\times\CP^n) = 0$.
Therefore $e=0$ and (\ref{les:Gysin}) reduces to the short exact sequence
\begin{equation}\label{ses:Gysin}
    0\to H^*(A\times B) \xrightarrow{(p_1)^*} H^*(E_{A,B})
    \xrightarrow{(p_1)_*} H^{*-2}(A\times B) \to 0.
\end{equation}

\begin{proof}[Proof of Proposition~\ref{prop:pjassociativity}]
    Let us first define a projective join with 3 entries
    $\pj_*^3 : H_*(A\times B\times C) \to H_{*+4}(A*B*C)$
    then prove that
    \begin{equation}\label{eq:pj3}
        \pj_*(\pj_*(\alpha\times\beta)\times\gamma)
        = \pj_*^3(\alpha\times\beta\times\gamma)
        = \pj_*(\alpha\times\pj_*(\beta\times\gamma)).
    \end{equation}
    Given three points $a$, $b$ and $c$ of some projective space $\CP^N$
    that are projectively independent,
    we use the classical notation $(abc)\subset \CP^N$ to denote
    the complex projective plane 
    Let $E_{A,B,C} \subset (\CP^n)^3\times \CP^{3n+2}$ be the set
    \begin{equation*}
        E_{A,B,C} := \left\{ (a,b,c,z) \in A\times B\times C \times (A*B*C)
        \ |\ z\in (abc) \right\}
    \end{equation*}
    and let
    $P_1 : E_{A,B,C} \to A\times B\times C$ and
    $P_2 : E_{A,B,C} \to A*B*C$ be the associated projection maps.
    The map $P_1$ defines a $\CP^2$-fiber bundle
    which is the restriction of the fiber bundle
    $E_{\CP^n,\CP^n,\CP^n} \to (\CP^n)^3$,
    so the action of $\pi_1(A\times B\times C)$ on the homology
    group $H_*(\CP^2)$ of a fiber is the restriction of
    the action of $\pi_1((\CP^n)^3)=0$ \emph{i.e.} trivial.
    We define $\pj_*^3$ by $\pj_*^3 := (P_2)_*\circ (P_1)^*$
    where $(P_1)^* : H_*(A\times B\times C) \to H_{*+4}(E_{A,B,C})$
    denotes the morphism dual to the integration
    along the fiber of the fibration $P_1$
    (the complex structure of $\CP^2$ gives a natural
    identification $H_4(\CP^2)\simeq R$).
    We refer to Section~\ref{se:fiberintegration} for
    the definition and properties of this morphism.

    In order to prove (\ref{eq:pj3}), let us introduce the set
    \begin{equation*}
        E_{(A,B),C} := \{ (a,b,c,x,z)\in A\times B\times C\times (A*B)\times
        (A*B*C)\ |\ x\in (ab) \text{ and } z\in (xc) \}
    \end{equation*}
    with the projection maps $P'_2 : E_{(A,B),C} \to A*B*C$,
    $P'_1 : E_{(A,B),C} \to E_{A,B}\times C$ sending $(a,b,c,x,z)$
    to $(a,b,x;c)$, $\tilde{f} : E_{(A,B),C} \to E_{A*B,C}$
    sending $(a,b,c,x,z)$ to $(x,c,z)$ and
    $g : E_{(A,B),C} \to E_{A,B,C}$ sending
    $(a,b,c,x,z)$ to $(a,b,c,z)$.
    The map $P'_1$ is a $\CP^1$-fiber bundle, in fact
    $\tilde{f}$ is a morphism of fiber bundle
    with base-space morphism $f:= p_2\times \id_C$.
    In order to summarize the situation, we have the following commutative
    diagram:
    \begin{equation}\label{dia:pj3}
        \begin{gathered}
            \xymatrix{
            &A*B*C \\
            E_{A,B,C} \ar[ddr]_-{P_1} \ar[ur]^-{P_2}
            &E_{(A,B),C} \ar[r]^-{\tilde{f}} \ar[d]^-{P'_1} \ar[u]^-{P'_2}
            \ar[l]^-{g}
            & E_{A*B,C} \ar[d]^-{p'_1} \ar[ul]_-{p'_2} \\
            &E_{A,B} \times C \ar[r]^-{f} \ar[d]^-{p_1\times\id_C}
            & (A*B) \times C \\
            &A\times B\times C
            }
        \end{gathered}.
    \end{equation}
    According to the naturality of the integration along the fiber
    (\ref{dia:naturalfiberintegration}), it follows
    that $(p'_2)_*(p'_1)^*f_*(p_1\times\id_C)^* = (P'_2)_*(P'_1)^*(p_1\times\id_C)^*$,
    which means that
    \begin{equation}\label{eq:pjpj}
        \pj_*(\pj_*(\alpha\times\beta)\times\gamma)
        = (P'_2)_*(P'_1)^*(p_1\times\id_C)^*(\alpha\times\beta\times\gamma).
    \end{equation}
    The map $g$ commutes with the Serre fibration $q:=(p_1\times\id_C)\circ P'_1$
    and the fiber bundle $P_1$.
    Let us fix a base-point $(a,b,c)\in A\times B\times C$,
    of fiber $F=a\times b\times c \times (abc) \simeq (abc)$ for $P_1$
    and of fiber $F'\simeq \{ (x,z) \ |\ x\in (ab) \text{ and } z\in (xc) \}$
    for $q$.
    According to Proposition~\ref{prop:compositionfiberintegration}
    and the remark after it,
    in order to show that the following diagram commutes:
    \begin{equation}\label{dia:g}
        \begin{gathered}
            \xymatrix{
                H_{*+4}(E_{A,B,C})
            & H_{*+4}(E_{(A,B),C}) \ar[l]^-{g_*} \\
            & H_*(A\times B\times C) \ar[lu]^-{P_1^*} \ar[u]^-{q^*}
            }
        \end{gathered}
    \end{equation}
    (with coefficients of every $H_*$ in the same ring $R$),
    one must prove that
    $g_* : H_4(F') \to H_4(F)$ commutes with the identity of $R$
    under the isomorphisms $H_4(F)\simeq R$ and $H_4(F')\simeq R$
    given by the local complex orientation.
    This comes from the fact that the quotient space $F'/((ab)\times c )$
    is canonically homeomorphic to $(abc)\simeq F$
    (in particular, preserving the orientation),
    the homeomorphism being induced by $g|_{F'}$.
    The long exact sequence of the couple $(F',(ab)\times c)$
    concludes.
    Therefore, diagram (\ref{dia:g}) commutes. 
    Thus, according to
    the left hand side of the diagram (\ref{dia:pj3}) together
    with the composition property $q^* = (P'_1)^*(p_1\times \id_C)^*$,
    one has
    $(P'_2)_*(P'_1)^*(p_1\times \id_C)^* = (P_2)_*(P_1)^*$
    and (\ref{eq:pjpj}) gives the first equality of (\ref{eq:pj3}).

    The second equality is proven in a symmetric way.
\end{proof}

\begin{proof}[Proof of Proposition~\ref{prop:pjuk}]
For now, let us work on $E_{m,n}$.
First, let us see that $\pj^*u=1$.
By naturality (\ref{dia:natural}), it boils down to
showing that $\pj^*:H^2((ab)) \to H^0(a\times b)$
maps the restriction of $u$ to $(ab)$ to $1\in H^0(a\times b)$ for all
$(a,b)\in\CP^m\times\CP^n$.
Now $E_{a,b} = a\times b\times (ab)$ so that $(p_2)^*$ is an isomorphism
sending the orientation class of $(ab)$ to the orientation class of $E_{a,b}$.
According to (\ref{ses:Gysin}), $(p_1)_*$ is also an isomorphism
(preserving the orientation), hence the result.

Let $u_0:= (p_2)^*u \in H^2(E_{m,n})$.
We must now study $(p_1)_*u_0^k$ for $k\in\N^*$.

Let $T\subset \CP^{m+n+1}$ be a tubular neighborhood of $\CP^n$ that is a deformation
retract. Its pullback $p_2^{-1}(T)$ is a deformation retract of $p_2^{-1}(\CP^n)$,
hence
\begin{multline}\label{eq:excision}
    H^*\left(E_{m,n},p_2^{-1}(\CP^n)\right) \simeq H^*\left(E_{m,n},p_2^{-1}(T)\right)\\
    \simeq H^*\left(E_{m,n}\setminus p_2^{-1}(\CP^n), p_2^{-1}(T\setminus\CP^n)\right),
\end{multline}
where the isomorphisms are induced by inclusion, the second one coming from excision.
The space $E_{m,n}\setminus p_2^{-1}(\CP^n)$ is the $\CP^1$-fiber bundle $E_{m,n}$
with one global section taken away, so it is a $\C$-fiber bundle.
Let $t\in H^2(E_{m,n},p_2^{-1}(\CP^n))$ be its Thom class
(under the natural identification given by (\ref{eq:excision})).
Therefore, according to Thom isomorphism theorem, the map
\begin{equation}\label{eq:Thom}
    H^*(\CP^m\times\CP^n) \to H^{*+2}(E_{m,n},p_2^{-1}(\CP^n)),
    \quad \alpha \mapsto t\smile (p_1)^*\alpha
\end{equation}
is an isomorphism.
By looking at restrictions to $E_{a,b}$'s, we see that
$t$ is non-zero on $H^*(E_{m,n})$ and is sent to $1\in H^0(\CP^m\times\CP^n)$
by $(p_1)_*$.
According to (\ref{ses:Gysin}), on $H^*(E_{m,n})$ we have $u_0-t=(p_1)^*v$
for some $v\in H^2(\CP^m\times\CP^n)$.
In order to find $v$, we consider the following commutative diagram:
\begin{equation}\label{dia:relative}
    \begin{gathered}
        \xymatrix{
            & & H^*(\CP^m\times\CP^n) \ar[ld]^-{(p_1)^*} \ar[d]^-{(p_1)^*} \\
            H^*\left(E_{m,n},p_2^{-1}(\CP^n)\right) \ar[r] & H^*(E_{m,n}) \ar[r]
            & H^*\left(p_2^{-1}(\CP^n)\right)\\
            H^*(\CP^{m+n+1},\CP^n) \ar[u]^-{(p_2)^*} \ar[r] & 
            H^*(\CP^{m+n+1}) \ar[u]^-{(p_2)^*} \ar[r] &
            H^*(\CP^n) \ar[u]^-{(p_2)^*} 
        }
    \end{gathered},
\end{equation}
where horizontal arrows are induced by inclusion and form exact sequences.
The restriction of $p_1$ to $p_2^{-1}(\CP^n)$ induces a homeomorphism
$\CP^m\times\CP^n \simeq p_2^{-1}(\CP^n)$.
Under this identification,
the restriction of $p_2$ to $p_2^{-1}(\CP^n)$ is the projection on
the second factor $\CP^n$.
Hence, the right-hand side vertical arrow $(p_2)^*$ of (\ref{dia:relative})
sends $u^k$ to $(p_1)^*u_2^k$.
In particular, by commutativity of (\ref{dia:relative}),
$u_0\in H^2(E_{m,n})$ is sent to $(p_1)^*u_2$.
Since $t\in H^2(E_{m,n})$ is in the image of the top left arrow,
it is sent to $0$ in $H^2(p_2^{-1}(\CP^n))$ by exactness.
Thus $u_0-t\in H^2(E_{m,n})$ is sent to $(p_1)^*u_2$ whereas
$(p_1)^*v\in H^2(E_{m,n})$ is sent to $(p_1)^*v\in H^2(p_2^{-1}(\CP^n))$
by commutativity of the up right triangle (with a slight abuse of notation).
Therefore $v=u_2$.

In order to study the powers of $u_0$, we now study the powers of $t\in H^2(E_{m,n})$.
Seen in $H^4(E_{m,n},p_2^{-1}(\CP^n))$, $t^2 = t\smile (p_1)^* (\lambda u_1 + \mu u_2)$
for some $\lambda,\mu\in\Z$ according to Thom isomorphism (\ref{eq:Thom}).
Let us first find the value of $\lambda$ by restricting the complex line bundle
$E_{m,n}\setminus p_2^{-1}(\CP^n)$ to the base space $\CP^m\times \CP^0$.
This complex line bundle is $E_{m,0}\setminus p_2^{-1}(\CP^0)$ and its Thom class $t'$
is the restriction of $t$ to 
\begin{equation*}
    H^2\left(E_{m,0},p_2^{-1}(\CP^0)\right) \simeq H^2 \left( E_{m,0} /
    p_2^{-1}(\CP^0) \right),
\end{equation*}
so that $t'^2 = \lambda t'\smile (p_1)^*u_1$.
Since $\CP^0$ is just a point, $p_2$ factors in a homeomorphism between
$E_{m,0} /p_2^{-1}(\CP^0)$ and $\CP^m * \CP^0$.
Thus $p_2$ induces an isomorphism of $\Z$-algebras
\begin{equation}\label{iso:Em0}
    H^*(\CP^m *\CP^0,\CP^0) \xrightarrow{\simeq} H^*(E_{m,0},p_2^{-1}(\CP^0)).
\end{equation}
According to the long exact sequence of the couple $(\CP^m*\CP^0,\CP^0)$,
$H^2(\CP^m*\CP^0,\CP^0)\simeq H^2(\CP^m*\CP^0)$ so that the generator
$u\in H^2(\CP^m*\CP^0)$ can naturally be seen in $H^2(\CP^m*\CP^0,\CP^0)$.
Under the isomorphism (\ref{iso:Em0}), $u$ is mapped to $t'$
so that $u^2$ is mapped to $t'^2$.
Thus $t'^2$ must be a generator of $H^4(E_{m,0},p_2^{-1}(\CP^0))$,
hence $\lambda = \pm 1$. By applying the orientation preserving
morphism $(p_1)_*$, we see that $\lambda = 1$.

Now, since $u_0 = t+(p_1)^*u_2$, one has
\begin{equation*}
    u_0^2 = t\smile (p_1)^*(u_1+(\mu + 2)u_2) + (p_1)^* u_2^2
\end{equation*}
hence $(p_1)_* u_0^2 = u_1 + (\mu +2)u_2$.
By symmetry of $(p_1)_* u_0^2$ in $u_1$ and $u_2$,
$\mu$ must be $-1$. Indeed, the above identity is
still true by restricting ourselves to $E_{1,1}$
and must be invariant under the map induced by $(a,b,c)\mapsto (b,a,c)$
that swaps $u_1$ and $u_2$.
Since $u_0 = t+(p_1)^*u_2$, one has that
\begin{equation*}
    u_0^k = \sum_{i+j=k} \binom{k}{i} t^i \smile (p_1)^* u_2^j.
\end{equation*}
Using $t^i = t\smile (p_1)^*(u_1-u_2)^{i-1}$ for $i\in\N$
and $(p_1)_*(t\smile (p_1)^*w)=w$, one finally gets
\begin{equation}\label{eq:p1u0k}
    (p_1)_* u_0^k = \sum_{i+j=k} \binom{k}{i} (u_1-u_2)^{i-1} u_2^j
    = \sum_{i+j=k-1} u_1^i u_2^j,
\end{equation}
the last equality can be obtained by identification of coefficients
of the polynomial expression in $u_1$ and $u_2$.
\end{proof}

\begin{proof}[Proof of Proposition~\ref{prop:fpjtimesbeta}]
    We first remark that there is a well-defined cap-product
    $H_*(E_{A,B})\times H^*(E_{A,B},p_2^{-1}(B))\to H_*(E_{A,B}\setminus p_2^{-1}(B))$
    compatible with the one defines in $A*B$ through the map $p_2$
    and defined the same way.
    This is summed up by saying that the left hand side of the following
    diagram is ``commutative'':
    \begin{equation}\label{dia:capEAB}
    \begin{gathered}
        \xymatrix{
            H_*(E_{A,B})\times H^*(E_{A,B},p_2^{-1}(B)) \ar[r]^-{\frown}
            \ar@<-8ex>[d]^-{(p_2)_*} 
            &
            H_*(E_{A,B}\setminus p_2^{-1}(B)) \ar[d]^-{(p_2)_*}
            \ar[r]^-{(p_1)_*}_-{\simeq} &
            H_*(A\times B) \ar[d]^-{(\textup{pr}_1)_*} \\
            H_*(A*B)\times H^*(A*B,B) \ar[r]^-{\frown} \ar@<-8ex>[u]^-{(p_2)^*} &
            H_*(A*B\setminus B)
            \ar[r]^-{(p_A)_*}_-{\simeq} &
            H_*(A) 
        }
    \end{gathered}.
\end{equation}
In this diagram, $\textup{pr}_1:A\times B\to A$ is the projection on the first factor.
The commutativity of the right hand side of the diagram comes from the
obvious commutativity
of the associated continuous maps: $p_A\circ p_2 = \textup{pr}_1\circ p_1$ on
$E_{A,B}\setminus p_2^{-1}(B)$.
Let $v'_B := (p_2)^*v_B \in H^{2\ell}(E_{A,B},p_2^{-1}(B))$.
The space $E_{A,B}\setminus p_2^{-1}(B)$ is the restriction to $A\times B$
of the $\C$-fiber bundle $E_{m,n}\setminus p_2^{-1}(\CP^n)$, so that
the following map
is an isomorphism for the same reason the map (\ref{eq:Thom}) was:
\begin{equation*}
    H^*(A\times B) \to H^{*+2}(E_{A,B},p_2^{-1}(B)),
    \quad w \mapsto t'\smile (p_1)^*w,
\end{equation*}
where $t'\in H^2(E_{A,B},p_2^{-1}(B))$ is the restriction of
the class $t$ in (\ref{eq:Thom}).
Therefore, $v'_B = t'\smile (p_1)^*w$ for some $w\in H^{2\ell-2}(A\times B)$
satisfying $w=(p_1)_*(v'_B)$ (we recall that $(p_1)_*t'=1$).
Seen in $H^{2\ell}(A*B)$, the class $v_B$ is the restriction of
$u^\ell$, so that, seen in $H^{2\ell}(E_{A,B})$, $v_B'$ is the
restriction of $u_0^\ell$. According to the identity (\ref{eq:p1u0k}),
one has
\begin{equation}\label{eq:p1vB}
    (p_1)_* v_B' = w = \sum_{i+j=\ell-1} u_1^i u_2^j,
\end{equation}
identifying $u_1$ and $u_2$ with their restrictions to
$H^*(A\times B)$ by a slight abuse of notation.
We can now compute, for all $\alpha\in H_*(A)$,
\begin{eqnarray*}
    (p_A)_*(\pj_*(\alpha\times\beta)\frown v_B) 
    &=&
    (p_A)_*\circ(p_2)_* ((p_1)^*(\alpha\times\beta)\frown v_B') \\
    &=&
    (\textup{pr}_1)_*\circ(p_1)_* ((p_1)^*(\alpha\times \beta) \frown v_B')\\
    &=&
    (\textup{pr}_1)_*\left((\alpha\times \beta)\frown \sum_{i+j=\ell-1} u_1^i
    u_2^j\right) \\
    &=& \sum_{i+j=\ell-1}
    (\textup{pr}_1)_*\left((\alpha\frown u^i)\times (\beta\frown u^j)\right) \\
    &=& \la u^{\ell-1},\beta\ra \alpha \frown u^0 = \alpha.
\end{eqnarray*}
The second equality follows from commutativity of the diagram (\ref{dia:capEAB}),
the third uses (\ref{eq:p1vB}) together with the projection formula 
$p_*(p^*\gamma\frown w) = \gamma\frown p_* w$
where $p$ is a sphere bundle.
By grading issues, only the indices $(i,j)=(0,\ell-1)$ contribute to the sum
and, by definition of $\beta$, $\la u^{\ell-1},\beta \ra = 1$.
The result of this computation is the statement we wanted to prove.
\end{proof}

\subsection{Integration along the fiber}
\label{se:fiberintegration}

In this section, we recall some well known properties of the
morphism of integration along fiber
(see \cite[Section~A.2]{CM12} and references therein).
Since we cannot find any reference concerning the proof of the composition
property, we give a proof of this key property used to prove
Proposition~\ref{prop:pjassociativity}.

Let $G$ be a group. Throughout this section,
$H_*(X)$ and $H^*(X)$ will denote respectively the singular homology
and the singular cohomology of the topological space or pair $X$
with coefficients in $G$.
If we want to put another group of coefficients $G'$,
we will write down explicitly $H_*(X;G')$ or $H^*(X;G')$
so that for instance $H_*(X;H_d(Y)) = H_*(X;H_d(Y;G))$
where $Y$ is a topological space or pair.

Let us assume that $\pi:X\to B$ is a Serre fibration with fiber $F$
which has the type of a CW complex of dimension $d$
(throughout this section, we will simply write that the fibers of $\pi$
have dimension $d$).
In order to simplify the statements, we will always assume that
$\pi_1(B)$ acts trivially on $H_*(F)$.
Then the Serre spectral sequence in homology $(E^r_{p,q})$ of $\pi$ satisfies
$E^2_{p,q} \simeq H_p(B;H_q(F)) = 0$ for $q>d$.
Hence, it gives natural morphisms
$E^2_{p,d} \to E^\infty_{p,d} \to H_{p+d}(X)$
for all $p$.
Let $\pi^* : H_*(B;H_q(F)) \to H_{*+d}(X)$ denote the composition of
the Serre isomorphism $H_*(B;H_d(F)) \simeq E^2_{*,d}$ with this map
$E^2_{*,d}\to H_{*+d}(X)$.
Dually, one can define a morphism $\pi_* : H^*(X) \to H^{*-d}(B;H^d(F))$.
In de Rham cohomology, the map $\pi_*$ can be easily defined on compact
smooth fiber bundles as induced by the integration of differential forms
along the fibers.

In the special case $F\simeq \sphere{n}$, the map $\pi_*$ corresponds
to the Gysin morphism.
This definition extends directly to relative fibrations
$\pi:(X,X')\to (B,B')$ and the induced maps commute with the long exact sequences
of pair and triple by naturality of the Serre spectral sequence.
Given a commuting square of (possibly relative) fibrations
\begin{equation*}
    \begin{gathered}
        \xymatrix{
            X \ar[d]^-{p}_-{F} \ar[r]^-{\tilde{f}} & Y \ar[d]^-{q}_-{F'}\\
            B \ar[r]^-{f} & C
        }
    \end{gathered},
\end{equation*}
with fibers $F$ and $F'$ of the same dimension $d$,
the naturality of the Serre spectral sequence induces the commutative square
\begin{equation}\label{dia:naturalfiberintegration}
    \begin{gathered}
        \xymatrix{
            H_{*+d}(X) \ar[r]^-{\tilde{f}_*} & H_{*+d}(Y)\\
            H_*(B;H_d(F)) \ar[u]^-{p^*} \ar[r]^-{f_*} &
            H_*(B;H_d(F')) \ar[u]^-{q^*}
        }
    \end{gathered},
\end{equation}
where $f_*$ sends a class $[\sigma\otimes h]$,
$\sigma\in C_*(B;\Z)$ and $h\in H_d(F)$,
on $[f_*\sigma \otimes \tilde{f}_* h ]$.

Finally, it satisfies the following composition property.
Let $\pi_1 : Y \to X$ and $\pi_2 : X \to B$ be (possibly relative) Serre fibrations
of respective fibers $F_1$ and $F_2$ of dimension $d_1$ and $d_2$.
Then $\pi := \pi_2 \circ\pi_1$ is a fibration whose fiber $F$
is a fibration over $F_2$ with fibers $F_1$.
According to the Serre spectral sequence, it has dimension $d_1+d_2$
and $H_{d_1+d_2}(F)$ is naturally isomorphic to $H_{d_2}(F_2;H_{d_1}(F_1))$.
\begin{prop}\label{prop:compositionfiberintegration}
    The following diagram commutes
\begin{equation*}
    \begin{gathered}
        \xymatrix{
            H_{*+d_2}(X;H_{d_1}(F_1)) \ar[r]^-{\pi_1^*}
            & H_{*+d_1+d_2}(Y) \\
            H_*(B;H_{d_2}(F_2;H_{d_1}(F_1))) \ar[u]^-{\pi_2^*} \ar[r]^-{\simeq}
            & H_*(B;H_{d_1+d_2}(F)) \ar[u]^-{\pi^*}
        }
    \end{gathered}
\end{equation*}
where the bottom isomorphism is induced by the isomorphism
$H_{d_2}(F_2 ; H_{d_1}(F_1) ) \simeq H_{d_1+d_2}(F)$ between the groups
of coefficients.
\end{prop}

In our article, the fibers $F_1$, $F_2$ and $F$ are naturally oriented
so that there are canonical isomorphisms between their top degree homology
groups and the coefficient group and one does not have to bother with
changes of coefficient group.
However, if one wants to avoid the change of coefficients in the naturality
statement (\ref{dia:naturalfiberintegration}) when
$H_d(F) \simeq G$ and $H_d(F') \simeq G$,
the map $\tilde{f}_* : H_d(F) \to H_d(F')$ must send preferred generator
to preferred generator.

\begin{proof}[Proof of Proposition~\ref{prop:compositionfiberintegration}]
    Without loss of generality, one can assume that $B$ and $X$ are actual
    CW complexes and that $\pi_2$ is a locally trivial cellular fibration
    \cite{bar87}.
    We denote by $E$, $E_1$ and $E_2$ Serre spectral sequences of
    $\pi$, $\pi_1$ and $\pi_2$ respectively, $E_2$ having $H_{d_1}(F_1)$
    coefficients.
    Let $B^p$ and $X^p$ denote respectively the $p$-skeleton of $B$ and
    the $p$-skeleton of $X$.
    Let $X_p := \pi_2^{-1}(B^p)$, $Y_p := \pi^{-1}(B^p)$ and
    $Y_{1;p} := \pi_1^{-1}(X^p)$ denote the filtration of the spaces
    $X$ and $Y$ associated with $E_2$, $E$ and $E_1$ respectively.
    Therefore, for instance the first page of $E_2$ is given by
    $E^1_{2;p,q} := H_{p+q}(X_p,X_{p-1})$.

    Since $\pi_2$ is a cellular fibration with fibers of dimension $d_2$,
    $X_p = \pi_2^{-1}(B^p)$ is included in $X^{p+d_2}$.
    Hence, $Y_p \subset Y_{1;p+d_2}$ and this inclusion between filtrations
    induces a morphism of spectral sequences (with a shift in degree)
    $E^r_{p,q} \to E^r_{1;p+d_2,q-d_2}$.
    Therefore, one gets the following commutative diagram
    \begin{equation*}
        \begin{gathered}
            \xymatrix{
                & H_{*+d_1+d_2}(Y) \ar@{=}[r] & H_{*+d_1+d_2}(Y) \\
                & E^\infty_{*,d_1+d_2} \ar[r] \ar[u]
                & E^\infty_{1;*+d_2,d_1} \ar[u] \\
                H_*(B;H_{d_1+d_2}(F)) \ar[uur]^-{\pi^*}
                & E^2_{*,d_1+d_2} \ar[l]^-{\simeq} \ar[r] \ar[u]
                & E^2_{1;*+d_2,d_1} \ar[r]^-{\simeq} \ar[u]
                & H_{*+d_2}(X;H_{d_1}(F_1)) \ar[luu]_-{\pi_1^*}
            }
        \end{gathered},
    \end{equation*}
    where both the left and the right ``squares'' are the ones defining morphisms
    $\pi^*$ and $\pi_1^*$.
    The bottom row allows us to define a morphism
    $f:H_*(B;H_{d_1+d_2}(F)) \to H_{*+d_2}(X;H_{d_1}(F_1))$.
    According to the last diagram, it is enough to prove that $f=\pi_2^*$
    under the identification $H_{d_1+d_2}(F) \simeq H_{d_2}(F_2;H_{d_1}(F_1))$
    to conclude.

    We recall that the Serre isomorphism $E^2_{*,d_1+d_2} \simeq H_*(B;H_{d_1+d_2}(F))$
    is induced by a chain isomorphism between the chain complex
    $E^1_{*,d_1+d_2} = H_{*+d_1+d_2}(Y_*,Y_{*-1})$
    and the chain complex of the cellular filtration $(B^p)$ of $B$
    that is $H_*(B^*,B^{*-1};H_{d_1+d_2}(F))$.
    We denote by $\Psi : H_{*+d_1+d_2}(Y_*,Y_{*-1}) \to H_*(B^*,B^{*-1};H_{d_1+d_2}(F))$
    the chain isomorphism associated with $\pi$
    and by $\Psi_1$ and $\Psi_2$ the chain isomorphisms associated with $\pi_1$
    and $\pi_2$ respectively.
    Since these chain isomorphisms are natural (this is included in the proof of
    the naturality of the Serre spectral sequence),
    one gets the following commutative diagram of chain complexes:
    \begin{equation*}
        \begin{gathered}
            \xymatrix{
                H_{*+d_1+d_2}(Y_* , Y_{*-1}) \ar[r] \ar[d]^-{\Psi}_-{\simeq}
            & H_{*+d_1+d_2}(Y_{1;*+d_2},Y_{1;*+d_2-1}) \ar[d]^-{\Psi_1}_-{\simeq} \\
            H_*(B^*,B^{*-1};\Z)\otimes H_{d_1+d_2}(F) \ar[d]^-{\id\otimes u}_-{\simeq}
            & H_{*+d_2}(X^{*+d_2},X^{*+d_2-1};\Z)\otimes H_{d_1}(F_1)\\
            H_*(B^*,B^{*-1};\Z)\otimes H_{d_2}(F_2;H_{d_1}(F_1))
            & H_{*+d_2}(X_*,X_{*-1})\otimes H_{d_1}(F_1) \ar[u]
            \ar[l]^-{\Psi_2}_-{\simeq}
            }
        \end{gathered},
    \end{equation*}
    where $u$ denotes the natural isomorphism $H_{d_1+d_2}(F)\to H_{d_2}(F_2;H_{d_1}(F_1))$.
    By passing to homology, one gets the following commutative diagram:
\begin{equation*}
        \begin{gathered}
            \xymatrix{
                E^2_{*,d_1+d_2} \ar[r] \ar[d]^-{\Psi_*}_-{\simeq}
            & E^2_{1;*+d_2,d_1} \ar[d]^-{(\Psi_1)_*}_-{\simeq} \\
            H_*(B;H_{d_1+d_2}(F)) \ar[d]^-{(\id\otimes u)_*}_-{\simeq}
            \ar[r]^-{f}
            & H_{*+d_2}(X;H_{d_1}(F_1)) \\
            H_*(B;H_{d_2}(F_2;H_{d_1}(F_1)))
            & E^2_{2;*,d_2} \ar[u] \ar[r]
            \ar[l]^-{(\Psi_2)_*}_-{\simeq}
            & E^\infty_{2;*,d_2} \ar[lu]
            }
        \end{gathered},
    \end{equation*}
    where only the commutativity of the triangle is not a direct consequence
    of the previous diagram.
    The commutativity of the triangle is a consequence of the naturality
    of the morphism induced by the inclusion of filtrations
    $X_p \subset X^{p+d_2}$ between the associated spectral sequences.
    Indeed, since $(X^{p+d_2})$ is a cellular filtration,
    the associated spectral sequence whose first page
    is $(H_{p+q}(X^{p+d_2},X^{p+d_2-1}))_{p,q}$ degenerates
    at the second page.
    The bottom part of the diagram shows that $f$ is indeed
    $\pi_2^*$ under the identification induced by $u$.
\end{proof}

\section{Smith-type inequality for barcodes of Hamiltonian diffeomorphisms via
the minimal Novikov field\\ by Egor Shelukhin}
\label{se:She}

We present an alternative argument proving \cite[Theorem D]{She19}, which is the key Smith-type inequality for barcodes of Hamiltonian diffeomorphisms for the argument proving the Hofer-Zehnder conjecture for closed monotone symplectic manifolds with semi-simple even quantum homology algebra. While the approach described in \cite{She19} is more algebraically conceptual, it relies on the use of Floer complexes with coefficients in the universal Novikov ring, which are quite difficult to make sense of geometrically. The current approach, at the cost of certain more ad hoc arguments, reduces the consideration to filtered Floer homology with the standard minimal Novikov field (with quantum variable of degree twice the minimal Chern number). This choice of coefficients is more geometric, and it is this new proof that is most amenable to interpretation in the setting of generating functions. 

To set up the result and its proof we recall a few basic notions on barcodes of Hamiltonian diffeomorphisms. We follow the exposition in \cite{She19} and \cite{AtShe20}, whereto we refer for further detail and discussion of relevant notions. Suppose that $(M,\om)$ is a closed monotone symplectic manifold, that is there exists $\kappa \in \R_{>0}$ such that for each $A \in H^S_2(M;\Z),$ $\langle [\om], A \rangle = \kappa \langle c_1(TM), A \rangle,$ where $H^S_2(M;\Z)$ is the image of the Hurewicz map $\pi_2(M) \to H_2(M;\Z).$ The minimal Chern number of $(M,\om)$ is the positive generator $N > 0$ of the image of the map $c_1(TM):  H^S_2(M;\Z) \to \Z,$ given by $A \mapsto \langle c_1(TM), A \rangle.$ Set the rationality constant $\rho \in \R_{>0}$ on $(M,\om)$ to be $\rho = \kappa \cdot N.$

Fix a ground field $\bK.$ For a time-dependent Hamiltonian $H$ on a closed monotone symplectic manifold $(M,\om)$, denote by $\mrm{Spec}(H)\subset \R$ the set of all actions \[\cl{A}_H(\ol{x}) = \int_0^1 H(t,x(t)) - \int_{\ol{x}} \om,\] where $\ol{x}$ runs over contracting disks (considered up to the equivalence relation of having the same $\om$-areas) of $1$-periodic orbit $x$ of the Hamiltonian flow of $H.$ If $H$ is non-degenerate, that is its time-one map $\phi$ has its graph transverse to the diagonal in $M \times M,$ we denote by $CF(H,J)$ the Floer complex of $H$ with respect to a generic $\om$-compatible time-dependent almost complex structure $J = (J_t)_{t \in S^1}.$ This complex is filtered by a natural extension of $\cl{A}_H$ to $CF(H,J),$ in the sense that $\cl{A}_H(d_{H,J} x) < \cl{A}_H(x)$ for all non-zero $x \in CF(H,J).$ For $a,b \notin \Spec(H),$ $-\infty \leq a < b \leq +\infty$ we call the interval $I =(a,b)$ admissible, and denote by $CF(H,J)^{<a},$ $CF(H,J)^{<b}$ the subcomplexes given by elements of filtration level smaller than $a$ and $b$ respectively, and set $CF(H,J)^{I} = CF(H,J)^{<b}/CF(H,J)^{<a}$ for the quotient complex. We denote by  $HF(H)^I$ the homology of this complex, called the Floer homology of $H$ in action window $I,$ and the collection of $HF(H)^{I}$ for all admissible intervals $I,$ together with natural maps $HF(H)^{I} \to HF(H)^{I'}$ for $I = (a,b), I'=(a',b')$ with $a \leq a', b\leq b'$ induced by inclusions of complexes, is called the filtered Floer homology of $H.$ For later use, we call an admissible interval $I = (a,b)$ $p$-admissible for $p \in \Z_{>0}$ if $pa, pb \notin \Spec(H)^{+p},$ where the $p$-fold sum $A^{+p}$ of a subset $A \subset \R$ is the set of all sums $a_1 + \ldots + a_p,$ where $a_j \in A$ for all $1 \leq j \leq p.$

Note that for $I = (-\infty,\infty),$ (a suitable completion of) $CF(H,J) = CF(H,J)^{I},$ as well as $HF(H) = HF(H)^I$ can be considered to be modules over the minimal Novikov field \[ \Lambda_{\tmin} = \bK[[q^{-1},q]\] where $q$ is a formal variable of degree $2N.$ By the PSS-isomorphism \cite{PSS96} $HF(H) \cong QH(M;\Lambda_{\tmin}),$ the right-hand side denoting the quantum homology algebra of $(M,\om).$ As a $\Lambda_{\tmin}$-module, $QH(M;\Lambda_{\tmin}) = H_*(M;\bK) \otimes_{\bK} \Lambda_{\tmin},$ and the PSS-isomorphism is an isomorphism of $\Lambda_{\tmin}$-modules. We note that ignoring the grading, for each admissible interval $I,$ we have $HF(H)^{I} \cong HF(H)^{I + \rho},$ the isomorphism being given by multiplication by $q.$ We express this fact by saying that the area, or the valuation, of $q$ is $\rho.$

To a time-dependent Hamiltonian $H$ on a closed monotone symplectic manifold $(M,\om)$, the set $\mrm{Fix}_c(\phi)$ of whose contractible fixed points is finite, we can associate a $\rho$-periodic barcode, $\cl{B}(H;\bK) = \{ (I_j,m_j)\},$ where $I_j,$ called {\em bars}, are intervals of the form $(a,b],$ or $(a,\infty),$ with endpoints in $\mrm{Spec}(H)$ and $m_j \in \Z_{>0}$ are their multiplicities. The $\rho$-periodicity means that if $(I,m) \in \cl{B}(H;\bK)$ then $(I+\rho,m) \in \cl{B}(H;\bK),$ or alternatively that the group $\rho \cdot \Z$ acts on $\cl{B}(H;\bK).$ This barcode has a number of properties summarized in \cite[Proposition 22]{AtShe20}:
	
\begin{enumerate}[label = (\roman*)]
	\item For each window $J = (a,b)$ in $\R,$ with $a,b \notin \Spec(H),$ only a finite number of intervals $I$ with $(I,m) \in \cl{B}$ have endpoints in $J.$ Furthermore, \[ \dim_{\bK} HF(H)^{J} = \displaystyle\sum_{(I,m) \in \cl{B}(H),\; \# \partial I \cap J = 1} m,\] where for an interval $I = (a,b],$ $\partial I = \{a,b\},$ and for $I = (a, \infty),$ $\partial I = \{a\}.$
	\item In particular for $a \in \Spec(H),$ and $\eps > 0$ sufficiently small, so that we have $(a-\eps, a+\eps) \cap \Spec(H) = \{a\},$ \[ \dim_{\bK} HF(H)^{(a-\eps, a+\eps)} = \displaystyle\sum_{(I,m) \in \cl{B}(H),\; a \in \partial I} m,\] \[ \dim_{\bK} HF(H)^{(a-\eps, a+\eps)} = \displaystyle \sum_{ \cl{A}(\ol{x}) = a } \dim_{\bK} HF^{\loc}(H, \ol{x}).\]
	\item There are $K(\phi,\bK)$ orbits of finite bars counted with multiplicity, and $B(\bK)$ orbits of infinite bars counted with multiplicity, under the $\rho \cdot \Z$ action on $\cl{B}(H).$ These numbers satisfy: \[B(\bK) = \dim_{\bK} H_*(M;\bK)\] and \[ N(\phi,\bK) = 2 K(\phi,\bK) + B(\bK),\] where \[ N(\phi,\bK) = \sum \dim_{\bK} HF^{\loc}(\phi, x)\] is the {\em homological count of the fixed points of $\phi,$} the sum running over all the set $\fix(\phi)$ of its fixed points.
	\item There are $K(\phi,\bK)$ {\em bar-lengths} corresponding to the finite orbits, \[ 0 < \beta_1(\phi,\bK) \leq \ldots \leq \beta_{K(\phi,\bK)}(\phi,\bK),\] which depend only on $\phi.$ We call \[ \beta(\phi,\bK) = \beta_{K(\phi,\bK)}(\phi,\bK) \] the {\em boundary-depth} of $\phi,$ and \[ \beta_{\mrm{tot}}(\phi,\bK) = \sum_{1 \leq j \leq {K(\phi,\bK)}} \beta_j(\phi,\bK)\] its {\em total bar-length}.
	\item Each spectral invariant $c(\alpha, H) \in \Spec(H)$ for $\alpha \in QH_*(M) \setminus \{0\}$ is a starting point of an infinite bar in $\cl{B}(H),$ and each such starting point is given by a spectral invariant.
	\item \label{barcodes, change of H} If $H'$ is another Hamiltonian generating $\phi,$ then $\cl{B}(H') = \cl{B}(H)[c],$ for a certain constant $c \in \R,$ where $\cl{B}(H)[c] = \{ (I_i-c,m_i)\}_{i \in \cl{I}}.$ 
	\item \label{barcodes, change of coefficients} If $\mathbb{\bK}$ is a field extension of $\mathbb{F},$ and $H$ is a Hamiltonian, then $\cl{B}(H; \bK) = \cl{B}(H; \bF).$ In particular $\cl{B}(H; \bK) = \cl{B}(H; \F_p)$ if $\mrm{char}(\bK) = p,$ and $\cl{B}(H; \bK) = \cl{B}(H; \Q)$ if $\mrm{char}(\bK) = 0.$
		
	\end{enumerate}

We are now in a situation to formulate the key Smith-type inequality.

\begin{thm}[{\cite[Theorem D]{She19}}]\label{thm: beta tot}
	Let $\phi \in \ham(M,\om)$ be a Hamiltonian diffeomorphism of a closed monotone symplectic manifold $(M,\om).$ Suppose that $\fix(\phi^p)$ is finite. Then \[ p\cdot\beta_{\tot}(\phi,\F_p) \leq \beta_{\tot}(\phi^p,\F_p).\]
\end{thm}

We now give a new proof of this result, that stays within the realm of filtered Floer homology with coefficients in $\Lambda_{\tmin},$ temporarily passing through a slightly larger field.

\begin{proof}[Proof of Theorem \ref{thm: beta tot}]
	
Suppose for this proof that $p$ is odd, as the case of $p=2$ is simpler and is left to the reader.
	
We start by observing that for a $p$-admissible interval $I = (a,b),$ the $\Z/(p)$-equivariant pants product \[\cl{P}^{p\cdot I}: \wh{H}(\Z/(p), (CF(H)^{\otimes p})^{p\cdot I}) \to \wh{HF}_{eq}(H^{(p)})^{p\cdot I}\] is well-defined by arguments from \cite[Section 6]{She19}, and is an isomorphism by the same spectral sequence argument as in \cite{SheZhao19}. It is crucial to note that here, the coefficients are taken in the Novikov field $\Lambda_{\bK} = \bK[[q^{-1},q],$ where our base field is the field $\bK = \F_p[u^{-1},u]]$ of Laurent polynomials in a formal variable $u,$ and the tensor product of Floer complexes  $CF(H)^{\otimes p}$ is taken over $\Lambda_{\bK}.$

Disregarding grading, we define a new Novikov field \[\Lambda^{1/p}_{\bK} = \Lambda_{\bK}(q^{1/p}) = \bK[[t^{-1},t],\] into which $\Lambda_{\bK}$ embeds by sending $q$ to $t^p.$ Observe that $t$ has valuation $\rho/p.$ The theory of quasi-Frobenius maps as in \cite{SheZhao19} now yields an isomorphism \[F^{I}: HF(H;\Lambda^{1/p}_{\bK})^{I} \otimes \bK \langle{\theta}\rangle \to \wh{H}(\Z/(p), (CF(H)^{\otimes p})^{p\cdot I}),\] where $HF(H;\Lambda^{1/p}_{\bK})^{I}$ denotes the Floer homology of the Hamiltonian $H$ with coefficients in $\Lambda^{1/p}_{\bK}$ taken in the action window $I,$ and $\langle{\theta}\rangle$ is the exterior algebra over $\F_p$ on the formal variable $\theta$ of degree $1.$ Note that $\dim_{\F_p} \langle{\theta}\rangle = 2.$

Now rewrite $HF(H;\Lambda^{1/p}_{\bK})^{I}$ in terms of usual filtered Floer homology, considered as ungraded: \[HF(H;\Lambda^{1/p}_{\bK})^{I} \cong \displaystyle\bigoplus_{0 \leq j < p }  HF(H)^{I+j \rho/p}.\] Note that as $HF(H)^{I + \rho} \cong HF(H)^{I},$ we can replace the interval $0 \leq j < p,$ by any other interval with $p$ integers. For example we can pick the symmetric interval $-(p-1)/2 \leq j \leq (p-1)/2$ when $p$ is odd. Taking dimensions over $\bK,$ and using the inequality \[ \dim_{\bK} \wh{HF}_{eq}(H^{(p)})^{p\cdot I} \leq 2 \dim_{\bK}HF(H^{(p)})^{p\cdot I},\] whose proof follows that in \cite{SheZhao19} verbatim, we obtain after canceling a multiple of $2$ that \begin{equation}\label{eq: Smith Novikov} \sum_{j=0}^{p-1} \dim_{\bK} HF(H)^{I + j \rho/p} \leq \dim_{\bK}HF(H^{(p)})^{p\cdot I}.\end{equation} In fact, we could replace the right-hand side by the generally smaller quantity\footnote{This quantity can again be reduced. We refer to \cite{SheZhao19} for this argument.} $\dim_{\bK} (HF(H^{(p)})^{p\cdot I})^{\Z/(p)},$ the invariant subspace being considered with respect to a natural $\Z/(p)$-action.

We proceed by considering an interval $I_0$ of length $l = |I_0| = k\rho,$ with $k \in \Z_{>0}$ sufficiently large, so that $l > \beta(\phi,\bK),$ and $pl > \beta(\phi^p,\bK).$ Set $I_t = I_0 + t$ for the interval $I_0$ shifted by $t \in [0,l).$ We shall consider Equation \eqref{eq: Smith Novikov} for $I = I_t$ and integrate it over $t$ in $[0,l),$ keeping in mind that for an interval $J,$ $\dim HF(H)^{J}$ is given as the number of bars $(a,b]$ in the barcode of $H$ such that $\# (J \cap \{a,b\}) = 1.$ Note that any admissible interval of length $l = k \rho$ contains precisely $k \cdot B(\bK)$ infinite bars. Furthermore, each finite bar of length $\beta$ contributes $2 k \beta$ to the integral. Therefore we obtain that \[ p\cdot l(k B(\bK)) + p\cdot 2k\beta_{\tot}(\phi) \leq \frac{1}{p} \Big((pl) (pk B(\bK)) + 2pk \beta_{\tot}(\phi^p)\Big),\] since $p\cdot I$ is of length $pl = pk \rho.$ This simplifies to \[p\cdot \beta_{\tot}(\phi,\bK) \leq \beta_{\tot}(\phi^p,\bK),\] which yields the desired inequality, since by property \ref{barcodes, change of coefficients} of barcodes
of Hamiltonian diffeomorphisms $\beta_{\tot}(\phi,\bK) = \beta_{\tot}(\phi,\F_p),$ $\beta_{\tot}(\phi^p,\bK) = \beta_{\tot}(\phi^p,\F_p).$ 
\end{proof}

\bibliographystyle{plain}
\bibliography{biblio} 

\begin{thebibliography}{10}

\bibitem{periodicCPd}
Simon {Allais}.
\newblock {On periodic points of Hamiltonian diffeomorphisms of
  $\mathbb{C}\text{P}^d$ via generating functions}.
\newblock {\em arXiv e-prints}, page arXiv:2004.02165, April 2020.

\bibitem{AtShe20}
Marcelo~S. {Atallah} and Egor {Shelukhin}.
\newblock {Hamiltonian no-torsion}.
\newblock {\em arXiv e-prints}, page arXiv:2008.11758, August 2020.

\bibitem{Bar94}
Serguei~A. Barannikov.
\newblock The framed {M}orse complex and its invariants.
\newblock In {\em Singularities and bifurcations}, volume~21 of {\em Adv.
  Soviet Math.}, pages 93--115. Amer. Math. Soc., Providence, RI, 1994.

\bibitem{bar87}
Donald~W. Barnes.
\newblock The simplicial bundle of a {CW} fibration.
\newblock {\em Proc. Amer. Math. Soc.}, 101(3):559--562, 1987.

\bibitem{BL14}
Ulrich Bauer and Michael Lesnick.
\newblock Induced matchings of barcodes and the algebraic stability of
  persistence.
\newblock In {\em Computational geometry ({S}o{CG}'14)}, pages 355--364. ACM,
  New York, 2014.

\bibitem{Bor60}
Armand Borel.
\newblock {\em Seminar on transformation groups}.
\newblock With contributions by G. Bredon, E. E. Floyd, D. Montgomery, R.
  Palais. Annals of Mathematics Studies, No. 46. Princeton University Press,
  Princeton, N.J., 1960.

\bibitem{Cha84}
Marc Chaperon.
\newblock Une id\'{e}e du type ``g\'{e}od\'{e}siques bris\'{e}es'' pour les
  syst\`emes hamiltoniens.
\newblock {\em C. R. Acad. Sci. Paris S\'{e}r. I Math.}, 298(13):293--296,
  1984.

\bibitem{CM12}
David Chataur and Luc Menichi.
\newblock String topology of classifying spaces.
\newblock {\em J. Reine Angew. Math.}, 669:1--45, 2012.

\bibitem{CG20}
Erman {\c C\. inel\. i} and Viktor~L. {Ginzburg}.
\newblock {On the iterated Hamiltonian Floer homology}.
\newblock page arXiv:1902.06369.
\newblock To appear in \textit{Communications in Contemporary Mathematics}.

\bibitem{CKRTZ}
Brian Collier, Ely Kerman, Benjamin~M. Reiniger, Bolor Turmunkh, and Andrew
  Zimmer.
\newblock A symplectic proof of a theorem of {F}ranks.
\newblock {\em Compos. Math.}, 148(6):1969--1984, 2012.

\bibitem{CZ86}
C.~Conley and E.~Zehnder.
\newblock A global fixed point theorem for symplectic maps and subharmonic
  solutions of {H}amiltonian equations on tori.
\newblock In {\em Nonlinear functional analysis and its applications, {P}art 1
  ({B}erkeley, {C}alif., 1983)}, volume~45 of {\em Proc. Sympos. Pure Math.},
  pages 283--299. Amer. Math. Soc., Providence, RI, 1986.

\bibitem{For85}
Barry Fortune.
\newblock A symplectic fixed point theorem for {${\bf C}{\rm P}^n$}.
\newblock {\em Invent. Math.}, 81(1):29--46, 1985.

\bibitem{Fra92}
John Franks.
\newblock Geodesics on {$S^2$} and periodic points of annulus homeomorphisms.
\newblock {\em Invent. Math.}, 108(2):403--418, 1992.

\bibitem{Fra96}
John Franks.
\newblock Area preserving homeomorphisms of open surfaces of genus zero.
\newblock {\em New York J. Math.}, 2:1--19, electronic, 1996.

\bibitem{FH03}
John Franks and Michael Handel.
\newblock Periodic points of {H}amiltonian surface diffeomorphisms.
\newblock {\em Geom. Topol.}, 7:713--756, 2003.

\bibitem{Gin10}
Viktor~L. Ginzburg.
\newblock The {C}onley conjecture.
\newblock {\em Ann. of Math. (2)}, 172(2):1127--1180, 2010.

\bibitem{GG10}
Viktor~L. Ginzburg and Ba\c{s}ak~Z. G\"{u}rel.
\newblock Local {F}loer homology and the action gap.
\newblock {\em J. Symplectic Geom.}, 8(3):323--357, 2010.

\bibitem{Giv90}
Alexander~B. Givental\cprime.
\newblock Nonlinear generalization of the {M}aslov index.
\newblock In {\em Theory of singularities and its applications}, volume~1 of
  {\em Adv. Soviet Math.}, pages 71--103. Amer. Math. Soc., Providence, RI,
  1990.

\bibitem{GKPS}
Gustavo {Granja}, Yael {Karshon}, Milena {Pabiniak}, and Sheila {Sandon}.
\newblock {Givental's non-linear Maslov index on lens spaces}.
\newblock {\em arXiv e-prints}, page arXiv:1704.05827, April 2017.

\bibitem{GM69b}
Detlef Gromoll and Wolfgang Meyer.
\newblock On differentiable functions with isolated critical points.
\newblock {\em Topology}, 8:361--369, 1969.

\bibitem{Hatcher}
Allen Hatcher.
\newblock {\em Algebraic topology}.
\newblock Cambridge University Press, Cambridge, 2002.

\bibitem{Hin09}
Nancy Hingston.
\newblock Subharmonic solutions of {H}amiltonian equations on tori.
\newblock {\em Ann. of Math. (2)}, 170(2):529--560, 2009.

\bibitem{HZ94}
Helmut {Hofer} and Eduard {Zehnder}.
\newblock {\em {Symplectic invariants and Hamiltonian dynamics. Reprint of the
  1994 original.}}
\newblock Basel: Birkh\"auser, reprint of the 1994 original edition, 2011.

\bibitem{LeC06}
Patrice Le~Calvez.
\newblock Periodic orbits of {H}amiltonian homeomorphisms of surfaces.
\newblock {\em Duke Math. J.}, 133(1):125--184, 2006.

\bibitem{PSS96}
S.~Piunikhin, D.~Salamon, and M.~Schwarz.
\newblock Symplectic {F}loer-{D}onaldson theory and quantum cohomology.
\newblock In {\em Contact and symplectic geometry ({C}ambridge, 1994)},
  volume~8 of {\em Publ. Newton Inst.}, pages 171--200. Cambridge Univ. Press,
  Cambridge, 1996.

\bibitem{PS16}
Leonid Polterovich and Egor Shelukhin.
\newblock Autonomous {H}amiltonian flows, {H}ofer's geometry and persistence
  modules.
\newblock {\em Selecta Math. (N.S.)}, 22(1):227--296, 2016.

\bibitem{SZ92}
Dietmar Salamon and Eduard Zehnder.
\newblock Morse theory for periodic solutions of {H}amiltonian systems and the
  {M}aslov index.
\newblock {\em Comm. Pure Appl. Math.}, 45(10):1303--1360, 1992.

\bibitem{She19}
Egor {Shelukhin}.
\newblock {On the Hofer-Zehnder conjecture}.
\newblock {\em arXiv e-prints}, page arXiv:1905.04769, May 2019.

\bibitem{SheZhao19}
Egor {Shelukhin} and Jingyu {Zhao}.
\newblock {The $\mathbb{Z}/p \mathbb{Z}$-equivariant product-isomorphism in
  fixed point Floer cohomology}.
\newblock {\em arXiv e-prints}, page arXiv:1905.03666, May 2019.

\bibitem{The98}
David Th\'{e}ret.
\newblock Rotation numbers of {H}amiltonian isotopies in complex projective
  spaces.
\newblock {\em Duke Math. J.}, 94(1):13--27, 1998.

\bibitem{Tray}
Lisa Traynor.
\newblock Symplectic homology via generating functions.
\newblock {\em Geom. Funct. Anal.}, 4(6):718--748, 1994.

\bibitem{Ush11}
Michael Usher.
\newblock Boundary depth in {F}loer theory and its applications to
  {H}amiltonian dynamics and coisotropic submanifolds.
\newblock {\em Israel J. Math.}, 184:1--57, 2011.

\bibitem{Vit}
Claude Viterbo.
\newblock Symplectic topology as the geometry of generating functions.
\newblock {\em Math. Ann.}, 292(4):685--710, 1992.

\bibitem{Vit96}
Claude {Viterbo}.
\newblock {Functors and Computations in Floer homology with Applications Part
  II}.
\newblock {\em arXiv e-prints}, page arXiv:1805.01316, May 2018.

\bibitem{Whi56}
George~W. Whitehead.
\newblock Homotopy groups of joins and unions.
\newblock {\em Trans. Amer. Math. Soc.}, 83:55--69, 1956.

\bibitem{CZ05}
Afra Zomorodian and Gunnar Carlsson.
\newblock Computing persistent homology.
\newblock {\em Discrete Comput. Geom.}, 33(2):249--274, 2005.

\end{thebibliography}

\end{document}